\documentclass[12pt]{article}
\usepackage{amsmath,amsfonts,amssymb}
\usepackage{amsmath,amssymb}
\usepackage{amsthm}
\usepackage{graphicx}
\usepackage{eepic}
\usepackage{epic}

\usepackage[sort, numbers]{natbib}
\setcitestyle{square}       
\usepackage{bbold}
\usepackage{bbding}         
\usepackage{bm}             
\usepackage{graphicx}       
\usepackage{fancyvrb}       
\usepackage{indentfirst}    
\usepackage[nottoc]{tocbibind} 
\usepackage{icomma}         
\usepackage{dcolumn}        
\usepackage{booktabs}       
\usepackage{paralist}       
\usepackage[usenames]{xcolor}  

\newtheorem{theorem}{Theorem}[section]

\newtheorem{lemma}[theorem]{Lemma}

\theoremstyle{definition}
\newtheorem{definition}[theorem]{Definition}
\theoremstyle{remark}
\newtheorem{remark}[theorem]{Remark}

 \newcommand{\mN}{\mathbb N}
\newcommand{\be}{\begin{eqnarray}}
\newcommand{\ee}{\end{eqnarray}}
\newcommand{\bd}{\begin{definition}}
\newcommand{\ed}{\end{definition}}
\newcommand{\br}{\begin{remark}}
\newcommand{\er}{\end{remark}}

\newcommand{\bl}{\begin{lemma}}
\newcommand{\el}{\end{lemma}}

\newcommand{\bp}{\begin{picture}}
\newcommand{\ep}{\end{picture}}
\newcommand{\bi}{\begin{itemize}}
\newcommand{\ei}{\end{itemize}}
\newcommand{\bq}{\begin{quotation}}
\newcommand{\eq}{\end{quotation}}

\usepackage[dvipdfm]{hyperref}%

\begin{document}

\title{On the homotopy transfer of $A_\infty$ structures}
\author{Jakub Kop\v{r}iva}
\date{}
\maketitle

\centerline{\it{Dedicated to the memory of Martin Doubek}}

\abstract 

The present article is devoted to the study of transfers for $A_\infty$ structures, 
their maps and homotopies, as developed in \cite{Markl06}. In particular, we supply
the proofs of claims formulated therein and provide their extension by comparing them 
with the former approach based on the homological perturbation lemma. 
  
  

\endabstract

\tableofcontents  

\section{Introduction}

The notion of strongly homotopy associative  or $A_{\infty }$ algebras 
is a generalization of the concept of differential graded algebras.
These algebras were introduced by J. Stasheff with the aim of a characterization 
of (de)looping and bar construction in the category of topological spaces. 
Since then they found many applications ranging from algebraic topology and 
operads to quantum theories in theoretical physics.

We consider the following situation: 
let $(V, \partial_V)$ and $(W, \partial_W)$ be two chain complexes of modules, 
and $f: (V, \partial_V) \to (W, \partial_W)$ and $g: (W, \partial_W) \to (V, \partial_V)$ 
two mappings of chain complexes such that $gf$ is homotopic to the identity map on $V$
and $(V, \partial_V)$ is equipped with $A_\infty$ algebra structure.
Then a natural question arises $-$ can $A_\infty$ structure be transferred to 
$(W, \partial_W)$ and secondly, what is its explicit form in terms of $A_\infty$ 
algebra structure on $(V, \partial_V)$ and and in which sense it is unique? 

While the existence of a transfer follows from general model 
structure considerations, an unconditional and elaborate answer 
producing explicit formulas for the transferred objects was formulated 
in \cite{Markl06}. 
The present article contributes to the problem of transfer of $A_\infty$ 
structures. Its modest aim is to supply
detailed proofs of many claims omitted in the original article \cite{Markl06},
thereby facilitating complete subtle proofs to a reader interested 
in this topic. This exposition also extends the results of the aforementioned article
in several ways, and sheds a light on its relationship with the
homological perturbation lemma. 
 
The content of our article goes as follows. In the Section \ref{sec:2} we 
recall a well-known correspondence 
between $A_\infty$ algebras and codifferentials on reduced tensor coalgebras. 
This allows us to simplify the proofs in Section \ref{sec:3} considerably. 
The Section \ref{sec:3} is devoted to the problem of homotopy transfer  
of $A_\infty$ algebras. We first derive the formulas introduced in \cite{Markl06}, and 
then give their self-contained proofs. Here we achieve a substantial simplification of 
all proofs due to the reduction of sign factors. 
We also comment on another remark in \cite{Markl06}, namely, the
relationship between the homological perturbation lemma and homotopy transfer of
$A_\infty$ algebras. We prove that on certain assumptions
the explicit formulas in \cite{Markl06} do coincide with those coming 
from the homological perturbation lemma.

We shall work in the category of $\mathbb{Z}$-graded modules over an arbitrary 
commutative unital ring $R$, and their graded $R$-homomorphisms.

We first briefly recall the concepts of $A_\infty$ algebra,
$A_\infty$ morphism of $A_\infty$ algebras and $A_\infty$
homotopy of $A_\infty$ morphisms, cf. \cite{Markl06}, \cite{Keller01}.
\begin{definition}
Let $(V, \partial_V)$ be a chain complex of modules  
indexed by $\mathbb{Z}$, i.e. $(V, \partial_V)$ is a $\mathbb{Z}$-graded modules 
$V = \bigoplus_{i = -\infty}^{\infty}V_i$ with $\partial_V(V_i) \subset V_{i-1}$ and 
$\partial_V \circ \partial_V = 0$. Let $\mu_n: V^{\otimes n} \to V$ be a collection of
linear mappings of degree $n-2$ ($n \geq 2$), satisfying 
\begin{align}\label{ainftyrelations}
\partial_V & \mu_n - \sum_{i = 1}^{n} (-1)^{n} \mu_n\left(\mathbb{1}_V^{\otimes  i-1} 
\otimes \partial_V \otimes \mathbb{1}_V^{\otimes  n-i}\right) \nonumber \\
& = \sum_{A(n)} (-1)^{i(\ell+1) + n} \mu_k \left(\mathbb{1}_V^{\otimes  i-1} \otimes 
\mu_\ell  \otimes \mathbb{1}_V^{\otimes  k - i}\right)
\end{align}
for all $n \geq 2$ and $A(n) = \{k, \ell\in\mN |\, k+\ell = n+1, k,\ell \geq 2, 1 \leq i \leq k\}$. 
The structure $(V, \partial_V, \mu_2, \mu_3, \dots)$ is called $A_\infty$ algebra.
\end{definition}
Throughout the article, we use the Koszul sign convention. This means that 
for $U$, $V$ a $W$ graded modules and  
$f: U \to V$, $g: U \to V$, $h: V \to W$ and $i: V \to W$ linear maps of degrees 
$|f|$, $|g|$, $|h|$ and $|i|$, respectively, holds
$$(h \otimes i)(f \otimes g) = (-1)^{|f||i|} hf \otimes ig.$$ Similarly
for $u_1, u_2 \in U$ of degree $|u_1|$ and $|u_2|$, respectively, holds
$$(f \otimes g)(u_1 \otimes u_2) = (-1)^{|u_1||g|}f(u_1) \otimes g(u_2).$$

\begin{definition} 
Let $(V, \partial_V, \mu_2, \dots)$ and $(W, \partial_W, \nu_2, \dots)$ be $A_\infty$ algebras. 
Then the set $\{f_n: V^{\otimes n} \to W, |f_n| = n-1\}_{n \geq 1}$ is called $A_\infty$ morphism
if 
\begin{align}
& \partial_Wf_n + \sum_{B(n)}(-1)^{\vartheta(r_1, \dots, r_k)} \nu_k(f_{r_1} \otimes \dots \otimes f_{r_k}) 
\nonumber \\
& = f_1\mu_n - \sum_{i = 1}^{n} (-1)^{n} f_n\left(\mathbb{1}_V^{\otimes  i-1} \otimes \partial_V \otimes \mathbb{1}_V^{\otimes  n-i}\right) 
\nonumber \\
& -\sum_{A(n)} (-1)^{i(\ell+1) + n} f_k \left(\mathbb{1}_V^{\otimes  i-1} \otimes \mu_\ell  \otimes \mathbb{1}_V^{\otimes  k - i}\right)
\end{align}
holds for all $n \geq 1$ with $B(n) = \{k, r_1, \mbox{\dots}, r_k\in\mN |\, k \geq 2,  r_1, \mbox{\dots}, r_k \geq 1,  
r_1 + \mbox{\dots} + r_k = n\}$ and $\vartheta(r_1, \dots, r_k) = \sum_{1 \leq i < j \leq k} r_i(r_j + 1)$.
\end{definition}
Morphisms of $A_\infty$ algebras can be composed: for $(U, \partial_U, \varrho_2, \dots)$, 
$(V, \partial_V, \mu_2, \dots)$ and $(W, \partial_W, \nu_2, \dots)$ $A_\infty$ algebras,
$\{f_n: U^{\otimes n} \to V\}_{n \geq 1}$ and $\{g_n: V^{\otimes n} \to W\}_{n \geq 1}$
$A_\infty$ morphisms, their composition   
$\{(gf)_n: U^{\otimes n} \to W\}_{n \geq 1}$ is defined as
\begin{align}\label{composition}
(gf)_n = g_1f_n + \sum_{B(n)}(-1)^{\vartheta(r_1, \dots, r_k)} g_k(f_{r_1} \otimes \dots \otimes f_{r_k}).
\end{align}

\begin{definition} \label{ainftyhomotopy}
Let $\{f_n: V^{\otimes n} \to W\}_{n \geq 1}$ and $\{g_n: V^{\otimes n} \to W\}_{n \geq 1}$ 
be morphisms between $A_\infty$ algebras $(V, \partial_V, \mu_2, \dots)$ and $(W, \partial_W, \nu_2, \dots)$. The 
set of linear mappings $\{h_n: V^{\otimes n} \to W, |h_n| = n\}_{n \geq 1}$ is an $A_\infty$ homotopy between 
$A_\infty$ morphisms $\{f_n: V^{\otimes n} \to W\}_{n \geq 1}$ and $\{g_n: V^{\otimes n} \to W\}_{n \geq 1}$
provided 
\begin{align}
& f_n - g_n = h_1\mu_n - \sum_{i = 1}^{n} (-1)^{n} h_n\left(\mathbb{1}_V^{\otimes  i-1} \otimes \partial_V \otimes \mathbb{1}_V^{\otimes  n-i}\right)
\nonumber \\
& -\sum_{A(n)} (-1)^{i(\ell+1) + n} h_k \left(\mathbb{1}_V^{\otimes  i-1} \otimes \mu_\ell  \otimes \mathbb{1}_V^{\otimes  k - i}\right) + \delta_W h_n 
\nonumber \\
& + \sum_{B(n)}\sum_{1 \leq i \leq k}(-1)^{\vartheta(r_1, \dots, r_k)} \nu_k(f_{r_1} \otimes \dots \otimes f_{r_{i-1}} \otimes h_{r_i} \otimes g_{r_{i+1}} \otimes \dots \otimes g_{r_k}),
\end{align}
is true for all $n \geq 1$ with 
$B(n) = \{k, r_1, \mbox{\dots}, r_k\in\mN |\, k \geq 2,  r_1, \mbox{\dots}, r_k \geq 1,  r_1 + \mbox{\dots} + r_k = n\}$.
\end{definition}

\section{Reduced tensor coalgebras}
\label{sec:2}

In the present section we introduce a bijective correspondence between 
$A_\infty$ algebras and codifferentials on reduced tensor coalgebras, cf.
\cite{Keller01}. We retain the notation 
$V = \bigoplus_{i = -\infty}^{\infty} V_i$ for $\mathbb{Z}-$ graded modules as well as
\begin{displaymath}
\tag{A}
\label{A}
A(n) = \{k, \ell\in\mN |\, k+\ell = n+1, k,\ell \geq 2, 1 \leq i \leq k\},
\end{displaymath}
\begin{displaymath}
\tag{B}
\label{B}
B(n) = \{k, r_1, \mbox{\dots}, r_k\in\mN |\, k \geq 2,  r_1, \mbox{\dots}, r_k \geq 1,  r_1 + \mbox{\dots} + r_k = n\}
\end{displaymath}
for $n \in \mathbb{N}$, and $A(1) = A(2) = B(1) = \emptyset.$ We use a few natural variations on this notation, 
e.g. $A'(n) = \{k', \ell'\in\mN |\, k'+\ell' = n+1, k',\ell' \geq 2, 1 \leq i' \leq k'\}$. 

\subsection{Codiferentials on tensor coalgebras}

\begin{definition} \label{k1d1}
Let $\overline{T}V = \bigoplus_{n=1}^\infty V^{\otimes n}$, where the elements in $V^{\otimes i}$ 
have degree (or homogeneity) $i$, and let the mapping $C: \overline{T}V \to \overline{T}V \otimes \overline{T}V$ 
be defined in such a way that $C: v \mapsto 0$ for $v \in V^{\otimes 1} = V$ and 
\begin{align}
C: v_1 \otimes \mbox{\dots} \otimes v_n \mapsto \sum_{i=1}^{n-1} (v_1 \otimes \mbox{\dots} \otimes v_i) \otimes (v_{i+1} \otimes \mbox{\dots} \otimes v_n),
\end{align} 
for $n \geq 2$ and $v_1, \mbox{\dots}, v_n \in V$. The pair $(\overline{T}V, C)$ 
is called the reduced tensor coalgebra. 
\end{definition}

\begin{definition}\label{k1d2}
A linear mapping $\delta: \overline{T}V \to \overline{T}V$ of degree $-1$ is called coderivation 
if $C \circ \delta = (\delta \otimes \mathbb{1} + \mathbb{1} \otimes \delta) \circ C$. Moreover,
if $\delta$ satisfies $\delta \circ \delta = 0$, it is called codifferential. 
\end{definition}

\begin{remark}
We notice that $C$ is coassociative, $(\mathbb{1} \otimes C) \circ C = (C \otimes \mathbb{1}) \circ C$.
For all $v \in \overline{T}V$ holds $C(v) = 0$ if and only if $v$ is of homogeneity $1$.
For all maps $\varphi: V^{\otimes n} \to \overline{T}W$, $n \geq 1$, holds $C_{\overline{T}W} \circ \varphi = 0$ 
if and only if $\varphi\left(V^{\otimes n}\right) \subseteq W$.
For all $v = v_1 \otimes \mbox{\dots} \otimes v_n \in \overline{T}V$ and 
$w = w_1 \otimes \mbox{\dots} \otimes w_m \in \overline{T}V$, we have
$$C(v \otimes w) = \sum_{i = 1}^{n-1} \left(v_{1,i}\right) \otimes \left(v_{i+1, n} \otimes w\right) + \left(v\right)\otimes \left(w\right) + \sum_{i = 1}^{m-1} \left(v \otimes w_{1,i}\right) \otimes \left(w_{i+1,m}\right),$$
with $v_{i,j} = v_i \otimes \mbox{\dots} \otimes v_j$, $i \leq j$, $i,j \in \{1, \dots, n\}$, and analogously
for $w_{i,j}$. This little calculation expresses a fact that $\overline{T}V$ is a bialgebra which is, as a conilpotent coalgebra, cogenerated by $V$.
\end{remark}
\begin{lemma}\label{k1l3}
Let $E: \overline{T}V \to \overline{T}W$ be a linear mapping for which there exist 
$\{e_n: V^{\otimes n} \to W\}_{n \geq 1}$ with
$E|_{V^{\otimes n}} = e_n + \sum_{B(n)} e_{r_1} \otimes \mbox{\dots} \otimes e_{r_{k}}$, 
and $B(n)$ given in \eqref{B}. Then 
\begin{align}
C_{\overline{T}W} \circ E|_{V^{\otimes n}} = \sum_{i=1}^{n-1} \left(E|_{V^{\otimes i}}\right) 
\otimes \left(E|_{V^{\otimes n-i}}\right).
\end{align}
\end{lemma}
\begin{proof}
Obviously, we can write 
$E|_{V^{\otimes n}} = e_n + \sum_{i=1}^{n-1} e_{i} \otimes E|_{V^{\otimes n-i}}$. 
The proof is by induction on $n$: the claim holds for $n=1$ and we assume it
is true foll all natural numbers less than $n$. Then
$$C_{\overline{T}W} \circ E|_{V^{\otimes n}} = C_{\overline{T}W} \circ \left( e_n + \sum_{i=1}^{n-1} e_{i} \otimes E|_{V^{\otimes n-i}}\right)$$
$$= C_{\overline{T}W} \circ \left(\sum_{i=1}^{n-1} e_{i} \otimes E|_{V^{\otimes n-i}}\right)$$ 
$$= \sum_{i=1}^{n-1} \left(e_i \right) \otimes \left(E|_{V^{\otimes n-i}}\right) + \sum_{i=1}^{n-1}\sum_{j=1}^{n-1-i} \left(e_i \otimes E|_{V^{\otimes j}}\right) \otimes \left(E|_{V^{\otimes n-i-j}}\right) $$ $$=\sum_{i=1}^{n-1} \left(e_i \right) \otimes \left(E|_{V^{\otimes n-i}}\right) + \sum_{\ell=2}^{n-1}\sum_{j=1}^{\ell-1} \left(e_j \otimes E|_{V^{\otimes \ell-j}}\right) \otimes \left(E|_{V^{\otimes n-\ell}}\right)$$ $$= \left(e_1 \right) \otimes \left(E|_{V^{\otimes n-1}}\right) + \sum_{\ell=2}^{n-1}\left(e_\ell + \sum_{j=1}^{\ell-1} e_j \otimes E|_{V^{\otimes \ell-j}}\right) \otimes \left(E|_{V^{\otimes n-\ell}}\right),$$ and the proof follows by induction hypothesis from 
$E|_{V^{\otimes \ell}} = e_\ell + \sum_{i=1}^{\ell-1} e_{i} \otimes E|_{V^{\otimes \ell-i}}.$
\end{proof}

\begin{theorem}\label{k1v4}
Let $E: \overline{T}V \to \overline{T}W$ a $G: \overline{T}V \to \overline{T}W$ be linear mappings for which
there exist linear mappings $\{e_n: V^{\otimes n} \to W\}_{n \geq 1}$, $\{g_n: V^{\otimes n} \to W\}_{n \geq 1}$ 
such that $E|_{V^{\otimes n}} = e_n + \sum_{B(n)} e_{r_1} \otimes \mbox{\dots} \otimes e_{r_{k}}$ and 
$G|_{V^{\otimes n}} = g_n + \sum_{B(n)} g_{r_1} \otimes \mbox{\dots} \otimes g_{r_{k}}$ with $B(n)$ given in 
\eqref{B}. Given a linear mapping $F: \overline{T}V \to \overline{T}W$, the following conditions are equivalent: 
\begin{enumerate}
\item $C_{\overline{T}W} \circ F = \left(E \otimes F + F \otimes G\right) \circ C_{\overline{T}V}$,
\item there exist linear mappings $\{f_n: V^{\otimes n} \to W\}_{n \geq 1}$ such that 

$$F|_{V^{\otimes n}} = f_n + \sum_{B(n)}\sum_{1 \leq i \leq k} e_{r_1} \otimes \mbox{\dots} \otimes e_{r_{i-1}} \otimes f_{r_i} \otimes g_{r_{i+1}} \otimes \mbox{\dots} \otimes g_{r_k}.$$
\end{enumerate}
\end{theorem}
\begin{proof}
$(2) \Rightarrow (1)$: We have 
$F|_{V^{\otimes n}} = f_n + \sum_{i=1}^{n-1} E|_{V^{\otimes i}} \otimes f_{n-i} +  \sum_{i=1}^{n-1} f_i \otimes G|_{V^{\otimes n-i}} + \sum_{i=1}^{n-1}\sum_{j = 1}^{n-i-1} E|_{V^{\otimes j}} \otimes f_{i} \otimes G|_{V^{\otimes n-i-j}}$
for all $n \geq 1$.
We now verify $1.$ by expanding both sides: 
$$\left(E \otimes F + F \otimes G\right) \circ C_{\overline{T}V}|_{V^{\otimes n}} = \left(E \otimes F + F \otimes G\right) \circ \sum_{i=1}^{n-1} \left(\mathbb{1}_V^{\otimes  n-i}\right) \otimes \left(\mathbb{1}_V^{\otimes  i}\right) $$
$$= \sum_{i=1}^{n-1} \left[\left(E|_{V^{\otimes n-i}}\right) \otimes \left(F|_{V^{\otimes i}}\right) + \left(F|_{V^{\otimes n-i}}\right) \otimes \left(G|_{V^{\otimes i}}\right)\right],$$
and by Lemma \ref{k1l3} we get 
$$C_{\overline{T}W} \circ \left(\sum_{i=1}^{n-1} E|_{V^{\otimes i}} \otimes f_{n-i}\right) $$
$$= \sum_{i=1}^{n-1}\left(E|_{V^{\otimes n-i}}\right) \otimes \left(f_{i}\right) + \sum_{i=1}^{n-1} \sum_{j=1}^{n-1-i}\left(E|_{V^{\otimes n-i-j}} \right) \otimes \left(E|_{V^{\otimes j}} \otimes f_{i}\right),$$
$$C_{\overline{T}W} \circ \left(\sum_{i=1}^{n-1} f_i \otimes G|_{V^{\otimes n-i}}\right) $$
$$=\sum_{i=1}^{n-1} \left(f_i\right) \otimes \left(G|_{V^{\otimes n-i}}\right) + \sum_{i=1}^{n-1} \sum_{j=1}^{n-1-i}\left(f_i \otimes G|_{V^{\otimes j}}\right) \otimes \left(G|_{V^{\otimes n-i-j}}\right),$$
$$C_{\overline{T}W} \circ \left(\sum_{i=1}^{n-1}\sum_{j = 1}^{n-i-1} E|_{V^{\otimes j}} \otimes f_{i} \otimes G|_{V^{\otimes n-i-j}}\right) $$
$$=\sum_{i=1}^{n-1}\sum_{j = 1}^{n-i-1} \left(E|_{V^{\otimes n-i-j}}\right) \otimes \left(f_{i} \otimes G|_{V^{\otimes j}}\right) + \sum_{i=1}^{n-1}\sum_{j = 1}^{n-i-1} \left(E|_{V^{\otimes j}} \otimes f_{i}\right) \otimes \left(G|_{V^{\otimes n-i-j}}\right) $$
$$+ \sum_{i=1}^{n-1}\sum_{j = 1}^{n-i-1}\sum_{k=1}^{j-1} \left(E|_{V^{\otimes n-i-j-k}}\right) \otimes \left(E|_{V^{\otimes j}} \otimes f_{i} \otimes G|_{V^{\otimes k}}\right) $$
$$+\sum_{i=1}^{n-1}\sum_{j = 1}^{n-i-1} \sum_{k=1}^{j-1} \left(E|_{V^{\otimes j}} \otimes f_{i} \otimes G|_{V^{\otimes k}}\right) \otimes \left(G|_{V^{\otimes n-i-j-k}}\right).$$ 
The summation in the variables $i+j$ and $i+j+k$, respectively, yields 
$$C_{\overline{T}W} \circ \left(\sum_{i=1}^{n-1} E|_{V^{\otimes i}} \otimes f_{n-i}\right) $$
$$= \sum_{i=1}^{n-1}\left(E|_{V^{\otimes n-i}}\right) \otimes \left(f_{i}\right) + \sum_{\ell=2}^{n-1} \sum_{j=1}^{\ell-1}\left(E|_{V^{\otimes n-\ell}} \right) \otimes \left(E|_{V^{\otimes \ell-j}} \otimes f_{j}\right),$$
$$C_{\overline{T}W} \circ \left(\sum_{i=1}^{n-1} f_i \otimes G|_{V^{\otimes n-i}}\right) $$
$$=\sum_{i=1}^{n-1} \left(f_i\right) \otimes \left(G|_{V^{\otimes n-i}}\right) + \sum_{\ell=2}^{n-1} \sum_{j=1}^{\ell-1}\left(f_j \otimes G|_{V^{\otimes \ell - j}}\right) \otimes \left(G|_{V^{\otimes n-\ell}}\right),$$
$$C_{\overline{T}W} \circ \left(\sum_{i=1}^{n-1}\sum_{j = 1}^{n-i-1} E|_{V^{\otimes j}} \otimes f_{i} \otimes G|_{V^{\otimes n-i-j}}\right) $$
$$=\sum_{\ell=2}^{n-1} \sum_{j=1}^{\ell-1} \left(E|_{V^{\otimes n-\ell}}\right) \otimes \left(f_{j} \otimes G|_{V^{\otimes \ell-j}}\right) + \sum_{\ell=2}^{n-1} \sum_{j=1}^{\ell-1} \left(E|_{V^{\otimes \ell-j}} \otimes f_{j}\right) \otimes \left(G|_{V^{\otimes n-\ell}}\right) $$
$$+ \sum_{\ell=3}^{n-1}\sum_{j = 1}^{m-1} \sum_{i=1}^{j-1}  \left(E|_{V^{\otimes n-\ell}}\right) \otimes \left(E|_{V^{\otimes \ell-j}} \otimes f_{i} \otimes G|_{V^{\otimes j-i}}\right) $$
$$+\sum_{\ell=3}^{n-1}\sum_{j = 1}^{m-1} \sum_{i=1}^{j-1} \left(E|_{V^{\otimes j-i}} \otimes f_{i} \otimes G|_{V^{\otimes \ell-j}}\right) \otimes \left(G|_{V^{\otimes n-\ell}}\right).$$
Taking all terms of the form 
$\left(E|_{V^{\otimes n-i}}\right) \otimes \star\,$ and $\,\star \otimes \left(G|_{V^{\otimes n-i}}\right)$ 
results in 
$$C_{\overline{T}W} \circ F|_{V^{\otimes n}} = \sum_{i=1}^{n-1} \left[\left(E|_{V^{\otimes n-i}}\right) \otimes \left(F|_{V^{\otimes i}}\right) + \left(F|_{V^{\otimes n-i}}\right) \otimes \left(G|_{V^{\otimes i}}\right)\right]$$
and the implication is proved.
Notice that we also proved, on the assumption 
$F|_{V^{\otimes m}} = f_n + \sum_{B(m)}\sum_{1 \leq i \leq k} e_{r_1} \otimes \mbox{\dots} \otimes e_{r_{i-1}} \otimes f_{r_i} \otimes g_{r_{i+1}} \otimes \mbox{\dots} \otimes g_{r_k}$ for $n > m \geq 1$, that 
$$C_{\overline{T}W} \circ \Bigg(\sum_{i=1}^{n-1} E|_{V^{\otimes i}} \otimes f_{n-i} +  \sum_{i=1}^{n-1} f_i \otimes G|_{V^{\otimes n-i}} + \sum_{i=1}^{n-1}\sum_{j = 1}^{n-i-1}\!E|_{V^{\otimes j}} \otimes f_{i} \otimes G|_{V^{\otimes n-i-j}}\Bigg)$$ 
$$= \sum_{i=1}^{n-1} \left[\left(E|_{V^{\otimes n-i}}\right) \otimes \left(F|_{V^{\otimes i}}\right) + \left(F|_{V^{\otimes n-i}}\right) \otimes \left(G|_{V^{\otimes i}}\right)\right].$$

\noindent $(1) \Rightarrow (2)$: The proof is again by induction. For all $v \in V$ holds 
$C_{\overline{T}W} \circ F(v) = 0$, which gives $F(V) \subset W$ and so there exists a 
linear mapping $f_1: V \to W$ such that $F|_V = f_1$.
Assume now the claim of the implication is true for all natural numbers less than
$n$, i.e. 
$F|_{V^{\otimes m}} = f_m + \sum_{B(m)}\sum_{1 \leq i \leq k} e_{r_1} \otimes \mbox{\dots} \otimes e_{r_{i-1}} \otimes f_{r_i} \otimes g_{r_{i+1}} \otimes \mbox{\dots} \otimes g_{r_k},$ for $n > m \geq 1$. The proof 
of the previous implication claims for 
$F|_{V^{\otimes m}} = f_m + \sum_{B(m), r_i > 0} e_{r_1} \otimes \mbox{\dots} \otimes e_{r_{i-1}} \otimes f_{r_i} \otimes g_{r_{i+1}} \otimes \mbox{\dots} \otimes g_{r_k}$ with  $n > m \geq 1$, that 
$$C_{\overline{T}W} \circ F|_{V^{\otimes n}}=\sum_{i=1}^{n-1} \left[\left(E|_{V^{\otimes n-i}}\right) \otimes \left(F|_{V^{\otimes i}}\right) + \left(F|_{V^{\otimes n-i}}\right) \otimes \left(G|_{V^{\otimes i}}\right)\right]$$ 
$$=C_{\overline{T}W} \circ \Bigg(\sum_{i=1}^{n-1} E|_{V^{\otimes i}} \otimes f_{n-i} + \sum_{i=1}^{n-1} f_i \otimes G|_{V^{\otimes n-i}} + \sum_{i=1}^{n-1}\sum_{j = 1}^{n-i-1}\!E|_{V^{\otimes j}} \otimes f_{i} \otimes G|_{V^{\otimes n-i-j}}\Bigg).$$
Because $C_{\overline{T}W}$ is linear, 
$F|_{V^{\otimes n}}$ differs from 
$\sum_{i=1}^{n-1} E|_{V^{\otimes i}} \otimes f_{n-i} +  \sum_{i=1}^{n-1} f_i \otimes G|_{V^{\otimes n-i}} + \sum_{i=1}^{n-1}\sum_{j = 1}^{n-i-1}\!E|_{V^{\otimes j}} \otimes f_{i} \otimes G|_{V^{\otimes n-i-j}}$ by a linear map
$f_n: V^{\otimes n} \to W$. This means $F|_{V^{\otimes n}}$ is of the required form and the proof is complete.
\end{proof}

\begin{theorem}\label{k1v5}
A linear mapping $\delta: \overline{T}V \to \overline{T}V$ of degree $-1$ fulfills 
$C \circ \delta = (\delta \otimes {\mathbb 1}_V + \mathbb{1}_V \otimes \delta) \circ C$ if and only if 
there exist a set of maps $\{\delta_n: V^{\otimes n} \to V\}_{n \geq 1}$ of degree $-1$ such that
$\delta|_V = \delta_1$ and for 
$n\geq 2$ holds $\delta|_{V^{\otimes n}} = \delta_n + \sum_{i = 1}^{n} \mathbb{1}_V^{\otimes  i-1} \otimes \delta_1 \otimes \mathbb{1}_V^{\otimes  n-i} + \sum_{A(n)} \mathbb{1}_V^{\otimes  i-1} \otimes \delta_\ell \otimes \mathbb{1}_V^{\otimes  k - i},$ where $A(n)$ is given by \eqref{A}.
\end{theorem}
\begin{proof}
In Theorem ~\ref{k1v4} we take $E=G=\mathbb{1}_V$, where 
$e_1 = g_1 = \mathbb{1}_{V}$ and $e_n = g_n = 0$ for $n \geq 2$.
\end{proof}
\begin{lemma}\label{k1t6}
Let $\delta: \overline{T}V \to \overline{T}V$ be a linear map of degree $-1$ such that 
$\delta|_V = \delta_1$ and for $n\geq 2$ holds 
$\delta|_{V^{\otimes n}} = \delta_n + \sum_{i = 1}^{n} \mathbb{1}_V^{\otimes  i-1} \otimes \delta_1 \otimes \mathbb{1}_V^{\otimes  n-i} + \sum_{A(n)} \mathbb{1}_V^{\otimes  i-1} \otimes \delta_\ell \otimes \mathbb{1}_V^{\otimes  k - i}$.
Then the following conditions are equivalent:
\begin{enumerate} 
\item $\delta \circ \delta = 0$,
\item $\delta_1 \circ \delta_1 = 0$ and for all $n \geq 2$ we have
\begin{equation}
\label{T6}
\begin{split}
\delta_1(\delta_n) + \sum_{i = 1}^{n} \delta_n \left(\mathbb{1}_V^{\otimes  i-1} \otimes \delta_1 \otimes \mathbb{1}_V^{\otimes  n-i}\right) \\
+ \sum_{A(n)} \delta_k\left(\mathbb{1}_V^{\otimes  i-1} \otimes \delta_\ell \otimes \mathbb{1}_V^{\otimes  k - i}\right) = 0,
\end{split}	
\end{equation}
\end{enumerate}
where $A(n)$ is given by \eqref{A}.
\end{lemma}
\begin{proof}
$(1) \Rightarrow (2)$: The proof goes by induction. By assumption we have for $v \in V$ 
$\delta(\delta_1(v)) = 0,$ so $\delta_1: V \to V$ implies $\delta_1(\delta_1(v)) = 0.$
Now assume \eqref{T6} is true for all natural numbers less than $n$. Then
$$\delta^2|_{V^{\otimes n}} = \delta_1(\delta_n) + \sum_{i = 1}^{n} \delta|_{V^{\otimes n}} \left(\mathbb{1}_V^{\otimes  i-1} \otimes \delta_1 \otimes \mathbb{1}_V^{\otimes  n-i}\right) $$$$+ \sum_{A(n)} \delta|_{V^{\otimes k}}\left(\mathbb{1}_V^{\otimes  i-1} \otimes \delta_\ell \otimes \mathbb{1}_V^{\otimes  k - i}\right).$$
Schematically, this means
$$\delta^2|_{V^{\otimes n}} = \delta_1(\delta_n) + \sum_{i = 1}^{n} \delta_n \left(\mathbb{1}_V^{\otimes  i-1} \otimes \delta_1 \otimes \mathbb{1}_V^{\otimes  n-i}\right) $$$$+ \sum_{A(n)} \delta_k\left(\mathbb{1}_V^{\otimes  i-1} \otimes \delta_\ell \otimes \mathbb{1}_V^{\otimes  k - i}\right) + \sum \mathbb{1}_V^{\otimes  a} \otimes \delta_{b+d+1}\left(\mathbb{1}_V^{\otimes  b} \otimes \delta_c \otimes \mathbb{1}_V^{\otimes  d}\right) \otimes \mathbb{1}_V^{\otimes  e} $$
$$+ \sum \mathbb{1}_V^{\otimes  a} \otimes \delta_b \otimes \mathbb{1}_V^{\otimes  c} \otimes \delta_d \otimes \mathbb{1}_V^{\otimes  e} -  \sum \mathbb{1}_V^{\otimes  a} \otimes \delta_b \otimes \mathbb{1}_V^{\otimes  c} \otimes \delta_d \otimes \mathbb{1}_V^{\otimes  e},$$
where the last row is a consequence of the Koszul sign convention:
$$\left(\mathbb{1}_V^{\otimes  a} \otimes \delta_b \otimes \mathbb{1}_V^{\otimes  c + 1 + e}\right)\left(\mathbb{1}_V^{\otimes  a + b + c} \otimes \delta_{d} \otimes \mathbb{1}_V^{\otimes  e}\right) = \mathbb{1}_V^{\otimes  a} \otimes \delta_b \otimes \mathbb{1}_V^{\otimes  c} \otimes \delta_d \otimes \mathbb{1}_V^{\otimes  e},$$
$$\left(\mathbb{1}_V^{\otimes  a + 1 + c} \otimes \delta_{d} \otimes \mathbb{1}_V^{\otimes  e}\right)\left(\mathbb{1}_V^{\otimes  a} \otimes \delta_{b} \otimes \mathbb{1}_V^{\otimes  c+ d + e}\right) = (-1)^{\left|\delta_b \right|\left|\delta_d\right|}\mathbb{1}_V^{\otimes  a} \otimes \delta_b \otimes \mathbb{1}_V^{\otimes  c} \otimes \delta_d \otimes \mathbb{1}_V^{\otimes  e}$$
with $\left|\delta_n\right| = -1$ for all $n \in \mathbb{N}$.
The term $\sum \mathbb{1}_V^{\otimes  a} \otimes \delta_{b+d+1}\left(\mathbb{1}_V^{\otimes  b} \otimes \delta_c \otimes \mathbb{1}_V^{\otimes  d}\right) \otimes \mathbb{1}_V^{\otimes  e}$ can be written as 
$$\mathbb{1}_V^{\otimes  a} \otimes \delta_{b+d+1}\left(\mathbb{1}_V^{\otimes  b} \otimes \delta_c \otimes \mathbb{1}_V^{\otimes  d}\right) \otimes \mathbb{1}_V^{\otimes  e} = \left(\mathbb{1}_V^{\otimes  a} \otimes \delta_{b+d+1} \otimes \mathbb{1}_V^{\otimes  e} \right)\left(\mathbb{1}_V^{\otimes  a+b} \otimes \delta_c \otimes \mathbb{1}_V^{\otimes  d+e}\right).$$
We have $a+b+c+d+e = n$, choose arbitrary $a,e \geq 0$, $1 \leq a + e < n$ and sum over all $b,c,d$ such that 
$0 \leq b,d \leq n-a-e$ and $1 \leq c \leq n-a-e$ such that $b+c+d = n -e - a$:
$$\sum_{b,c,d} \delta_{b+d+1}\left(\mathbb{1}_V^{\otimes  b} \otimes \delta_c \otimes \mathbb{1}_V^{\otimes  d}\right)
=\delta_1(\delta_{n'}) + \sum_{i = 1}^{n'} \delta_{n'} \left(\mathbb{1}_V^{\otimes  i-1} \otimes \delta_1 \otimes \mathbb{1}_V^{\otimes  n'-i}\right) $$$$+ \sum_{A(n')} \delta_k\left(\mathbb{1}_V^{\otimes  i-1} \otimes \delta_\ell \otimes \mathbb{1}_V^{\otimes  k - i}\right),$$
where $n'= n - a - e$. By induction hypothesis, the last display is equal to $0$, and we have
 $$\sum \mathbb{1}_V^{\otimes  a} \otimes \delta_{b+d+1}\left(\mathbb{1}_V^{\otimes  b} \otimes \delta_c \otimes \mathbb{1}_V^{\otimes  d}\right) \otimes \mathbb{1}_V^{\otimes  e} $$
 $$= \sum_{a,e} \mathbb{1}_V^{\otimes  a} \otimes\left(\sum_{b,c,d} \delta_{b+d+1}\left(\mathbb{1}_V^{\otimes  b} \otimes \delta_c \otimes \mathbb{1}_V^{\otimes  d}\right)\right) \otimes \mathbb{1}_V^{\otimes  e} = \sum_{a,e} \mathbb{1}_V^{\otimes  a} \otimes\, 0\, \otimes \mathbb{1}_V^{\otimes  e} = 0.$$
 Consequently, \eqref{T6} is true for $n$ and
 $$\delta_1(\delta_n) + \sum_{i = 1}^{n} \delta_n \left(\mathbb{1}_V^{\otimes  i-1} \otimes \delta_1 \otimes \mathbb{1}_V^{\otimes  n-i}\right) $$$$+ \sum_{A(n)} \delta_k\left(\mathbb{1}_V^{\otimes  i-1} \otimes \delta_\ell \otimes \mathbb{1}_V^{\otimes  k - i}\right) = \delta^2|_{V^{\otimes n}} = 0.$$
 
\noindent $(2) \Rightarrow (1)$: The second implication can be easily deduced from the first one.
\end{proof}

\subsection{Morphisms and homotopies}

\begin{definition}\label{k1d7}
Let $\delta^V$ be a codifferential on $(\overline{T}V, C)$ and $\delta^W$ be a codifferential on
$(\overline{T}W, C)$. A linear mapping $F: \left(\overline{T}V, C, \delta^V\right) \to \left(\overline{T}W, C, \delta^W\right)$ of degree $0$ is called morphism provided 
$C_{\overline{T}W} \circ F = \left(F \otimes F\right) \circ C_{\overline{T}V}$ and $\delta^W \circ F = F \circ \delta^V.$
\end{definition}

\begin{lemma}\label{k1t8}
Let $F: \left(\overline{T}V, \delta^V\right) \to \left(\overline{T}W, \delta^W\right)$ be a linear 
map of degree $0$. Then the following claims are equivalent:
\begin{enumerate} 
\item $C_{\overline{T}W} \circ F = \left(F \otimes F\right) \circ C_{\overline{T}V}$,
\item there is a set of linear mappings $\{f_n: V^{\otimes n} \to W\}_{n \geq 1}$ 
of degree $0$ such that $F|_{V^{\otimes n}} = f_n + \sum_{B(n)} f_{r_1} \otimes \mbox{\dots} \otimes f_{r_k}$,
with $B(n)$ given in \eqref{B}.
\end{enumerate}
\end{lemma}
\begin{proof} $(2) \Rightarrow (1)$: A consequence of Lemma~\ref{k1l3}.\\

\noindent $(1) \Rightarrow (2)$ The proof goes by induction. For $v \in V$ we have $C(v) = 0,$ 
which implies $0 = (F \otimes F) \circ C_{\overline{T}V} = C_{\overline{T}W} \circ F$ and so $F(v) \in W.$\\
Assuming the claim is true for all natural numbers less than $n$, 
$$(F \otimes F) \circ C|_{V^{\otimes n}} = (F \otimes F) \circ \sum_{i = 1}^{n-1} \left(\mathbb{1}^{\otimes i}\right) \otimes \left(\mathbb{1}^{\otimes n-i}\right) $$$$= \sum_{i = 1}^{n-1} \left(F|_{V^{\otimes i}}\right) \otimes \left(F|_{V^{\otimes n-i}}\right)$$
and by induction hypothesis $F|_{V^{\otimes m}} = f_m + \sum_{B(m)} f_{r_1} \otimes \mbox{\dots} \otimes f_{r_k}$  
for all $n > m \geq 1$. Lemma~\ref{k1l3} gives
$$\sum_{i = 1}^{n-1} \left(F|_{V^{\otimes i}}\right) \otimes \left(F|_{V^{\otimes n-i}}\right) = C_{\overline{T}W} \circ \left(\sum_{i = 1}^{n-1} f_i \otimes F|_{V^{\otimes n-i}}\right)$$
and because $C_{\overline{T}W}$ is linear, $F|_{V^{\otimes n}}$ differs from 
$\sum_{i = 1}^{n-1} f_i \otimes F|_{V^{\otimes n-i}}$ by a linear map $f_n: V^{\otimes n} \to W$. 
Then $F|_{V^{\otimes n}}$ is of the required form and the proof is complete.
\end{proof}

\begin{lemma}\label{k1t9}
Let $F: \left(\overline{T}V, \delta^V\right) \to \left(\overline{T}W, \delta^W\right)$ be a linear map 
of degree $0$ such that $F|_{V^{\otimes n}} = f_n + \sum_{B(n)} f_{r_1} \otimes \mbox{\dots} \otimes f_{r_k},$ 
with all $\{f_n: V^{\otimes n} \to W\}_{n \geq 1}$ linear of degree $0$. Then the following are equivalent:
\begin{enumerate} 
\item $\delta^W \circ F = F \circ \delta^V$, 
\item for all $n \geq 1$ holds 
\begin{equation}
\label{T9}
\begin{split}
& \delta^W_1(f_n) + \sum_{B(n)} \delta_{k}^W\left(f_{r_{1}} \otimes \mbox{\dots} \otimes f_{r_{k}} \right)= f_1\left(\delta^V_{n}\right) \\ 
& +\sum_{i = 1}^{n} f_{n}\left(\mathbb{1}_V^{\otimes  i-1} \otimes \delta^V_1 \otimes \mathbb{1}_V^{\otimes  n-i}\right) + \sum_{A(n)} f_k\left(\mathbb{1}_V^{\otimes  i-1} \otimes \delta^V_\ell \otimes \mathbb{1}_V^{\otimes  k - i}\right).	
\end{split}	
\end{equation}
\end{enumerate}
\end{lemma}
\begin{proof}
$(1) \Rightarrow (2)$: The proof goes by induction. The restriction to $V$,
$\delta^W \circ F|_{V} = F \circ \delta^V|_{V}$, corresponds to $\delta_1^W \circ f_1 = f_1 \circ \delta_1^V.$
We now assume \eqref{T9} applies to all natural numbers less than $n$. We expand both sides of \eqref{T9},
$$\delta^W \circ F|_{V^{\otimes n}} =$$
$$\delta_1^W\left(f_n\right) + \sum_{B(n)} \sum_{a,b} f_{r_1} \otimes \mbox{\dots} \otimes f_{r_a} \otimes \delta_{b}^W\left(f_{r_{a+1}} \otimes \mbox{\dots} \otimes f_{r_{a+b}}\right) \otimes f_{r_{a+b+1}} \otimes \mbox{\dots} \otimes f_{r_k},$$
$$F \circ \delta^V|_{V^{\otimes n}} =$$
$$f_1\left(\delta^V_1\right) + \sum_{B(n)} \sum_{j, \ell} f_{r_1} \otimes \mbox{\dots} \otimes f_{r_{i-1}} \otimes f_{r_i}\left(\mathbb{1}_V^{\otimes  j} \otimes \delta^V_\ell \otimes \mathbb{1}_V^{\otimes r_i - j - 1}\right) \otimes f_{r_{i+1}} \otimes \mbox{\dots} \otimes f_{r_k}$$
and compare the terms of same homogeneities. We fix $j \geq 1$ and $r_1, \mbox{\dots}, r_j \geq 1, r_1+ \mbox{\dots} + r_j < n$ and $0 \leq m \leq j$, and focus on terms of the form $f_{r_1} \star \mbox{\dots} \otimes f_{r_{i-1}} \otimes \star \otimes f_{r_{i}} \otimes \mbox{\dots} \otimes f_{r_j}$, where $\star$ is an expression of the form $\delta_{\star}^W\left(f_{{\star}} \otimes \mbox{\dots} \otimes f_{\star}\right)$ or $f_{\star}\left(\mathbb{1}_V^{\otimes  \star} \otimes \delta^V_\star \otimes \mathbb{1}_V^{\otimes  \star}\right)$.\\\\
Terms on the right hand side of the form $f_{r_1} \otimes \mbox{\dots} \otimes f_{r_{i-1}} \otimes \delta_{\star}^W\left(f_{{\star}} \otimes \mbox{\dots} \otimes f_{\star}\right) \otimes f_{r_{i}} \otimes \mbox{\dots} \otimes f_{r_j}$ 
correspond to
$$f_{r_1} \otimes \mbox{\dots} \otimes f_{r_{m}} \otimes\left(\delta^W_1(f_n') + \sum_{B(n')} \delta_{k}^W\left(f_{r'_{1}} \otimes \mbox{\dots} \otimes f_{r'_{k}}\right) \right)\otimes f_{r_{m+1}} \otimes \mbox{\dots} \otimes f_{r_j},$$
while the terms of the form
 $f_{r_1} \otimes \mbox{\dots} \otimes f_{r_{i-1}} \otimes f_{\star}\left(\mathbb{1}_V^{\otimes  \star} \otimes \delta^V_\star \otimes \mathbb{1}_V^{\otimes  \star}\right) \otimes f_{r_{i}} \otimes \mbox{\dots} \otimes f_{r_j}$ 
correspond to
$$ f_{r_1} \otimes \mbox{\dots} \otimes f_{r_{m}} \otimes $$
$$\otimes\Bigg(f_1\left(\delta^V_{n'}\right) + \sum_{i = 1}^{n'} f_{n'}\left(\mathbb{1}_V^{\otimes  i-1} \otimes \delta^V_1 \otimes \mathbb{1}_V^{\otimes  n'-i}\right) $$$$+ \sum_{A(n')} f_k\left(\mathbb{1}_V^{\otimes  i-1} \otimes \delta^V_\ell \otimes \mathbb{1}_V^{\otimes  k - i}\right)\Bigg) \otimes \otimes f_{r_{m+1}} \otimes \mbox{\dots} \otimes f_{r_j}$$
with $n' = n - r_1+ \mbox{\dots} + r_j.$ Because $n' < n,$ they fulfill the equality \eqref{T9}
and hence are equal. Subtracting from both sides all elements of homogeneity greater than $1$,
we arrive at
$$\delta^W_1(f_n) + \sum_{B(n)} \delta_{k}^W\left(f_{r_{1}} \otimes \mbox{\dots} \otimes f_{r_{k}} \right)$$
$$= f_1\left(\delta^V_{n}\right) + \sum_{i = 1}^{n} f_{n}\left(\mathbb{1}_V^{\otimes  i-1} \otimes \delta^V_1 \otimes \mathbb{1}_V^{\otimes  n-i}\right) $$$$+ \sum_{A(n)} f_k\left(\mathbb{1}_V^{\otimes  i-1} \otimes \delta^V_\ell \otimes \mathbb{1}_V^{\otimes  k - i}\right).$$
However, this equality is true by \eqref{T9} for $n$.\\
 
\noindent $(2) \Rightarrow (1)$: This implication can be again reduced to the previous one.
\end{proof}

\begin{definition}\label{k1d10}
Let $\delta^V$ be a codifferential on $(\overline{T}V, C)$ and $\delta^W$ be a codifferential on 
$(\overline{T}W, C)$. Let $F: \left(\overline{T}V, C, \delta^V\right) \to \left(\overline{T}W, C, \delta^W\right)$ 
and $G: \left(\overline{T}V, C, \delta^V\right) \to \left(\overline{T}W, C, \delta^W\right)$ be 
morphisms. $F$ and $G$ are homotopy equivalent provided there exist linear maps 
$H: \overline{T}V \to \overline{T}W$ of degree $1$ such that 
$C_{\overline{T}W} \circ H = \left(F \otimes H + H \otimes G\right) \circ C_{\overline{T}V}$ and 
$F - G = H\delta^V + \delta^W H.$ The map $H$ is a homotopy between $F$ a $G$.
\end{definition}

\begin{remark}
Theorem ~\ref{k1v4} implies that $H: \overline{T}V \to \overline{T}W$ of degree $1$ fulfills 
$C_{\overline{T}W} \circ H = \left(F \otimes H + H \otimes G\right) \circ C_{\overline{T}V}$ 
if and only if there is a set of maps $\{h_n: V^{\otimes n} \to W\}_{n \geq 1}$ of degree $1$ 
such that $H|_{V^{\otimes n}} = h_n + \sum_{B(n), r_i > 0} f_{r_1} \otimes \mbox{\dots} \otimes f_{r_{i-1}} \otimes h_{r_i} \otimes g_{r_{i+1}} \otimes \mbox{\dots} \otimes g_{r_k}.$
\end{remark}

\begin{theorem}\label{k1v11}
We retain the assumptions of Definition \ref{k1d10}, and in addition assume the existence
of the set of linear maps
 $\{e_n: V^{\otimes n} \to W\}_{n \geq 1}$, $\{g_n: V^{\otimes n} \to W\}_{n \geq 1}$ of even degree 
$d$ such that $E|_{V^{\otimes n}} = e_n + \sum_{B(n)} e_{r_1} \otimes \mbox{\dots} \otimes e_{r_{k}}$ and $G|_{V^{\otimes n}} = g_n + \sum_{B(n)} g_{r_1} \otimes \mbox{\dots} \otimes g_{r_{k}}$. Let $F: \overline{T}V \to \overline{T}W$ be 
a linear mapping for which there exists a set of linear maps  $\{f_n: V^{\otimes n} \to W\}_{n \geq 1}$ of odd
degree $d+1$ fulfilling
$$F|_{V^{\otimes n}} = f_n + \sum_{B(n), r_i > 0} e_{r_1} \otimes \mbox{\dots} \otimes e_{r_{i-1}} \otimes f_{r_i} \otimes g_{r_{i+1}} \otimes \mbox{\dots} \otimes g_{r_k}.$$
Then the following assertions are equivalent:
\begin{enumerate}
\item  $E - G = F\delta^V + \delta^W F$,
\item $e_n - g_n = f_1(\delta^V_n) + \sum_{i=1}^{n}f_n(\mathbb{1}_V^{\otimes  i-1} \otimes \delta^V_1 \otimes \mathbb{1}_V^{\otimes  n-i}) + \sum_{A(n)} f_k(\mathbb{1}_V^{\otimes  i-1} \otimes \delta^V_\ell \otimes \mathbb{1}_V^{\otimes  k-i}) + \delta_1^W(f_n) + \sum_{B(n), r_i > 0} \delta^W_k(e_{r_1} \otimes \mbox{\dots} \otimes e_{r_{i-1}} \otimes f_{r_i} \otimes g_{r_{i+1}} \otimes \mbox{\dots} \otimes g_{r_k})$ for all ${n \geq 1}$.
\end{enumerate}
\end{theorem}
\begin{proof}
The proof can be done along the same lines as the proofs of Lemma~\ref{k1t6} and Lemma~\ref{k1t9}.
\end{proof}

\subsection{Codifferentials and $A_\infty$ algebras}

\begin{definition}\label{k1d11}
For $V$ graded we define $sV$ in such a way that $\left(sV\right)_i = V_{i-1}$. The graded modules 
$V$ and $sV$ are canonically isomorphic: $s: V \to sV$ is a linear map of degree $1$ called suspension, 
$\omega : sV \to V$ is a linear map of degree $-1$ called desuspension.
\end{definition}

\begin{remark}
We have $s^{\otimes n} \otimes \omega^{\otimes n} = (-1)^{\frac{n(n-1)}{2}}$ by the Koszul sign convention.
\end{remark}

\begin{theorem}\label{k1v13}
The following claims are equivalent:
\begin{enumerate}
\item $\{\mu_n: V^{\otimes n} \to V; |\mu_n| = n-2\}_{n \geq 1}$ is $A_\infty$ structure on $V$,
\item The linear maps $\delta_n = s \circ \mu_n \circ \omega^{\otimes n}$ are of degree $-1$, and
are the components of a codifferential on $\overline{T}sV$ in the sense of Theorem~\ref{k1v5}.
\end{enumerate}
\end{theorem}
\begin{proof}
$(2) \Rightarrow (1)$: $\delta_n = s \circ \mu_n \circ \omega^{\otimes n}$ are the components
of a codifferential, and so we have for all $n \geq 1$  
$$\delta_1(\delta_n) + \sum_{i = 1}^{n} \delta_n \left(\mathbb{1}_V^{\otimes  i-1} \otimes \delta_1 \otimes \mathbb{1}_V^{\otimes  n-i}\right) $$$$+ \sum_{A(n)} \delta_k\left(\mathbb{1}_V^{\otimes  i-1} \otimes \delta_\ell \otimes \mathbb{1}_V^{\otimes  k - i}\right) = 0.$$
This can be rewritten, by Koszul sign convention, as 
$$\delta_1(\delta_n) = s \circ \mu_1 \circ \omega \circ s \circ \mu_n \circ \omega^{\otimes n} = s \circ \mu_1(\mu_n) \circ \omega^{\otimes n},$$
$$\sum_{i = 1}^{n} \delta_n \left(\mathbb{1}_V^{\otimes  i-1} \otimes \delta_1 \otimes \mathbb{1}_V^{\otimes  n-i}\right) 
= \sum_{i = 1}^{n} s \circ \mu_n \circ \omega^{\otimes n} \left(\mathbb{1}_V^{\otimes  i-1} \otimes s \circ \mu_1 \circ \omega \otimes \mathbb{1}_V^{\otimes  n-i}\right)$$
$$=\sum_{i = 1}^{n} (-1)^{n-i} s \circ \mu_n\left(\omega^{\otimes i-1} \otimes \mu_1 \circ \omega \otimes \omega^{\otimes n-i}\right)$$
$$= \sum_{i = 1}^{n} (-1)^{n-i}(-1)^{i-1} s \circ \mu_n\left(\mathbb{1}_V^{\otimes  i-1} \otimes \mu_1 \otimes \mathbb{1}_V^{\otimes  n-i}\right) \circ \omega^{\otimes n},$$
$$\sum_{A(n)} \delta_k\left(\mathbb{1}_V^{\otimes  i-1} \otimes \delta_\ell \otimes \mathbb{1}_V^{\otimes  k - i}\right) $$$$= \sum_{A(n)} s \circ \mu_k \circ \omega^{\otimes k}\left(\mathbb{1}_V^{\otimes  i-1} \otimes s \circ \mu_\ell \circ \omega^{\otimes \ell} \otimes \mathbb{1}_V^{\otimes  k - i}\right) $$
$$= \sum_{A(n)} (-1)^{k-i}s \circ \mu_k \left(\omega^{\otimes i-1} \otimes \mu_\ell \circ \omega^{\otimes \ell} \otimes \omega^{\otimes k - i}\right) $$
$$= \sum_{A(n)} (-1)^{k-i} (-1)^{\ell(i-1)}s \circ \mu_k \left(\mathbb{1}_V^{\otimes  i-1} \otimes \mu_\ell  \otimes \mathbb{1}_V^{\otimes  k - i}\right)\circ \omega^{\otimes n}.$$
The mappings $s$ and $\omega$ are linear, hence  
$$s \circ \left(\mu_1(\mu_n) + \sum_{i = 1}^{n} (-1)^{n-1} \mu_n\left(\mathbb{1}_V^{\otimes  i-1} \otimes \mu_1 \otimes \mathbb{1}_V^{\otimes  n-i}\right) \right.$$
$$ +\left. \sum_{A(n)} (-1)^{i(\ell+1) + n-1} \mu_k \left(\mathbb{1}_V^{\otimes  i-1} \otimes \mu_\ell  \otimes \mathbb{1}_V^{\otimes  k - i}\right)\right) \circ \omega^{\otimes n} = 0.$$
$(1) \Rightarrow (2)$: This can be easily reduced to the proof of the previous implication.
\end{proof}

\begin{theorem}\label{k1v14}
The following claims are equivalent:
\begin{enumerate}
\item $\{\varphi_n: V^{\otimes n} \to W; |\varphi_n| = n-1\}_{n \geq 1}$ is $A_\infty$ morphism
from $(V, \pmb{\mu})$ to $(W, \pmb{\nu})$,
\item the mappings 
$$f_n = s_W \circ \varphi_n \circ \omega_V^{\otimes n}$$
are of degree $0$, and are the components of $A_\infty$ morphism from $(\overline{T}sV, \delta^V)$ to 
$(\overline{T}sW, \delta^W)$ in the sense of Lemma~\ref{k1t8}. The codifferentials are given
by $A_\infty$ structures on $V$ and $W$, respectively, via Theorem \ref{k1v13}.
\end{enumerate}
The following claims are equivalent:
\begin{enumerate}
\item $\{h_n: V^{\otimes n} \to W; |h_n| = n\}_{n \geq 1}$ is $A_\infty$ homotopy between 
$A_\infty$ morphisms $\pmb{\varphi}$ with components  
$\{\varphi_n: V^{\otimes n} \to W; |\varphi_n| = n-1\}_{n \geq 1}$ and 
$\pmb{\psi}$ with components $\{\psi_n: V^{\otimes n} \to W; |\psi_n| = n-1\}_{n \geq 1}$, respectively, from 
$(V, \pmb{\mu})$ to $(W, \pmb{\nu})$,
\item 
$$h_n = s_W \circ h_n \circ \omega_V^{\otimes n}$$
are of degree $1$, and are the components of $A_\infty$ homotopy between morphisms $\pmb{F}$ and 
$\pmb{G}$ from $(\overline{T}sV, \delta^V)$ to $(\overline{T}sW, \delta^W)$, where $\pmb{F}$ 
corresponds to $\pmb{\varphi}$ and $\pmb{G}$ corresponds to $\pmb{\psi}$ 
in the sense of the first equivalence in the theorem. 
The codifferentials are given by $A_\infty$ structures on $V$ and $W$, respectively, as in Theorem~\ref{k1v13}.
\end{enumerate}
\end{theorem}
\begin{proof}
The proof goes along the same lines as in Theorem~\ref{k1v13}.
\end{proof}

\section{Homotopy transfer of $A_\infty$ algebras}\label{sect:hotr}
\label{sec:3}

The starting point for the present section are 
the chain complexes $(V, \partial_V)$ and $(W, \partial_W)$, 
$f: V \to W$, $g: W \to V$ their morphisms such that $gf$ is 
homotopy equivalent to $\mathbb{1}_V$ by a homotopy $h$. Let $(V, \partial_V)$ 
be equipped with $A_\infty$ algebra structure, which means that there 
is a set of multilinear maps $\pmb{\mu} = (\mu_2, \mu_3, \mbox{\dots})$ 
satisfying the relations \eqref{ainftyrelations}.
We would like to induce $A_\infty$ structure 
$(W, \partial_W, \nu_2, \nu_3, \mbox{\dots})$ on 
$(W, \partial_W)$ by transferring $(V, \partial_V, \mu_2, \mu_3, \mbox{\dots})$,
as well as the morphisms of $A_\infty$ algebras $\pmb{\psi} = (g, \psi_2, \psi_3, \mbox{\dots})$ 
from $(W, \partial_W, \pmb{\nu})$ to $(V, \partial_V, \pmb{\mu})$ and 
$\pmb{\varphi} = (f, \varphi_2, \varphi_3, \mbox{\dots})$ acting in the opposite 
direction such that their composition $\pmb{\psi \varphi}$ is $A_\infty$ homotopy 
equivalent with the identity map via $\pmb{H}= (h, H_2, H_3, \mbox{\dots})$.

The strategy to solve this problem, cf. \cite{Markl06}, suggests to 
construct the set of maps $\{\pmb{p}_n: V^{\otimes n} \to V\}_{n \geq 2}$ of degree $n-2$ 
called $\pmb{p}-$kernels, and the set of maps $\{\pmb{q}_n: V^{\otimes n} \to V\}_{n \geq 1}$ 
of degree $n-1$ called $\pmb{q}-$kernels in such a way that  
$\nu_n, \varphi_n, \psi_n$ and $H_n$ defined by 
\begin{equation}
    \label{Antanz}
	\nu_n := f \circ \pmb{p}_n \circ g^{\otimes n}, \;\;\; \varphi_n := f \circ \pmb{q}_n, \;\;\; \psi_n := h \circ \pmb{p}_n \circ g^{\otimes n}, \;\;\; H_n = h \circ \pmb{q}_n,
\end{equation}
fulfill the transfer problem of $A_\infty$ algebra as discussed in the 
previous paragraph.

We shall first introduce the $\pmb{p}-$kernels and based on them we introduce the $\pmb{q}-$kernels
later on.
Apart from \eqref{A} a \eqref{B}, we shall rely on the notation (cf., \cite{Markl06}) 
\begin{equation}
	\tag{C}
	\label{C}
\begin{split}
	C(n) = \{k,i, r_1, \mbox{\dots}, r_i\in\mN | \, 2 \leq k \leq n, 1 \leq i \leq k,\\ r_1, \mbox{\dots}, r_i \geq 1, r_1 + \mbox{\dots} + r_i + k - i = n\},	
\end{split}
\end{equation}
for $n \in \mathbb{N}$, and
\begin{equation}
\tag{$\vartheta$}
\label{theta}
\vartheta(u_1, \mbox{\dots}, u_k) = \sum_{1 \leq i < j \leq k} u_i(u_j + 1), 	
\end{equation}
for arbitrary $u_1, \mbox{\dots}, u_k, k \in \mathbb{N}$.

\subsection{\texorpdfstring{$\pmb{p}-$kernels}{p-kernels}}

\begin{lemma}\label{k2t15}
The $\pmb{p}-$kernels together with $\partial_W$ constitute an $A_\infty$ structure
on $(W, \partial_W)$ via \eqref{Antanz} if and only if for all $n \geq 2$ holds
$$f \circ \left(\partial_V\pmb{p}_n - \sum_{u=1}^{n}(-1)^{n} \pmb{p}_n(\mathbb{1}_V^{\otimes  u-1} \otimes \partial_V \otimes \mathbb{1}_V^{\otimes  n-u})  \right.$$
$$\left. - \sum_{A(n)}(-1)^{i(\ell+1) + n}\pmb{p}_k(\mathbb{1}_V^{\otimes  i-1} \otimes gf \circ \pmb{p}_\ell \otimes \mathbb{1}_V^{\otimes  k - i})\right) \circ g^{\otimes n} = 0$$
\end{lemma}
\begin{proof}
$(W, \partial_W, \nu_2 \mbox{\dots})$ is an $A_\infty$ algebra if we have for all $n \geq 1$
$$\partial_W \nu_n - \sum_{u=1}^{n}(-1)^{n} \nu_n(\mathbb{1}_W^{\otimes u-1} \otimes \partial_W \otimes \mathbb{1}_W^{\otimes n-u})$$ $$- \sum_{A(n)}(-1)^{i(\ell+1) + n}\nu_k(\mathbb{1}_W^{\otimes i-1} \otimes \nu_\ell \otimes \mathbb{1}_W^{\otimes k - i})= 0.$$ 
This is true for $n = 1$, because $(W, \partial_W)$ is the chain complex ($f \circ \partial_V = \partial_W \circ f$ 
and analogously for $g$.) Now expand $\nu_n$ following \eqref{Antanz}:
$$\partial_W \nu_n - \sum_{u=1}^{n}(-1)^{n} \nu_n(\mathbb{1}_W^{\otimes u-1} \otimes \partial_W \otimes \mathbb{1}_W^{\otimes n-u})$$$$- \sum_{A(n)}(-1)^{i(\ell+1) + n}\nu_k(\mathbb{1}_W^{\otimes i-1} \otimes \nu_\ell \otimes \mathbb{1}_W^{\otimes k - i}) $$
$$= \partial_W \left(f \circ \pmb{p}_n \circ g^{\otimes n} \right)  - \sum_{u=1}^{n}(-1)^{n} \left(f \circ \pmb{p}_n \circ g^{\otimes n} \right)(\mathbb{1}_W^{\otimes u-1} \otimes \partial_W \otimes \mathbb{1}_W^{\otimes n-u})$$
$$- \sum_{A(n)}(-1)^{i(\ell+1) + n}\left(f \circ \pmb{p}_k \circ g^{\otimes k} \right)(\mathbb{1}_W^{\otimes i-1} \otimes \left(f \circ \pmb{p}_\ell \circ g^{\otimes \ell} \right) \otimes \mathbb{1}_W^{\otimes k - i}).$$
 Because both $f$ and $g$ are linear maps of degree $0$, this equals to 
 $$f\circ \left(\partial_V \circ \pmb{p}_n  \right) \circ g^{\otimes n}  - f \circ \left(\sum_{u=1}^{n}(-1)^{n} \pmb{p}_n(g^{\otimes u-1} \otimes g \circ \partial_W \otimes g^{\otimes n-u})\right)$$
 $$- f \circ\left(\sum_{A(n)}(-1)^{i(\ell+1) + n}\pmb{p}_k(g^{\otimes i-1} \otimes gf \circ \pmb{p}_\ell \circ g^{\otimes \ell} \otimes g^{\otimes k - i})\right),$$
which is
 $$f\circ \left(\partial_V \circ \pmb{p}_n  \right) \circ g^{\otimes n}  - f \circ \left(\sum_{u=1}^{n}(-1)^{n} \pmb{p}_n(\mathbb{1}_V^{\otimes  u-1} \otimes \partial_V \otimes \mathbb{1}_V^{\otimes  n-u})\right)  \circ g^{\otimes n} $$
 $$- f \circ\left(\sum_{A(n)}(-1)^{i(\ell+1) + n}\pmb{p}_k(\mathbb{1}_V^{\otimes  i-1} \otimes gf \circ \pmb{p}_\ell \otimes \mathbb{1}_V^{\otimes  k - i})\right)\circ g^{\otimes n} = 0.$$
\end{proof}

\begin{lemma}\label{k2t16}
Let us assume that $\pmb{p}-$kernels induce the transfer of $A_\infty$ algebra as formulated 
above, and they fulfill ($n \geq 2$)
\begin{equation}
\label{T16}
	\begin{split}
		\partial_V\pmb{p}_n - \sum_{u=1}^{n}(-1)^{n} \pmb{p}_n(\mathbb{1}_V^{\otimes  u-1} \otimes \partial_V \otimes \mathbb{1}_V^{\otimes  n-u})\\
		-\sum_{A(n)}(-1)^{i(\ell+1) + n}\pmb{p}_k(\mathbb{1}_V^{\otimes  i-1} \otimes gf \circ \pmb{p}_\ell \otimes \mathbb{1}_V^{\otimes  k - i})= 0.
	\end{split}
\end{equation} 
Then
$$\pmb{p}_n \circ g^{\otimes n}= \left(\sum_{B(n)} (-1)^{\vartheta(r_1, \mbox{\dots}, r_k)}\mu_k(h \circ \pmb{p}_{r_1} \otimes \mbox{\dots} \otimes h \circ \pmb{p}_{r_k})\right) \circ g^{\otimes n},$$
where we define $h \circ \pmb{p}_1 = \mathbb{1}_V$.
\end{lemma}
\begin{proof}
According to \eqref{T16} these $\pmb{p}-$kernels induce 
$A_\infty$ structure on $(W, \partial_W)$ by Lemma~\ref{k2t15}.
It remains to verify that they give $A_\infty$ morphism from 
$(V, \partial_V, \pmb{\mu})$ to $(W, \partial_W, \pmb{\nu})$, i.e.
$$\partial_V \psi_n + \sum_{B(n)} (-1)^{\vartheta(r_1, \mbox{\dots}, r_k)}\mu_k(\psi_{r_1} \otimes \mbox{\dots} \otimes \psi_{r_k})$$
$$= \psi_1 \nu_n -\sum_{u=1}^{n}(-1)^{n} \psi_n(\mathbb{1}_W^{\otimes u-1} \otimes \partial_W \otimes \mathbb{1}_W^{\otimes n-u}) $$$$- \sum_{A(n)}(-1)^{i(\ell+1) + n}\psi_k(\mathbb{1}_W^{\otimes i-1} \otimes \nu_\ell \otimes \mathbb{1}_W^{\otimes k - i}),$$
 which by \eqref{Antanz} can be formulated as 
$$\partial_V h \circ \pmb{p}_n \circ g^{\otimes n} + \left(\sum_{B(n)} (-1)^{\vartheta(r_1, \mbox{\dots}, r_k)}\mu_k(h \circ \pmb{p}_{r_1} \otimes \mbox{\dots} \otimes h \circ \pmb{p}_{r_k})\right) \circ g^{\otimes n}$$
$$= gf \circ \pmb{p}_n \circ g^{\otimes n} -h \circ \left(\sum_{u=1}^{n}(-1)^{n} \pmb{p}_n(\mathbb{1}_V^{\otimes  u-1} \otimes \partial_V \otimes \mathbb{1}_V^{\otimes  n-u})\right) \circ g^{\otimes n} $$
$$- h \circ \left(\sum_{A(n)}(-1)^{i(\ell+1) + n}\pmb{p}_k(\mathbb{1}_V^{\otimes  i-1} \otimes gf \circ \pmb{p}_\ell \otimes \mathbb{1}_V^{\otimes  k - i})\right)\circ g^{\otimes n}.$$
Due to $gf - \mathbb{1}_V = \partial_V h + h \partial_V,$ we have 
$$\left(\sum_{B(n)} (-1)^{\vartheta(r_1, \mbox{\dots}, r_k)}\mu_k(h \circ \pmb{p}_{r_1} \otimes \mbox{\dots} \otimes h \circ \pmb{p}_{r_k})\right) \circ g^{\otimes n}$$
$$= \pmb{p}_n \circ g^{\otimes n} -h \circ \left(-\partial_V \pmb{p}_n + \sum_{u=1}^{n}(-1)^{n} \pmb{p}_n(\mathbb{1}_V^{\otimes  u-1} \otimes \partial_V \otimes \mathbb{1}_V^{\otimes  n-u})\right) \circ g^{\otimes n} $$
$$- h \circ \left(\sum_{A(n)}(-1)^{i(\ell+1) + n}\pmb{p}_k(\mathbb{1}_V^{\otimes  i-1} \otimes gf \circ \pmb{p}_\ell \otimes \mathbb{1}_V^{\otimes  k - i})\right)\circ g^{\otimes n}.$$
By assumption \eqref{T16}, we obtain
$$ h \circ \left(-\partial_V \pmb{p}_n + \sum_{u=1}^{n}(-1)^{n} \pmb{p}_n(\mathbb{1}_V^{\otimes  u-1} \otimes \partial_V \otimes \mathbb{1}_V^{\otimes  n-u})\right) \circ g^{\otimes n} $$
$$ + h \circ \left(\sum_{A(n)}(-1)^{i(\ell+1) + n}\pmb{p}_k(\mathbb{1}_V^{\otimes  i-1} \otimes gf \circ \pmb{p}_\ell \otimes \mathbb{1}_V^{\otimes  k - i})\right) \circ g^{\otimes n} = 0,$$
which reduces to
$$\left(\pmb{p}_n-\sum_{B(n)} (-1)^{\vartheta(r_1, \mbox{\dots}, r_k)}\mu_k(h \circ \pmb{p}_{r_1} \otimes \mbox{\dots} \otimes h \circ \pmb{p}_{r_k})\right) \circ g^{\otimes n}= 0.$$
\end{proof}

\begin{remark}
The assumption of Lemma~\ref{k2t16} can be weaken to
\begin{equation}
\label{T16g}
	\begin{split}
		\partial_V\pmb{p}_n\circ g^{\otimes n} - \sum_{u=1}^{n}(-1)^{n} \pmb{p}_n(\mathbb{1}_V^{\otimes  u-1} \otimes \partial_V \otimes \mathbb{1}_V^{\otimes  n-u})\circ g^{\otimes n} \\
		 -\sum_{A(n)}(-1)^{i(\ell+1) + n}\pmb{p}_k(\mathbb{1}_V^{\otimes  i-1} \otimes gf \circ \pmb{p}_\ell \otimes \mathbb{1}_V^{\otimes  k - i}) \circ g^{\otimes n} = 0,
	\end{split}
\end{equation} 
where \eqref{T16g} is fulfilled if the $\pmb{p}-$kernels define $A_\infty$ structure on $(W, \partial_W)$, and 
$f$ is a monomorphism. In the situation of interest is $f$, however, assumed to be an epimorphism.\\
\end{remark}

\begin{definition}[$\pmb{p}-$kernels, \cite{Markl06}]\label{k2d17}
We define for each $n \geq 2$: 
\begin{align}\label{pkernels}
\pmb{p}_n = \sum_{B(n)} (-1)^{\vartheta(r_1, \mbox{\dots}, r_k)}\mu_k(h \circ \pmb{p}_{r_1} \otimes \mbox{\dots} \otimes h \circ \pmb{p}_{r_k}),
\end{align}
where $h \circ \pmb{p}_1 = \mathbb{1}_V$, with $B(n)$ given in \eqref{B} and $\vartheta(r_1, \mbox{\dots}, r_k)$ 
given in \eqref{theta}.
\end{definition}

\begin{remark}
For $\pmb{p}-$kernels there exists a non-inductive explicit expression. Each term in the $\pmb{p}-$kernel 
can be represented by a rooted plane tree, and there is a function which associates to a rooted plane tree a
sign corresponding to our inductive definition.
\end{remark}

\begin{theorem}\label{k2v18}
The $\pmb{p}-$kernels introduced in \cite{Markl06} satisfy
\begin{equation}
\tag{\ref{T16}}\label{pkernelform}
	\begin{split}
		\partial_V\pmb{p}_n - \sum_{u=1}^{n}(-1)^{n} \pmb{p}_n(\mathbb{1}_V^{\otimes  u-1} \otimes \partial_V \otimes \mathbb{1}_V^{\otimes  n-u})\\
		-\sum_{A(n)}(-1)^{i(\ell+1) + n}\pmb{p}_k(\mathbb{1}_V^{\otimes  i-1} \otimes gf \circ \pmb{p}_\ell \otimes \mathbb{1}_V^{\otimes  k - i})= 0,
	\end{split}
\end{equation}  for all $n \geq 2$.
\end{theorem}
\begin{proof}
Let us first simplify our situation by passing to the suspension 
$\overline{T}sV$ with the induced codifferential $\delta$.
Because $s$ and $\omega$ are by Definition~\ref{k1d11} izomorphisms, \eqref{pkernelform} is true if and only if 
$$s\circ \left(\partial_V\pmb{p}_n - \sum_{u=1}^{n}(-1)^{n} \pmb{p}_n(\mathbb{1}_V^{\otimes  u-1} \otimes \partial_V \otimes \mathbb{1}_V^{\otimes  n-u})\right)\circ \omega^{\otimes n} $$
$$= s \circ \left(\sum_{A(n)}(-1)^{i(\ell+1) + n}\pmb{p}_k(\mathbb{1}_V^{\otimes  i-1} \otimes gf \circ \pmb{p}_\ell \otimes \mathbb{1}_V^{\otimes  k - i})\right) \circ \omega^{\otimes n}.$$
Introducing $\hat{\pmb{p}}_m = s \circ \pmb{p}_m \circ \omega^{\otimes m},$ $\hat{g} = s \circ g \circ \omega$ and $\hat{f} = s \circ f \circ \omega$ ($|\hat{\pmb{p}}_m| = -1$, $|\hat{g}| = |\hat{f}| = 0$), we have
$$\delta_1\hat{\pmb{p}}_n + \sum_{u=1}^{n}\hat{\pmb{p}}_n(\mathbb{1}_V^{\otimes  u-1} \otimes \delta_1 \otimes \mathbb{1}_V^{\otimes  n-u}) + \sum_{A(n)}\hat{\pmb{p}}_k(\mathbb{1}_V^{\otimes  i-1} \otimes \hat{g}\hat{f} \circ \hat{\pmb{p}}_\ell \otimes \mathbb{1}_V^{\otimes  k - i}) = 0.$$
The proof of the last claim goes by induction. The case $n = 2$ corresponds to
$$\delta_1 \delta_2 + \delta_2(\mathbb{1}_V \otimes \delta_1) + \delta_2(\delta_1 \otimes \mathbb{1}_V) = 0,$$
which is certainly true because $\{\delta_n: V^{\otimes n} \to V\}_{n \geq 1}$ are the components of the
codifferential on $\overline{T}sV$ (cf., \eqref{T6} for $n = 2$ in Lemma~\ref{k1t6}.)\\

\noindent By induction hypothesis, we assume the claim is true for all natural numbers less than $n$. 
The proof is naturally divided into three steps:\\

\noindent \textbf{I.} We shall first expand the term $\delta_1 \hat{\pmb{p}}_n$: we have 
$\hat{\pmb{p}}_n = s \circ \pmb{p}_n \circ \omega^{\otimes n},$ so by Definition~\ref{k2d17}
$$ \hat{\pmb{p}}_n = s \circ \left(\sum_{B(n)} (-1)^{\vartheta(r_1, \mbox{\dots}, r_k)}\mu_k(h \circ \pmb{p}_{r_1} \otimes \mbox{\dots} \otimes h \circ \pmb{p}_{r_k})\right) \circ \omega^{\otimes n} 
$$$$ =\sum_{B(n)} (-1)^{\vartheta(r_1, \mbox{\dots}, r_k)}(-1)^{\sigma}s \circ \mu_k(\omega \circ s \circ h \circ \pmb{p}_{r_1} \circ \omega^{\otimes r_1} \otimes \mbox{\dots}$$$$\mbox{\dots} \otimes \omega \circ s \circ h \circ \pmb{p}_{r_k} \circ \omega^{\otimes r_k})
$$  with $\sigma = \sum_{1 \leq i < j \leq k} r_i(r_j + 1)$. However, 
	$|s \circ h \circ \pmb{p}_{r_i} \circ \omega^{\otimes r_i}| = 1 + 1 + (r_i - 2) - r_i = 0$, so the last 
	display equals to
\begin{align}\nonumber	
 \sum_{B(n)} s \circ \mu_k \circ \omega^{\otimes k}(s \circ h \circ \omega \circ s \circ \pmb{p}_{r_1} \circ \omega^{\otimes r_1} \otimes \mbox{\dots} \otimes s \circ h \circ \omega \circ s \circ \pmb{p}_{r_k} \circ \omega^{\otimes r_k}).
\end{align}
Consequently,
\begin{equation}
\label{p-kernel}
	\hat{\pmb{p}}_n = \sum_{B(n)} \delta_k(\hat{h} \circ \hat{\pmb{p}}_{r_1} \otimes \mbox{\dots} \otimes \hat{h} \circ \hat{\pmb{p}}_{r_k}),\quad \hat{h} = s \circ h \circ \omega\quad (|\hat{h}| =1),
\end{equation}
and so
\begin{align}
\delta_1 \hat{\pmb{p}}_n =& \sum_{B(n)} \delta_1\delta_{k}(\hat{h} \circ \hat{\pmb{p}}_{r_1} \otimes \mbox{\dots} \otimes \hat{h} \circ \hat{\pmb{p}}_{r_{k}})
\nonumber \\
=&  -\sum_{B(n)} \left(\sum_{i = 1}^{k} \delta_k \left(\mathbb{1}_V^{\otimes  i-1} \otimes \delta_1 \otimes \mathbb{1}_V^{\otimes  k-i}\right)\right)(\hat{h} \circ \hat{\pmb{p}}_{r_1} \otimes \mbox{\dots} \otimes \hat{h} \circ \hat{\pmb{p}}_{r_k}) 
\nonumber \\ \nonumber 
&- \sum_{B(n)}\left(\sum_{A(k)} \delta_{k'}\left(\mathbb{1}_V^{\otimes  i-1} \otimes \delta_\ell \otimes \mathbb{1}_V^{\otimes  k' - i}\right) \right)(\hat{h} \circ \hat{\pmb{p}}_{r_1} \otimes \mbox{\dots} \otimes \hat{h} \circ \hat{\pmb{p}}_{r_k}).
\end{align}
The last summation can be rewritten as 
$$\sum_{B(n)}\left(\sum_{A(k)} \delta_{k'}\left(\mathbb{1}_V^{\otimes  i-1} \otimes \delta_\ell \otimes \mathbb{1}_V^{\otimes  k' - i}\right) \right)(\hat{h} \circ \hat{\pmb{p}}_{r_1} \otimes \mbox{\dots} \otimes \hat{h} \circ \hat{\pmb{p}}_{r_k}) $$
$$= \sum_{B(n)}\sum_{A(k)} \delta_{k'}\left(\hat{h} \circ \hat{\pmb{p}}_{r_1} \otimes \mbox{\dots} \otimes \delta_\ell(\hat{h} \circ \hat{\pmb{p}}_{r_i} \otimes \mbox{\dots} \otimes \hat{h} \circ \hat{\pmb{p}}_{r_{i+\ell}}) \otimes \mbox{\dots} \otimes \hat{h} \circ \hat{\pmb{p}}_{r_k}\right)$$ 
$$= \sum_{B(n),\, r_i > 1} \delta_{k}(\hat{h} \circ \hat{\pmb{p}}_{r_1} \otimes \mbox{\dots} \otimes \hat{\pmb{p}}_{r_i} \otimes \mbox{\dots} \otimes \hat{h} \circ \hat{\pmb{p}}_{r_{k}}),$$
where the last equality comes from the summation over all $r_1,\ldots,r_{i+l}$ with 
$r_i + \mbox{\dots} + r_{i+\ell}$ fixed. We conclude
\begin{align}
\delta_1 \hat{\pmb{p}}_n = & -\sum_{B(n), r_i > 1} \delta_{k}(\hat{h} \circ \hat{\pmb{p}}_{r_1} \otimes \mbox{\dots} \otimes (\delta_1\hat{h} + \mathbb{1}_V)\hat{\pmb{p}}_{r_i} \otimes \mbox{\dots} \otimes \hat{h} \circ \hat{\pmb{p}}_{r_{k}})
\nonumber \\ \nonumber
& -\sum_{B(n), r_i = 1} \delta_{k}(\hat{h} \circ \hat{\pmb{p}}_{r_1} \otimes \mbox{\dots} \otimes \delta_1\hat{h}\circ \hat{\pmb{p}}_{r_i} \otimes \mbox{\dots} \otimes \hat{h} \circ \hat{\pmb{p}}_{r_{k}}).
\end{align}

\textbf{II.} We shall apply the induction hypothesis to $\delta_1 \hat{\pmb{p}}_n$.
We remind the formal equality $\hat{h} \circ \hat{\pmb{p}}_{1} = \mathbb{1}_V$ and also
$gf - \mathbb{1}_V = \partial_Vh + h\partial_V$ equivalent to 
$\delta_1 \hat{h} + \mathbb{1}_V = \hat{g}\hat{f} - \hat{h} \delta_1$. Then
\begin{align}
 \label{V24_1}
	\delta_1 \hat{\pmb{p}}_n &= \sum_{B(n), r_i > 1} \delta_{k}(\hat{h} \circ \hat{\pmb{p}}_{r_1} \otimes 
	\mbox{\dots} \otimes (\hat{h}\delta_1 - \hat{g}\hat{f})\hat{\pmb{p}}_{r_i} \otimes \mbox{\dots} \otimes \hat{h} \circ \hat{\pmb{p}}_{r_{k}}) \nonumber \\
&-\sum_{B(n), r_i = 1} \delta_{k}(\hat{h} \circ \hat{\pmb{p}}_{r_1} \otimes \mbox{\dots} \otimes \delta_1\hat{h}\circ \hat{\pmb{p}}_{r_i} \otimes \mbox{\dots} \otimes \hat{h} \circ \hat{\pmb{p}}_{r_{k}}).
\end{align}
  The second part of the first term on the right hand side \eqref{V24_1} equals
$$
 -\sum_{B(n)} \delta_{k}(\hat{h} \circ \hat{\pmb{p}}_{r_1} \otimes \mbox{\dots} \otimes \hat{g}\hat{f} \circ \hat{\pmb{p}}_{r_i} \otimes \mbox{\dots} \otimes \hat{h} \circ \hat{\pmb{p}}_{r_{k}})$$
$$ = \sum_{B(n)} \delta_{k}(\hat{h} \circ \hat{\pmb{p}}_{r_1} \otimes \mbox{\dots} \otimes \hat{h} \circ \hat{\pmb{p}}_{1} \otimes \mbox{\dots} \otimes \hat{h} \circ \hat{\pmb{p}}_{r_{k}})(\mathbb{1}_V^{\otimes  s} \otimes \hat{g}\hat{f} \circ\hat{\pmb{p}}_{r_i} \otimes \mathbb{1}_V^{\otimes  n - s - r_i}),
$$
where $s = \sum_{j < i} r_j.$ The second term in \eqref{V24_1} equals 
$$-\sum_{B(n), r_i = 1} \delta_{k}(\hat{h} \circ \hat{\pmb{p}}_{r_1} \otimes \mbox{\dots} \otimes \delta_1\hat{h}\circ \hat{\pmb{p}}_{r_i} \otimes \mbox{\dots} \otimes \hat{h} \circ \hat{\pmb{p}}_{r_{k}}) $$ 
$$= -\sum_{B(n), r_i = 1} \delta_{k}(\hat{h} \circ \hat{\pmb{p}}_{r_1} \otimes \mbox{\dots} \otimes \hat{h}\circ \hat{\pmb{p}}_{r_i} \otimes \mbox{\dots} \otimes \hat{h} \circ \hat{\pmb{p}}_{r_{k}})(\mathbb{1}_V^{\otimes  s} \otimes \delta_1 \otimes \mathbb{1}_V^{\otimes  n - s - r_i}).$$ 
By induction hypothesis, we have for all $m < n$
$$\delta_1\hat{\pmb{p}}_m = -\sum_{u=1}^{m}\hat{\pmb{p}}_m(\mathbb{1}_V^{\otimes  u-1} \otimes \delta_1 \otimes \mathbb{1}_V^{\otimes  m-u}) - \sum_{A(m)}\hat{\pmb{p}}_k(\mathbb{1}_V^{\otimes  i-1} \otimes \hat{g}\hat{f} \circ \hat{\pmb{p}}_\ell \otimes \mathbb{1}_V^{\otimes  k - i}).$$
Finally, the first part of the first term \eqref{V24_1} equals
$$\sum_{B(n), r_i > 1} \delta_{k}(\hat{h} \circ \hat{\pmb{p}}_{r_1} \otimes \mbox{\dots} \otimes \hat{h} \circ\delta_1 \hat{\pmb{p}}_{r_i} \otimes \mbox{\dots} \otimes \hat{h} \circ \hat{\pmb{p}}_{r_{k}})$$ 
$$=-\sum_{B(n), r_i > 1} \delta_{k}\left(\hat{h} \circ \hat{\pmb{p}}_{r_1} \otimes \mbox{\dots} \otimes \sum_{u=1}^{r_i}\hat{h} \circ \hat{\pmb{p}}_{r_i}(\mathbb{1}_V^{\otimes  u-1} \otimes \delta_1 \otimes \mathbb{1}_V^{\otimes  r_i-u}) \otimes \mbox{\dots} \otimes \hat{h} \circ \hat{\pmb{p}}_{r_{k}}\right) $$
$$- \sum_{B(n), r_i > 1} \delta_{k}\!\!\left(\hat{h} \circ \hat{\pmb{p}}_{r_1} \otimes \!\mbox{\dots}\! \otimes \! \sum_{A(r_i)}\hat{h} \circ \hat{\pmb{p}}_k(\mathbb{1}_V^{\otimes  i-1} \otimes \hat{g}\hat{f} \circ \hat{\pmb{p}}_\ell \otimes \mathbb{1}_V^{\otimes  k - i}) \otimes \!\mbox{\dots}\! \otimes \hat{h} \circ \hat{\pmb{p}}_{r_{k}}\right)\!\!.$$
\textbf{III.} Now we pair up the contributions appearing in the previous step: the right hand side of 
\eqref{V24_1} can be rewritten as
\begin{align} 
\tag{P1}
\label{P1}
&-\sum_{B(n), r_i > 1} \delta_{k}\!\!\left(\hat{h} \circ \hat{\pmb{p}}_{r_1} \otimes \mbox{\dots} \otimes \sum_{u=1}^{r_i}\hat{h} \circ \hat{\pmb{p}}_{r_i}(\mathbb{1}_V^{\otimes  u-1} \otimes \delta_1 \otimes \mathbb{1}_V^{\otimes  r_i-u}) \otimes \mbox{\dots} \otimes \hat{h} \circ \hat{\pmb{p}}_{r_{k}}\right)\\
\tag{P2}
\label{P2}
&-\sum_{B(n), r_i > 1} \delta_{k}\!\!\left(\hat{h} \circ \hat{\pmb{p}}_{r_1}\! \otimes \!\mbox{\dots} \! \otimes \! \sum_{A(r_i)}\hat{h} \circ \hat{\pmb{p}}_k(\mathbb{1}_V^{\otimes  i-1} \otimes \hat{g}\hat{f} \circ \hat{\pmb{p}}_\ell \! \otimes \! \mathbb{1}_V^{\otimes  k - i}) \otimes \! \mbox{\dots} \! \otimes \hat{h} \circ \hat{\pmb{p}}_{r_{k}}\right)\\
\tag{P3}
\label{P3} 
&-\sum_{B(n)} \delta_{k}(\hat{h} \circ \hat{\pmb{p}}_{r_1} \otimes \mbox{\dots} \otimes \hat{h} \circ \hat{\pmb{p}}_{1} \otimes \mbox{\dots} \otimes \hat{h} \circ \hat{\pmb{p}}_{r_{k}})(\mathbb{1}_V^{\otimes  s} \otimes \hat{g}\hat{f} \circ\hat{\pmb{p}}_{r_i} \otimes \mathbb{1}_V^{\otimes  n - s - r_i}) \\
\tag{P4}
\label{P4} 
&-\sum_{B(n), r_i = 1} \delta_{k}(\hat{h} \circ \hat{\pmb{p}}_{r_1} \otimes \mbox{\dots} \otimes \hat{h}\circ \hat{\pmb{p}}_{r_i} \otimes \mbox{\dots} \otimes \hat{h} \circ \hat{\pmb{p}}_{r_{k}})(\mathbb{1}_V^{\otimes s} \otimes \delta_1 \otimes \mathbb{1}_V^{\otimes  n - s - r_i}),
\end{align}
with $s = \sum_{j < i} r_j$, and we get 
\begin{align*} 
\eqref{P1} + \eqref{P4} &= -\sum_{u=1}^{n}\hat{\pmb{p}}_n(\mathbb{1}_V^{\otimes  u-1} \otimes \delta_1 \otimes \mathbb{1}_V^{\otimes  n-u}),\\
\eqref{P2} + \eqref{P3} &=- \sum_{A(n)}\hat{\pmb{p}}_k(\mathbb{1}_V^{\otimes  i-1} \otimes \hat{g}\hat{f} \circ \hat{\pmb{p}}_\ell \otimes \mathbb{1}_V^{\otimes  k - i}).
\end{align*}
\end{proof}

\begin{remark}
Theorem~\ref{k2v18} implies that the $\pmb{p}-$kernels in \cite{Markl06} fulfill \eqref{T16g}.
\end{remark}

\subsection{\texorpdfstring{$\pmb{q}-$kernels}{q-kernels}}

\begin{lemma}\label{k2t19}
The $\pmb{q}-$kernels constitute
$A_\infty$ morphism $\pmb{\varphi} = (f, \varphi_2, \varphi_3, \mbox{\dots})$,
$\varphi_n = f \circ \pmb{q}_n$ and $\nu_n = f \circ \pmb{p}_n \circ g^{\otimes n}$,
from $(V, \partial_V, \mu_2, \mu_3, \mbox{\dots})$ to 
$(W, \partial_W, \nu_2,$\linebreak$\nu_3, \mbox{\dots})$ if and only if for all $n \geq 2$:
$$f \circ \left(\partial_V\pmb{q}_n + \sum_{u=1}^{n}(-1)^{n} \pmb{q}_n(\mathbb{1}_V^{\otimes  u-1} \otimes \partial_V \otimes \mathbb{1}_V^{\otimes  n-u}) \right.$$
$$+ \sum_{B(n)}(-1)^{\vartheta(r_1, \mbox{\dots}, r_k)}\pmb{p}_k(gf \circ \pmb{q}_{r_1}\! \otimes \mbox{\dots} \otimes gf \circ \pmb{q}_{r_k}) $$
$$\left. +\sum_{A(n)}(-1)^{i(\ell + 1) + n}\pmb{q}_{k}(\mathbb{1}_V^{\otimes  i-1} \otimes \mu_{\ell} \otimes \mathbb{1}_V^{\otimes  n-k}) - \pmb{q}_1\mu_n\!\right)\! = 0.$$
\end{lemma}
\begin{proof}
The proof easily follows from the explicit expansion of $A_\infty$ morphism 
$\pmb{\varphi} = (f, \varphi_2, \varphi_3, \mbox{\dots})$, which maps 
$(V, \partial_V, \mu_2, \mu_3, \mbox{\dots})$ 
to $(W, \partial_W,$\linebreak$ \nu_2, \nu_3, \mbox{\dots})$ for
$\varphi_n = f \circ \pmb{q}_n$ and 
$\nu_n = f \circ \pmb{p}_n \circ g^{\otimes n}$ (cf. \eqref{Antanz}).
\end{proof}

\begin{lemma}\label{k2t20}
Let the $\pmb{q}-$kernels fulfill 
\begin{align}
\label{T20}
& \partial_V\pmb{q}_n +\sum_{u=1}^{n}(-1)^{n} \pmb{q}_n(\mathbb{1}_V^{\otimes  u-1} \otimes \partial_V \otimes \mathbb{1}_V^{\otimes  n-u})  \nonumber \\
& + \!\sum_{B(n)}(-1)^{\vartheta(r_1, \mbox{\dots}, r_k)}\pmb{p}_k(gf \circ \pmb{q}_{r_1}\! \otimes \mbox{\dots} \otimes gf \circ \pmb{q}_{r_k}) \\ \nonumber
& + \!\sum_{A(n)}(-1)^{i(\ell + 1) + n}\pmb{q}_{k}(\mathbb{1}_V^{\otimes  i-1} \otimes \mu_{\ell} \otimes \mathbb{1}_V^{\otimes  n-k})  - \pmb{q}_1\mu_n= 0.
\end{align}
for all $n \geq 2$. Then we have
$$\pmb{q}_n = \sum_{C(n)} (-1)^{n+r_i + \vartheta(r_1, \mbox{\dots}, r_i)}\mu_k((\pmb{\psi}\pmb{\varphi})_{r_1} \otimes \mbox{\dots} \otimes (\pmb{\psi}\pmb{\varphi})_{r_{i-1}} \otimes h \circ \pmb{q}_{r_i} \otimes \mathbb{1}_V^{\otimes k-i}),$$
for all $n \geq 2$, where $A_\infty$ morphisms 
$\pmb{\varphi}$ and $\pmb{\psi}$ are given by $\pmb{p}-$kernels and 
$\pmb{q}-$kernels, \eqref{Antanz}. We also used the notation $C(n)$ as in \eqref{C} and 
$\vartheta(r_1, \mbox{\dots}, r_k)$ as in \eqref{theta}.
\end{lemma}
\begin{proof} Assuming \eqref{T20}, the set of $\pmb{q}-$kernels constitutes by Lemma~\ref{k2t19}
$A_\infty$ morphism $\pmb{\varphi} = (f, \varphi_2, \varphi_3, \mbox{\dots})$ from 
$(V, \partial_V, \mu_2, \mu_3, \mbox{\dots})$ to $(W, \partial_W,$\linebreak$ \nu_2, \nu_3, \mbox{\dots})$.
We also demand the set of maps $H_n = h \circ \pmb{q}_n$ gives $A_\infty$ homotopy 
$\pmb{H}= (h, H_2, H_3, \mbox{\dots})$ between $\pmb{\psi \varphi}$ and $\mathbb{1}$.
This is equivalent by Definition~\ref{ainftyhomotopy} to  
 $$\partial_VH_n - \sum_{u=1}^{n}(-1)^{n} H_n(\mathbb{1}_V^{\otimes  u-1} \otimes \partial_V \otimes \mathbb{1}_V^{\otimes  n-u}) $$
 $$+\sum_{C(n)} (-1)^{n+r_i + \vartheta(r_1, \mbox{\dots}, r_i)}\mu_k((\pmb{\psi}\pmb{\varphi})_{r_1} \otimes \mbox{\dots} \otimes (\pmb{\psi}\pmb{\varphi})_{r_{i-1}} \otimes H_{r_i} \otimes \mathbb{1}_V^{\otimes k-i}) + H_1\mu_n $$
 $$= \sum_{A(n)}(-1)^{i(\ell + 1) + n} H_k(\mathbb{1}_V^{\otimes i-1} \otimes \mu_\ell \otimes \mathbb{1}_V^{\otimes n-k}) + (\pmb{\psi}\pmb{\varphi})_{n} - (\mathbb{1})_{n}$$
for all $n \geq 2$.
 According to \eqref{composition}, we have
 $$(\pmb{\psi}\pmb{\varphi})_m = \psi_1\varphi_m + \sum_{B(m)} (-1)^{\vartheta(r_1, \mbox{\dots}, r_k)}\psi_k(\varphi_{r_1} \otimes \mbox{\dots} \otimes \varphi_{r_k}),$$
 and so we can write the composition of $A_\infty$ morphisms in terms of $\pmb{p}-$kernels and $\pmb{q}-$kernels:
\begin{equation}
\label{p,q-slozeni}
	(\pmb{\psi}\pmb{\varphi})_m = gf\circ \pmb{q}_m + \sum_{B(m)} (-1)^{\vartheta(r_1, \mbox{\dots}, r_k)}h \circ \pmb{p}_k(gf \circ \pmb{q}_{r_1} \otimes \mbox{\dots} \otimes gf \circ \pmb{q}_{r_k}).
\end{equation}
By Definition~\ref{ainftyhomotopy}, the $A_\infty$ homotopy $\pmb{H}= (h, H_2, H_3, \mbox{\dots})$
can be rewritten in terms of $\pmb{p}-$kernels and $\pmb{q}-$kernels  
(we use again $\partial_Vh = gf - \mathbb{1}_V - h\partial_V$ and $(\mathbb{1})_n=0$):
 $$ gf\circ\pmb{q}_n - \pmb{q}_n - h\partial_V\pmb{q}_n - \sum_{u=1}^{n}(-1)^{n} h \circ \pmb{q}_n(\mathbb{1}_V^{\otimes  u-1} \otimes \partial_V \otimes \mathbb{1}_V^{\otimes  n-u}) $$
 $$+\sum_{C(n)} (-1)^{n+r_i + \vartheta(r_1, \mbox{\dots}, r_i)}\mu_k((\pmb{\psi}\pmb{\varphi})_{r_1} \otimes \mbox{\dots} \otimes (\pmb{\psi}\pmb{\varphi})_{r_{i-1}} \otimes h \circ \pmb{q}_{r_i} \otimes \mathbb{1}_V^{\otimes k-i}) + h \circ \pmb{q}_1 \mu_n $$
 $$= \sum_{A(n)}(-1)^{i(\ell + 1) + n} h\circ\pmb{q}_k(\mathbb{1}_V^{\otimes i-1} \otimes \mu_\ell \otimes \mathbb{1}_V^{\otimes n-k}) + gf\circ \pmb{q}_n $$
 $$+ \sum_{B(n)} (-1)^{\vartheta(r_1, \mbox{\dots}, r_k)}h \circ \pmb{p}_k(gf \circ \pmb{q}_{r_1} \otimes \mbox{\dots} \otimes gf \circ \pmb{q}_{r_k}).$$
 We subtract from both sides of the last display 
$gf\circ \pmb{q}_n$, and by \eqref{T20} conclude
 $$- h\partial_V\pmb{q}_n - \sum_{u=1}^{n}(-1)^{n} h \circ \pmb{q}_n(\mathbb{1}_V^{\otimes  u-1} \otimes \partial_V \otimes \mathbb{1}_V^{\otimes  n-u}) + h \circ \pmb{q}_1 \mu_n=$$
 $$= \sum_{A(n)}(-1)^{i(\ell + 1) + n} h\circ\pmb{q}_k(\mathbb{1}_V^{\otimes i-1} \otimes \mu_\ell \otimes \mathbb{1}_V^{\otimes n-k}) + $$
 $$+ \sum_{B(n)} (-1)^{\vartheta(r_1, \mbox{\dots}, r_k)}h \circ \pmb{p}_k(gf \circ \pmb{q}_{r_1} \otimes \mbox{\dots} \otimes gf \circ \pmb{q}_{r_k}),$$
 which finally results in
 $$\pmb{q}_n = \sum_{C(n)} (-1)^{n+r_i + \vartheta(r_1, \mbox{\dots}, r_i)}\mu_k((\pmb{\psi}\pmb{\varphi})_{r_1} \otimes \mbox{\dots} \otimes (\pmb{\psi}\pmb{\varphi})_{r_{i-1}} \otimes h \circ \pmb{q}_{r_i} \otimes \mathbb{1}_V^{\otimes k-i}).$$
\end{proof}

\begin{remark}
The assumption \eqref{T20} is fulfilled as soon as the $\pmb{q}-$kernels give a $A_\infty$ morphism 
$\pmb{\varphi} = (f, \varphi_2, \varphi_3, \mbox{\dots})$ from $(V, \partial_V, \mu_2, \mu_3, \mbox{\dots})$ 
to $(W, \partial_W, \nu_2, \nu_3, \mbox{\dots})$ and $f$ is a monomorphism.
\end{remark}

\begin{definition}[$\pmb{q}-$kernels, \cite{Markl06}]\label{k2d21} Let $n \geq 2$ and define $\pmb{q}_1 := \mathbb{1}_V$. 
We define $\pmb{q}-$kernels inductively by
$$\pmb{q}_n = \sum_{C(n)} (-1)^{n+r_i + \vartheta(r_1, \mbox{\dots}, r_i)}\mu_k((\pmb{\psi}\pmb{\varphi})_{r_1} \otimes \mbox{\dots} \otimes (\pmb{\psi}\pmb{\varphi})_{r_{i-1}} \otimes h \circ \pmb{q}_{r_i} \otimes \mathbb{1}_V^{\otimes k-i}),$$ where $(\pmb{\psi}\pmb{\varphi})_m = gf\circ \pmb{q}_m + \sum_{B(m)} (-1)^{\vartheta(r_1, \mbox{\dots}, r_k)}h \circ \pmb{p}_k(gf \circ \pmb{q}_{r_1} \otimes \mbox{\dots} \otimes gf \circ \pmb{q}_{r_k})$ (cf., \eqref{p,q-slozeni}), 
$\pmb{p}-$kernels were introduced in \ref{k2d17}, with $C(n)$ given in \eqref{C} and $\vartheta(u_1, \mbox{\dots}, u_k)$ in \eqref{theta}.
\end{definition}

\begin{remark}
There is an explicit description of the $\pmb{q}-$kernels in terms of rooted plane trees,
but it is much more complicated when compared to the analogous description for the $\pmb{p}-$kernels.
\end{remark}

We shall now prove that the $\pmb{q}-$kernels introduced in Definition \ref{k2d21} satisfy \eqref{T20}. 
Let us consider again the suspension $\overline{T}sV$ with the induced codifferential $\delta$ such that
$\delta_1 = s \circ \partial_V \circ \omega$ and 
$\delta_n = s \circ \mu_n \circ \omega^{\otimes n}$, $n \geq 2$. 
Then $\hat{\pmb{q}}_m = s \circ \pmb{q}_m \circ \omega^{\otimes m},$ $\hat{\pmb{\psi}}_m = s \circ \psi_m \circ \omega^{\otimes m}$ and $\hat{\pmb{\varphi}}_m = s \circ \varphi_m \circ \omega^{\otimes m}$ for $m \geq 2$
($|\hat{\pmb{q}}_m| = |\hat{\pmb{\varphi}}_m| = |\hat{\pmb{\psi}}_m| = 0$), and \eqref{T20} is equivalent to
\begin{equation}
\begin{split}
\notag
\delta_1\hat{\pmb{q}}_n + \sum_{B(n)}\hat{\pmb{p}}_k(\hat{g}\hat{f} \circ \hat{\pmb{q}}_{r_1} \otimes \mbox{\dots} \otimes \hat{g}\hat{f} \circ \hat{\pmb{q}}_{r_k}) & \\ =\sum_{u=1}^{n}\hat{\pmb{q}}_n(\mathbb{1}_V^{\otimes  u-1} \otimes \delta_1 \otimes \mathbb{1}_V^{\otimes  n-u}) + \sum_{A(n)}\hat{\pmb{q}}_{k}(\mathbb{1}_V^{\otimes  i-1} \otimes &\delta_{\ell} \otimes \mathbb{1}_V^{\otimes  n-k}) + \hat{\pmb{q}}_1\delta_n.
\end{split}
\end{equation}
In the following two lemmas we prove that $\hat{\pmb{\psi}}\hat{\pmb{\varphi}}$ is an $A_\infty$ morphism.

\begin{lemma}\label{k2l22}
Let us assume \eqref{T20} is true for all $n \leq m$.
Then the $\pmb{p}-$kernels in Definition \ref{k2d17} 
and the $\pmb{q}-$kernels in Definition \ref{k2d21}  fulfill
$$\delta_1(\hat{\pmb{\psi}}\hat{\pmb{\varphi}})_m = \sum_{u=1}^{m}(\hat{\pmb{\psi}}\hat{\pmb{\varphi}})_m(\mathbb{1}_V^{\otimes  u-1} \otimes \delta_1 \otimes \mathbb{1}_V^{\otimes  m-u})$$$$ + \sum_{A(m)}(\hat{\pmb{\psi}}\hat{\pmb{\varphi}})_k(\mathbb{1}_V^{\otimes  i-1} \otimes \delta_{\ell} \otimes \mathbb{1}_V^{\otimes  n-k}) + (\hat{\pmb{\psi}}\hat{\pmb{\varphi}})_1\delta_m$$$$ -\sum_{B(m)}\hat{\pmb{p}}_k(\hat{g}\hat{f} \circ \hat{\pmb{q}}_{r_1} \otimes \mbox{\dots} \otimes \hat{g}\hat{f} \circ \hat{\pmb{q}}_{r_k})$$
for all $m \geq 2$.
\end{lemma}
\begin{proof} We shall first expand the composition of morphisms in the suspended form 
as in \eqref{p,q-slozeni}, and also use the homotopy $\hat{h}$ between $\hat{g}\hat{f}$ and 
$\mathbb{1}_V$:
$$\delta_1(\hat{\pmb{\psi}}\hat{\pmb{\varphi}})_m = \hat{g}\hat{f}\circ \delta_1 \pmb{q}_m + \sum_{B(m)}\delta_1 \hat{h} \circ \hat{\pmb{p}}_k(\hat{g}\hat{f} \circ \hat{\pmb{q}}_{r_1} \otimes \mbox{\dots} \otimes \hat{g}\hat{f} \circ \hat{\pmb{q}}_{r_k}) $$
\begin{equation}
\label{L22_1}
= \hat{g}\hat{f}\circ \delta_1 \pmb{q}_m + \sum_{B(m)}(\hat{g}\hat{f} - \mathbb{1}_V - \hat{h}\delta_1) \circ \hat{\pmb{p}}_k(\hat{g}\hat{f} \circ \hat{\pmb{q}}_{r_1} \otimes \mbox{\dots} \otimes \hat{g}\hat{f} \circ \hat{\pmb{q}}_{r_k}).	
\end{equation}
By Theorem~\ref{k2v18}
$$\sum_{B(m)}\hat{h} \circ\delta_1 \hat{\pmb{p}}_k(\hat{g}\hat{f} \circ \hat{\pmb{q}}_{r_1} \otimes \mbox{\dots} \otimes \hat{g}\hat{f} \circ \hat{\pmb{q}}_{r_k}) $$
\begin{equation}
\label{L22_2}
= -\sum_{B(m)}\sum_{u=1}^{k}\hat{h} \circ\hat{\pmb{p}}_k(\mathbb{1}_V^{\otimes  u-1} \otimes \delta_1 \otimes \mathbb{1}_V^{\otimes  k-u})(\hat{g}\hat{f} \circ \hat{\pmb{q}}_{r_1} \otimes \mbox{\dots} \otimes \hat{g}\hat{f} \circ \hat{\pmb{q}}_{r_k}) 
\end{equation}
\begin{equation}
\label{L22_3}
-\sum_{B(m)}\sum_{A'(k)}\hat{h} \circ\hat{\pmb{p}}_{k'}(\mathbb{1}_V^{\otimes  i'-1} \otimes \hat{g}\hat{f} \circ \hat{\pmb{p}}_{\ell'} \otimes \mathbb{1}_V^{\otimes  k-k'})(\hat{g}\hat{f} \circ \hat{\pmb{q}}_{r_1} \otimes \mbox{\dots} \otimes \hat{g}\hat{f} \circ \hat{\pmb{q}}_{r_k})	
\end{equation}
and as $\hat{g}, \hat{f}$ and $\hat{\pmb{q}}_m$ are of degree $0$, we have
$$\eqref{L22_2} = -\sum_{B(m)}\sum_{u=1}^{k}\hat{h} \circ\hat{\pmb{p}}_k(\hat{g}\hat{f} \circ \hat{\pmb{q}}_{r_1} \otimes \mbox{\dots} \otimes \hat{g}\hat{f} \circ \delta_1\hat{\pmb{q}}_{r_u} \otimes \mbox{\dots} \otimes \hat{g}\hat{f} \circ \hat{\pmb{q}}_{r_k}).$$
By \eqref{T20} for $n \leq m$, we expand the terms of the form $\delta_1 \hat{\pmb{q}}_\bullet$ as
$$ -\sum_{B(m)}\sum_{u=1}^{k}\hat{\pmb{p}}_k(\hat{g}\hat{f} \circ \hat{\pmb{q}}_{r_1} \otimes \mbox{\dots} \otimes \sum_{v=1}^{r_u} \hat{g}\hat{f} \circ\hat{\pmb{q}}_{r_u}(\mathbb{1}_V^{\otimes  v-1} \otimes \delta_1 \otimes \mathbb{1}_V^{\otimes  r_u -v}) \otimes \mbox{\dots} \otimes \hat{g}\hat{f} \circ \hat{\pmb{q}}_{r_k}) $$
$$-\sum_{B(m)}\sum_{u=1}^{k}\hat{\pmb{p}}_k(\hat{g}\hat{f} \circ \hat{\pmb{q}}_{r_1} \otimes \mbox{\dots} \otimes \sum_{A'(r_u)} \hat{g}\hat{f} \circ\hat{\pmb{q}}_{k'}(\mathbb{1}_V^{\otimes  i'-1} \otimes \delta_{\ell'} \otimes \mathbb{1}_V^{\otimes  r_u-k'}) \otimes \mbox{\dots} \otimes \hat{g}\hat{f} \circ \hat{\pmb{q}}_{r_k}) $$
$$-\sum_{B(m)}\sum_{u=1}^{k}\hat{\pmb{p}}_k(\hat{g}\hat{f} \circ \hat{\pmb{q}}_{r_1} \otimes \mbox{\dots} \otimes \hat{g}\hat{f} \circ\hat{\pmb{q}}_{1}\delta_{r_u} \otimes \mbox{\dots} \otimes \hat{g}\hat{f} \circ \hat{\pmb{q}}_{r_k}) $$
$$+\sum_{B(m)}\sum_{u=1}^{k}\hat{\pmb{p}}_k(\hat{g}\hat{f} \circ \hat{\pmb{q}}_{r_1} \otimes \mbox{\dots} \otimes \sum_{B'(r_u)} \hat{g}\hat{f} \circ\hat{\pmb{p}}_{k'}(\hat{g}\hat{f} \circ \hat{\pmb{q}}_{r'_1} \otimes \mbox{\dots} \otimes \hat{g}\hat{f} \circ \hat{\pmb{q}}_{r'_{k'}}) \otimes \mbox{\dots} \otimes \hat{g}\hat{f} \circ \hat{\pmb{q}}_{r_k}).$$
In the second contribution \eqref{L22_3}, which equals to
$$-\sum_{B(m)}\sum_{A'(k)}\hat{\pmb{p}}_{k'}(\hat{g}\hat{f} \circ \hat{\pmb{q}}_{r_1} \otimes \mbox{\dots} \otimes$$
$$\otimes \hat{g}\hat{f} \circ \hat{\pmb{p}}_{\ell'}(\hat{g}\hat{f} \circ \hat{\pmb{q}}_{r_{i}} \otimes \mbox{\dots} \otimes \hat{g}\hat{f} \circ \hat{\pmb{q}}_{r_{i + \ell'-1}}) \otimes  \mbox{\dots} \otimes \hat{g}\hat{f} \circ \hat{\pmb{q}}_{r_{k'}}),$$
we sum over all inner positions of $\hat{\pmb{p}}_{k'}(\hat{g}\hat{f} \circ \hat{\pmb{q}}_{r_1} \otimes \mbox{\dots} \otimes \bullet \otimes  \mbox{\dots} \otimes \hat{g}\hat{f} \circ \hat{\pmb{q}}_{r_{k'}})$ and get 
$$\eqref{L22_3} = -\sum_{B(m)}\sum_{u=1}^{k}\hat{\pmb{p}}_k(\hat{g}\hat{f} \circ \hat{\pmb{q}}_{r_1} \otimes \mbox{\dots} \otimes$$
$$\otimes \sum_{B'(r_u)} \hat{g}\hat{f} \circ\hat{\pmb{p}}_{k'}(\hat{g}\hat{f} \circ \hat{\pmb{q}}_{r'_1} \otimes \mbox{\dots} \otimes \hat{g}\hat{f} \circ \hat{\pmb{q}}_{r'_{k'}}) \otimes \mbox{\dots} \otimes \hat{g}\hat{f} \circ \hat{\pmb{q}}_{r_k}).$$
Up to a sign, this is the same expression as the expression on the fourth line of the expansion \eqref{L22_2}.
We substitute into \eqref{L22_1} for $\sum_{B(m)}\hat{h} \circ\delta_1 \hat{\pmb{p}}_k(\hat{g}\hat{f} \circ \hat{\pmb{q}}_{r_1} \otimes \mbox{\dots} \otimes \hat{g}\hat{f} \circ \hat{\pmb{q}}_{r_k})$ 
the combination $\eqref{L22_2} + \eqref{L22_3}$ and also substitute for $\delta_1\hat{\pmb{q}}_m$ according to \eqref{T20}: 
$$\delta_1(\hat{\pmb{\psi}}\hat{\pmb{\varphi}})_m = -\sum_{B(m)}\hat{g}\hat{f} \circ \hat{\pmb{p}}_k(\hat{g}\hat{f} \circ \hat{\pmb{q}}_{r_1} \otimes \mbox{\dots} \otimes \hat{g}\hat{f} \circ \hat{\pmb{q}}_{r_k}) $$$$+ \sum_{u=1}^{m}\hat{g}\hat{f} \circ\hat{\pmb{q}}_m(\mathbb{1}_V^{\otimes  u-1} \otimes \delta_1 \otimes \mathbb{1}_V^{\otimes  m-u}) $$
$$+ \sum_{A(m)}\hat{g}\hat{f} \circ\hat{\pmb{q}}_{k}(\mathbb{1}_V^{\otimes  i-1} \otimes \delta_{\ell} \otimes \mathbb{1}_V^{\otimes  n-k}) + \hat{g}\hat{f} \circ \hat{\pmb{q}}_1\delta_m $$
$$+  \sum_{B(m)}(\hat{g}\hat{f} - \mathbb{1}_V)\circ \hat{\pmb{p}}_k(\hat{g}\hat{f} \circ \hat{\pmb{q}}_{r_1} \otimes \mbox{\dots} \otimes \hat{g}\hat{f} \circ \hat{\pmb{q}}_{r_k}) $$
$$+ \sum_{B(m)}\sum_{u=1}^{k}\hat{\pmb{p}}_k(\hat{g}\hat{f} \circ \hat{\pmb{q}}_{r_1} \otimes \mbox{\dots} \otimes \sum_{v=1}^{r_u} \hat{g}\hat{f} \circ\hat{\pmb{q}}_{r_u}(\mathbb{1}_V^{\otimes  v-1} \otimes \delta_1 \otimes \mathbb{1}_V^{\otimes  r_u -v}) \otimes \mbox{\dots} \otimes \hat{g}\hat{f} \circ \hat{\pmb{q}}_{r_k}) $$
$$+ \sum_{B(m)}\sum_{u=1}^{k}\hat{\pmb{p}}_k(\hat{g}\hat{f} \circ \hat{\pmb{q}}_{r_1} \otimes \mbox{\dots} \otimes \sum_{A'(r_u)} \hat{g}\hat{f} \circ\hat{\pmb{q}}_{k'}(\mathbb{1}_V^{\otimes  i'-1} \otimes \delta_{\ell'} \otimes \mathbb{1}_V^{\otimes  r_u-k'}) \otimes \mbox{\dots} \otimes \hat{g}\hat{f} \circ \hat{\pmb{q}}_{r_k}) $$
$$ + \sum_{B(m)}\sum_{u=1}^{k}\hat{\pmb{p}}_k(\hat{g}\hat{f} \circ \hat{\pmb{q}}_{r_1} \otimes \mbox{\dots} \otimes \hat{g}\hat{f} \circ\hat{\pmb{q}}_{1}\delta_{r_u} \otimes \mbox{\dots} \otimes \hat{g}\hat{f} \circ \hat{\pmb{q}}_{r_k}).$$
This completes the proof.
\end{proof}

\begin{lemma}\label{k2l23}
The $\pmb{p}-$kernels in Definition \ref{k2d17} 
and the $\pmb{q}-$kernels in Definition \ref{k2d21} fulfill
$$\sum_{B(m)}\delta_k((\hat{\pmb{\psi}}\hat{\pmb{\varphi}})_{r_1} \otimes \mbox{\dots} \otimes (\hat{\pmb{\psi}}\hat{\pmb{\varphi}})_{r_k}) = \sum_{B(m)}\hat{\pmb{p}}_k(\hat{g}\hat{f} \circ \hat{\pmb{q}}_{r_1} \otimes \mbox{\dots} \otimes \hat{g}\hat{f} \circ \hat{\pmb{q}}_{r_k})$$
for all $m \geq 2$.
\end{lemma}
\begin{proof} By \eqref{p-kernel}, we have
$$\sum_{B(m)}\hat{\pmb{p}}_k(\hat{g}\hat{f} \circ \hat{\pmb{q}}_{r_1} \otimes \mbox{\dots} \otimes \hat{g}\hat{f} \circ \hat{\pmb{q}}_{r_k})$$
$$= \sum_{B(m)} \sum_{B'(k)} \delta_{k'}(h \circ \hat{\pmb{p}}_{r'_1} \otimes \mbox{\dots} \otimes h \circ \hat{\pmb{p}}_{r'_{k'}})(\hat{g}\hat{f} \circ \hat{\pmb{q}}_{r_1} \otimes \mbox{\dots} \otimes \hat{g}\hat{f} \circ \hat{\pmb{q}}_{r_k}).$$ 
Taking into account that $\hat{g}, \hat{f}$ and $\hat{\pmb{q}}_m$ are of degree $0$,
the last display equals to
$$\sum_{B(m)} \sum_{B'(k)} \delta_{k'}(h \circ \hat{\pmb{p}}_{r'_1}(\hat{g}\hat{f} \circ \hat{\pmb{q}}_{r_1} \otimes \mbox{\dots} \otimes \hat{g}\hat{f} \circ \hat{\pmb{q}}_{r_{r'_1}}) \otimes \mbox{\dots}$$
$$\mbox{\dots} \otimes h \circ \hat{\pmb{p}}_{r'_{k'}}(\hat{g}\hat{f} \circ \hat{\pmb{q}}_{r_{{r_{k'-1}} + 1}} \otimes \mbox{\dots} \otimes \hat{g}\hat{f} \circ \hat{\pmb{q}}_{r_{k'}}))$$
and the summation over the terms $\delta_k(\star_{r_1} \otimes \mbox{\dots} \otimes \star_{r_k})$  in all possible indices ($\star_j$ denoting a map $V^{\otimes j} \to V$) gives 
$$\sum_{B(m)}\delta_k\left[\left(\hat{g}\hat{f} \circ \hat{\pmb{q}}_{r_1} + \sum_{B'(r_1)}\hat{\pmb{p}}_{k'}(\hat{g}\hat{f} \circ \hat{\pmb{q}}_{r'_1} \otimes \mbox{\dots} \otimes \hat{g}\hat{f} \circ \hat{\pmb{q}}_{r_{k'}})\right) \otimes \mbox{\dots}\right.$$
$$ \left. \mbox{\dots} \otimes \left(\hat{g}\hat{f} \circ \hat{\pmb{q}}_{r_{k}} +\sum_{B'(r_k)}\hat{\pmb{p}}_{k'}(\hat{g}\hat{f} \circ \hat{\pmb{q}}_{r'_1} \otimes \mbox{\dots}
 \otimes \hat{g}\hat{f} \circ \hat{\pmb{q}}_{r_{k'}})\right)\right].$$
However this is already \eqref{p,q-slozeni} composed with the suspension, and the proof is 
complete. 
\end{proof}

Because the formula for the $\pmb{q}-$kernels in Lemma~\ref{k2t20} was based on 
the assumption \eqref{T20}, we have to prove that it is fulfilled by the 
$\pmb{q}-$kernels in Definition \ref{k2d21}. 

\begin{theorem}\label{k2v24}
The $\pmb{p}-$kernels in Definition \ref{k2d17} 
and the $\pmb{q}-$kernels in Definition \ref{k2d21} 
fulfill \eqref{T20}, i.e.
\begin{equation}
\notag
\begin{split}
\partial_V\pmb{q}_n + \sum_{u=1}^{n}(-1)^{n} \pmb{q}_n(\mathbb{1}_V^{\otimes  u-1} \otimes \partial_V \otimes \mathbb{1}_V^{\otimes  n-u}) &\\ +\!\sum_{B(n)}(-1)^{\vartheta(r_1, \mbox{\dots}, r_k)}\pmb{p}_k(gf \circ \pmb{q}_{r_1}\! \otimes \mbox{\dots} \otimes gf \circ \pmb{q}_{r_k}) &\\ +\!\sum_{A(n)}(-1)^{i(\ell + 1) + n}\pmb{q}_{k}(\mathbb{1}_V^{\otimes  i-1} \otimes \mu_{\ell} \otimes \mathbb{1}_V^{\otimes  n-k}) - \pmb{q}_1\mu_n&= 0.
\end{split}
\end{equation}
 This means that the objects introduced in \eqref{Antanz} solve the 
problem of the transfer of $A_\infty$ structure.
\end{theorem}
\begin{proof} We shall prove an equivalent assertion: 
\begin{equation}
\begin{split}
\notag
\delta_1\hat{\pmb{q}}_n + \sum_{B(n)}\hat{\pmb{p}}_k(\hat{g}\hat{f} \circ \hat{\pmb{q}}_{r_1} \otimes \mbox{\dots} \otimes \hat{g}\hat{f} \circ \hat{\pmb{q}}_{r_k}) &=\\
=\sum_{u=1}^{n}\hat{\pmb{q}}_n(\mathbb{1}_V^{\otimes  u-1} \otimes \delta_1 \otimes \mathbb{1}_V^{\otimes  n-u}) + \sum_{A(n)}\hat{\pmb{q}}_{k}(\mathbb{1}_V^{\otimes  i-1} \otimes &\delta_{\ell} \otimes \mathbb{1}_V^{\otimes  n-k}) + \hat{\pmb{q}}_1\delta_n,
\end{split}
\end{equation}
with suspended $\pmb{q}-$kernels given by Definition~\ref{k2d21}:
\begin{equation}
	\label{q-kernel}
	\hat{\pmb{q}}_n = \sum_{C(n)}\delta_k((\hat{\pmb{\psi}}\hat{\pmb{\varphi}})_{r_1} \otimes \mbox{\dots} \otimes (\hat{\pmb{\psi}}\hat{\pmb{\varphi}})_{r_{i-1}} \otimes \hat{h} \circ \hat{\pmb{q}}_{r_i} \otimes \mathbb{1}_V^{\otimes k-i}).
\end{equation}
The proof goes by induction on $n$: for $n = 2$, we have by  
\eqref{T6} (for $n=2$) and \eqref{q-kernel}:
$$\delta_1\hat{\pmb{q}}_2 = \delta_1(\delta_2(\hat{g}\hat{f} \otimes \hat{h}) + \delta_2(\hat{h} \otimes \mathbb{1}_V)) $$
$$= -\delta_2(\delta_1\otimes \mathbb{1}_V)(\hat{g}\hat{f} \otimes \hat{h}) -\delta_2(\mathbb{1}_V \otimes \delta_1)(\hat{g}\hat{f} \otimes \hat{h}) $$
$$-\delta_2(\delta_1\otimes \mathbb{1}_V)(\hat{h} \otimes \mathbb{1}_V) -\delta_2(\mathbb{1}_V \otimes \delta_1)(\hat{h} \otimes \mathbb{1}_V).~$$
By the Koszul sign convention
\begin{align*}
	\delta_2(\delta_1\otimes \mathbb{1}_V)(\hat{g}\hat{f} \otimes \hat{h}) &= (-1)^{|\delta_1||\hat{h}|}\delta_2(\hat{g}\hat{f} \otimes \hat{h})(\delta_1\otimes \mathbb{1}_V),\\
	\delta_2(\mathbb{1}_V \otimes \delta_1)(\hat{g}\hat{f} \otimes \hat{h}) &= \delta_2(\hat{g}\hat{f} \otimes \hat{g}\hat{f} - \mathbb{1}_V -\hat{h}\delta_1) =\\
	 &=  \delta_2(\hat{g}\hat{f} \otimes \hat{g}\hat{f}) -  \delta_2(\hat{g}\hat{f} \otimes \mathbb{1}_V) - \delta_2(\hat{g}\hat{f} \otimes \hat{h})(\mathbb{1}_V \otimes \delta_1),\\
	\delta_2(\delta_1\otimes \mathbb{1}_V)(\hat{h} \otimes \mathbb{1}_V) &= \delta_2(\hat{g}\hat{f} - \mathbb{1}_V -\hat{h}\delta_1 \otimes \mathbb{1}_V) =\\
	&= \delta_2(\hat{g}\hat{f} \otimes \mathbb{1}_V) - \delta_2(\mathbb{1}_V \otimes \mathbb{1}_V) - \delta_2(\hat{h} \otimes \mathbb{1}_V)(\delta_1 \otimes \mathbb{1}_V),\\
	\delta_2(\mathbb{1}_V \otimes \delta_1)(\hat{h} \otimes \mathbb{1}_V) &= (-1)^{|\delta_1||\hat{h}|}\delta_2(\hat{h} \otimes \mathbb{1}_V)(\mathbb{1}_V \otimes \delta_1),
\end{align*}
where $(-1)^{|\delta_1||\hat{h}|} = -1$ is a consequence of $|\hat{h}| = |\delta_1| = 1$,
and so
$$ \delta_1\hat{\pmb{q}}_2 = \delta_2(\hat{g}\hat{f} \otimes \hat{h})(\delta_1\otimes \mathbb{1}_V) + \delta_2(\hat{h} \otimes \mathbb{1}_V)(\delta_1 \otimes \mathbb{1}_V) + \delta_2(\hat{g}\hat{f} \otimes \hat{h})(\mathbb{1}_V \otimes \delta_1) $$
$$+ \delta_2(\hat{h} \otimes \mathbb{1}_V)(\mathbb{1}_V \otimes \delta_1) + \delta_2(\mathbb{1}_V \otimes \mathbb{1}_V) - \delta_2(\hat{g}\hat{f} \otimes \hat{g}\hat{f}) $$
$$=\hat{\pmb{q}}_2(\delta_1 \otimes \mathbb{1}_V) + \hat{\pmb{q}}_2(\mathbb{1}_V \otimes \delta_1) + \hat{\pmb{q}}_1\delta_2 - \hat{\pmb{p}}_2(\hat{g}\hat{f} \otimes \hat{g}\hat{f}).$$
The induction step is divided into three steps:\\
\textbf{I.} We first expand the term $\delta_1 \hat{\pmb{q}}_n$: by \eqref{q-kernel} 
$$\delta_1\hat{\pmb{q}}_n = \sum_{C(n)}\delta_1\delta_k((\hat{\pmb{\psi}}\hat{\pmb{\varphi}})_{r_1} \otimes \mbox{\dots} \otimes (\hat{\pmb{\psi}}\hat{\pmb{\varphi}})_{r_{i-1}} \otimes \hat{h} \circ \hat{\pmb{q}}_{r_i} \otimes \mathbb{1}_V^{\otimes k-i}) $$
$$= -\sum_{C(n)}\left(\sum_{u = 1}^{k} \delta_k \left(\mathbb{1}_V^{\otimes  u-1} \otimes \delta_1 \otimes \mathbb{1}_V^{\otimes  k-u}\right)\right)((\hat{\pmb{\psi}}\hat{\pmb{\varphi}})_{r_1} \otimes \mbox{\dots} \otimes (\hat{\pmb{\psi}}\hat{\pmb{\varphi}})_{r_{i-1}} \otimes \hat{h} \circ \hat{\pmb{q}}_{r_i} \otimes \mathbb{1}_V^{\otimes k-i}) $$
$$- \sum_{C(n)}\left(\sum_{A'(k)} \delta_{k'}\left(\mathbb{1}_V^{\otimes  i'-1} \otimes \delta_{\ell'} \otimes \mathbb{1}_V^{\otimes  k' - i'}\right) \right)((\hat{\pmb{\psi}}\hat{\pmb{\varphi}})_{r_1} \otimes \mbox{\dots} \otimes (\hat{\pmb{\psi}}\hat{\pmb{\varphi}})_{r_{i-1}} \otimes \hat{h} \circ \hat{\pmb{q}}_{r_i} \otimes \mathbb{1}_V^{\otimes k-i}).$$
The first summation can be rewritten as
$$
-\sum_{C(n)}\left(\sum_{u = 1}^{k} \delta_k \left(\mathbb{1}_V^{\otimes  u-1} \otimes \delta_1 \otimes \mathbb{1}_V^{\otimes  k-u}\right)\right)((\hat{\pmb{\psi}}\hat{\pmb{\varphi}})_{r_1} \otimes \mbox{\dots} \otimes (\hat{\pmb{\psi}}\hat{\pmb{\varphi}})_{r_{i-1}} \otimes \hat{h} \circ \hat{\pmb{q}}_{r_i} \otimes \mathbb{1}_V^{\otimes k-i}) $$ 
\begin{align}
\tag{Q1.1}
\label{Q1.1}
&=-\sum_{C(n)}\sum_{u = 1}^{i-1} \delta_k((\hat{\pmb{\psi}}\hat{\pmb{\varphi}})_{r_1} \otimes \mbox{\dots} \otimes \delta_1(\hat{\pmb{\psi}}\hat{\pmb{\varphi}})_{r_{u}} \otimes \mbox{\dots} \otimes (\hat{\pmb{\psi}}\hat{\pmb{\varphi}})_{r_{i-1}} \otimes \hat{h} \circ \hat{\pmb{q}}_{r_i} \otimes \mathbb{1}_V^{\otimes k-i}) \\
\tag{Q1.2}
\label{Q1.2} 
&-\sum_{C(n)}\delta_k((\hat{\pmb{\psi}}\hat{\pmb{\varphi}})_{r_1} \otimes \mbox{\dots} \otimes (\hat{\pmb{\psi}}\hat{\pmb{\varphi}})_{r_{i-1}} \otimes \delta_1\hat{h} \circ \hat{\pmb{q}}_{r_i} \otimes \mathbb{1}_V^{\otimes k-i}) \\
\tag{Q1.3}
\label{Q1.3}
&-(-1)^{|\delta_1||\hat{h}|}\! \sum_{C(n)}\sum_{u=i+1}^{k}\!\!\delta_k((\hat{\pmb{\psi}}\hat{\pmb{\varphi}})_{r_1}\!\! \otimes \! \mbox{\dots} \! \otimes \! (\hat{\pmb{\psi}}\hat{\pmb{\varphi}})_{r_{i-1}}\!\! \otimes \hat{h} \circ \hat{\pmb{q}}_{r_i}\! \otimes \mathbb{1}_V^{\otimes  u-1}\! \otimes \delta_1 \! \otimes \mathbb{1}_V^{\otimes  k-u-i}),
\end{align} \\
while the second as
$$ -\sum_{C(n)}\left(\sum_{A'(k)} \delta_{k'}\left(\mathbb{1}_V^{\otimes  i'-1} \otimes \delta_{\ell'} \otimes \mathbb{1}_V^{\otimes  k' - i'}\right) \right) \circ$$
$$\circ ((\hat{\pmb{\psi}}\hat{\pmb{\varphi}})_{r_1} \otimes \mbox{\dots} \otimes (\hat{\pmb{\psi}}\hat{\pmb{\varphi}})_{r_{i-1}} \otimes \hat{h} \circ \hat{\pmb{q}}_{r_i} \otimes \mathbb{1}_V^{\otimes k-i})$$ 
\begin{align}
\tag{Q2.1}
\label{Q2.1}
&=- \sum_{C(n)}\delta_k((\hat{\pmb{\psi}}\hat{\pmb{\varphi}})_{r_1} \otimes \mbox{\dots} \otimes \delta_{\star}((\hat{\pmb{\psi}}\hat{\pmb{\varphi}})_\star \otimes \mbox{\dots} \otimes (\hat{\pmb{\psi}}\hat{\pmb{\varphi}})_\star ) \otimes \mbox{\dots}\\
 \notag
 &~~~~~~~~~~~~~~~~~~~~~~~~~~~~~~~~~\mbox{\dots} \otimes (\hat{\pmb{\psi}}\hat{\pmb{\varphi}})_{r_{i-1}} \otimes \hat{h} \circ \hat{\pmb{q}}_{r_i} \otimes \mathbb{1}_V^{\otimes k-i})\\
\tag{Q2.2}
\label{Q2.2}
&- \sum_{C(n)}\delta_k((\hat{\pmb{\psi}}\hat{\pmb{\varphi}})_{r_1} \otimes \mbox{\dots} \otimes \delta_{\star}((\hat{\pmb{\psi}}\hat{\pmb{\varphi}})_\star \otimes \mbox{\dots} \otimes  (\hat{\pmb{\psi}}\hat{\pmb{\varphi}})_{r_{i-1}} \otimes \hat{h} \circ \hat{\pmb{q}}_{r_i} \otimes \mathbb{1}_V^{\otimes \star}) \otimes \mathbb{1}_V^{\otimes k-\star}\\
\tag{Q2.3}
\label{Q2.3}
&-(-1)^{|\delta_{\ell'}||\hat{h}|}\!
\sum_{C(n)}\sum_{A'(k-i)}\!\!\delta_{k'+ i}((\hat{\pmb{\psi}}\hat{\pmb{\varphi}})_{r_1}\!\! \otimes\! \mbox{\dots}\! \otimes\! (\hat{\pmb{\psi}}\hat{\pmb{\varphi}})_{r_{i-1}}\!\! \otimes \! \hat{h} \circ \hat{\pmb{q}}_{r_i} \!\! \otimes \mathbb{1}_V^{\otimes  i'-1} \! \otimes \delta_{\ell'} \! \otimes \mathbb{1}_V^{\otimes  k' - i'}).
\end{align}
The summation over all indices in \eqref{Q2.1} terms of the form 
$\delta_k((\hat{\pmb{\psi}}\hat{\pmb{\varphi}})_{r_1} \otimes \mbox{\dots} \otimes \star \otimes \mbox{\dots} \otimes (\hat{\pmb{\psi}}\hat{\pmb{\varphi}})_{r_{i-1}} \otimes \hat{h} \circ \hat{\pmb{q}}_{r_i} \otimes \mathbb{1}_V^{\otimes k-i})$ leads to
$$\eqref{Q2.1} = -\sum_{C(n)} \sum_{u = 1}^{i-1} \delta_k((\hat{\pmb{\psi}}\hat{\pmb{\varphi}})_{r_1} \otimes \mbox{\dots} \otimes \sum_{B'(r_u)}\delta_{k'}((\hat{\pmb{\psi}}\hat{\pmb{\varphi}})_{r'_1} \otimes \mbox{\dots} \otimes (\hat{\pmb{\psi}}\hat{\pmb{\varphi}})_{r_{k'}}) \otimes \mbox{\dots}$$
$$ \mbox{\dots} \otimes (\hat{\pmb{\psi}}\hat{\pmb{\varphi}})_{r_{i-1}} \otimes \hat{h} \circ \hat{\pmb{q}}_{r_i} \otimes \mathbb{1}_V^{\otimes k-i}).$$
Analogously, the summation over all indices in \eqref{Q2.2} terms of the form 
$\delta_k((\hat{\pmb{\psi}}\hat{\pmb{\varphi}})_{r_1} \otimes \mbox{\dots} \otimes \star \otimes \mathbb{1}_V^{\otimes k-j})$ gives
$$\eqref{Q2.2} = -\sum_{C(n)}\sum_{r_i>1} \delta_k((\hat{\pmb{\psi}}\hat{\pmb{\varphi}})_{r_1} \otimes \mbox{\dots} \otimes (\hat{\pmb{\psi}}\hat{\pmb{\varphi}})_{r_{i-1}} \otimes \hat{\pmb{q}}_{r_i} \otimes \mathbb{1}_V^{\otimes k-i}).$$

\noindent \textbf{II.}
By Lemma~\ref{k2l23}:
$$\eqref{Q2.1} = -\sum_{C(n)} \sum_{u = 1}^{i-1} \delta_k((\hat{\pmb{\psi}}\hat{\pmb{\varphi}})_{r_1} \otimes \mbox{\dots} \otimes \sum_{B'(r_u)}\hat{\pmb{p}}_{k'}(\hat{g}\hat{f} \circ \hat{\pmb{q}}_{r'_1} \otimes \mbox{\dots} \otimes \hat{g}\hat{f} \circ \hat{\pmb{q}}_{r_{k'}}) \otimes \mbox{\dots}$$
$$ \mbox{\dots} \otimes (\hat{\pmb{\psi}}\hat{\pmb{\varphi}})_{r_{i-1}} \otimes \hat{h} \circ \hat{\pmb{q}}_{r_i} \otimes \mathbb{1}_V^{\otimes k-i}),$$
 and Lemma~\ref{k2l22} for $(\hat{\pmb{\psi}}\hat{\pmb{\varphi}})_m$
(Definition~\ref{k2d21} and definition of $C(n)$ in \eqref{C} imply 
that $m$ is strictly less than $n$, so that assumptions of 
Lemma~\ref{k2l22} are fulfilled by our induction hypothesis) gives
$$\eqref{Q1.1} + \eqref{Q2.1} =$$
\begin{equation}
\tag{Q1.1 + 2.1a}
\label{Q1.1 + 2.1a}
\sum_{C(n)}\sum_{u = 1}^{i-1}\sum_{r_u = 1} \delta_k((\hat{\pmb{\psi}}\hat{\pmb{\varphi}})_{r_1} \otimes \mbox{\dots} \otimes (-1)^{|\hat{h}||\delta_1|}(\hat{\pmb{\psi}}\hat{\pmb{\varphi}})_{r_{u}}\delta_1 \otimes \mbox{\dots} \otimes (\hat{\pmb{\psi}}\hat{\pmb{\varphi}})_{r_{i-1}} \otimes \hat{h} \circ \hat{\pmb{q}}_{r_i} \otimes \mathbb{1}_V^{\otimes k-i})
\end{equation}
\begin{equation}
\notag 
+ \sum_{C(n)}\sum_{u = 1}^{i-1}\sum_{r_u > 1} \delta_k((\hat{\pmb{\psi}}\hat{\pmb{\varphi}})_{r_1} \otimes \mbox{\dots} \otimes \sum_{u=1}^{r_i}(-1)^{|\hat{h}||\delta_1|}(\hat{\pmb{\psi}}\hat{\pmb{\varphi}})_{r_i}(\mathbb{1}_V^{\otimes  u-1} \otimes \delta_1 \otimes \mathbb{1}_V^{\otimes  r_i-u}) \otimes \mbox{\dots}
\end{equation}
\begin{equation}\tag{Q1.1 + 2.1b}\label{Q1.1 + 2.1b}\mbox{\dots} \otimes (\hat{\pmb{\psi}}\hat{\pmb{\varphi}})_{r_{i-1}} \otimes \hat{h} \circ \hat{\pmb{q}}_{r_i} \otimes \mathbb{1}_V^{\otimes k-i})
\end{equation} 
\begin{equation} \notag+ \sum_{C(n)}\sum_{u = 1}^{i-1}\sum_{r_u > 1} \delta_k((\hat{\pmb{\psi}}\hat{\pmb{\varphi}})_{r_1} \otimes \mbox{\dots} \otimes \sum_{A'(r_u)}(-1)^{|\hat{h}||\delta_{\ell'}|}(\hat{\pmb{\psi}}\hat{\pmb{\varphi}})_{k'}(\mathbb{1}_V^{\otimes  i'-1} \otimes \delta_{\ell'} \otimes \mathbb{1}_V^{\otimes  r_u-k'}) \otimes \mbox{\dots}
\end{equation}
\begin{equation}\tag{Q1.1 + 2.1c}\label{Q1.1 + 2.1c}\mbox{\dots} \otimes (\hat{\pmb{\psi}}\hat{\pmb{\varphi}})_{r_{i-1}} \otimes \hat{h} \circ \hat{\pmb{q}}_{r_i} \otimes \mathbb{1}_V^{\otimes k-i})
\end{equation}
\begin{equation}\tag{Q1.1 + 2.1d}\label{Q1.1 + 2.1d}+ \sum_{C(n)}\sum_{u = 1}^{i-1}\sum_{r_u > 1} \delta_k((\hat{\pmb{\psi}}\hat{\pmb{\varphi}})_{r_1} \otimes \mbox{\dots} \otimes (-1)^{|\hat{h}||\delta_{r_u}|}(\hat{\pmb{\psi}}\hat{\pmb{\varphi}})_{1}\delta_{r_u} \otimes \mbox{\dots} \otimes (\hat{\pmb{\psi}}\hat{\pmb{\varphi}})_{r_{i-1}} \otimes \hat{h} \circ \hat{\pmb{q}}_{r_i} \otimes \mathbb{1}_V^{\otimes k-i})
\end{equation}
\begin{equation}\notag+\sum_{C(n)}\sum_{u = 1}^{i-1}\sum_{r_u > 1} \delta_k((\hat{\pmb{\psi}}\hat{\pmb{\varphi}})_{r_1} \otimes \mbox{\dots} \otimes (-1)^{|\hat{h}||\delta_{r_u}|}(\hat{\pmb{\psi}}\hat{\pmb{\varphi}})_{1}\delta_{r_u} \otimes \mbox{\dots} \otimes (\hat{\pmb{\psi}}\hat{\pmb{\varphi}})_{r_{i-1}} \otimes \hat{h} \circ \hat{\pmb{q}}_{r_i} \otimes \mathbb{1}_V^{\otimes k-i})
\end{equation}
\begin{equation}\tag{\ref{Q2.1}}-\sum_{C(n)}\sum_{u = 1}^{i-1}\sum_{r_u > 1} \delta_k((\hat{\pmb{\psi}}\hat{\pmb{\varphi}})_{r_1} \otimes \mbox{\dots} \otimes (-1)^{|\hat{h}||\delta_{r_u}|}(\hat{\pmb{\psi}}\hat{\pmb{\varphi}})_{1}\delta_{r_u} \otimes \mbox{\dots} \otimes (\hat{\pmb{\psi}}\hat{\pmb{\varphi}})_{r_{i-1}} \otimes \hat{h} \circ \hat{\pmb{q}}_{r_i} \otimes \mathbb{1}_V^{\otimes k-i}),
\end{equation}
where the first five terms come from \eqref{Q1.2} by application of Lemma~\ref{k2l22}, and the fifth one cancels out when combined with \eqref{Q2.1}. Recall that we have $|\delta_\ell| = -1$ for all 
$\ell$, and so $(-1)^{|\hat{h}||\delta_\ell|} = -1$ as well as
$$\eqref{Q1.2} + \eqref{Q2.2} = \sum_{C(n), r_i > 1} \delta_k((\hat{\pmb{\psi}}\hat{\pmb{\varphi}})_{r_1} \otimes \mbox{\dots} \otimes (\hat{\pmb{\psi}}\hat{\pmb{\varphi}})_{r_{i-1}} \otimes (-\delta_1\hat{h} - \mathbb{1}_V)\hat{\pmb{q}}_{r_i} \otimes \mathbb{1}_V^{\otimes k-i})$$
$$+\sum_{C(n), r_i = 1} \delta_k((\hat{\pmb{\psi}}\hat{\pmb{\varphi}})_{r_1} \otimes \mbox{\dots} \otimes (\hat{\pmb{\psi}}\hat{\pmb{\varphi}})_{r_{i-1}} \otimes -\delta_1\hat{h}\circ \hat{\pmb{q}}_{r_i} \otimes \mathbb{1}_V^{\otimes k-i})$$
$$ = \sum_{C(n), r_i > 1} \delta_k((\hat{\pmb{\psi}}\hat{\pmb{\varphi}})_{r_1} \otimes \mbox{\dots} \otimes (\hat{\pmb{\psi}}\hat{\pmb{\varphi}})_{r_{i-1}} \otimes (\hat{h}\delta_1 - \hat{g}\hat{f})\hat{\pmb{q}}_{r_i} \otimes \mathbb{1}_V^{\otimes k-i}) $$
$$+ \sum_{C(n), r_i = 1} \delta_k((\hat{\pmb{\psi}}\hat{\pmb{\varphi}})_{r_1} \otimes \mbox{\dots} \otimes (\hat{\pmb{\psi}}\hat{\pmb{\varphi}})_{r_{i-1}} \otimes (\hat{h}\delta_1 - \hat{g}\hat{f} + \mathbb{1}_V)\hat{\pmb{q}}_{r_i} \otimes \mathbb{1}_V^{\otimes k-i}).$$
Thanks to the induction hypothesis we substitute for $\delta_1\hat{\pmb{q}}_\star$ and the last display turns into
$$ -\sum_{C(n), r_i > 1} \delta_k((\hat{\pmb{\psi}}\hat{\pmb{\varphi}})_{r_1} \otimes \mbox{\dots} \otimes (\hat{\pmb{\psi}}\hat{\pmb{\varphi}})_{r_{i-1}} \otimes$$
$$ \otimes \sum_{B'(r_i)}\hat{h} \circ \hat{\pmb{p}}_{k'}(\hat{g}\hat{f} \circ \hat{\pmb{q}}_{r'_1} \otimes \mbox{\dots} \otimes \hat{g}\hat{f} \circ \hat{\pmb{q}}_{r'_{k'}})\otimes \mathbb{1}_V^{\otimes k-i}) $$
$$-\sum_{C(n), r_i > 1} \delta_k((\hat{\pmb{\psi}}\hat{\pmb{\varphi}})_{r_1} \otimes \mbox{\dots} \otimes (\hat{\pmb{\psi}}\hat{\pmb{\varphi}})_{r_{i-1}} \otimes \hat{g}\hat{f}\circ \hat{\pmb{q}}_{r_i}\otimes \mathbb{1}_V^{\otimes k-i}) $$
\begin{equation}
\tag{Q1.2 + 2.2a}
\label{Q1.2 + 2.2a}
+ \sum_{C(n), r_i > 1} \delta_k((\hat{\pmb{\psi}}\hat{\pmb{\varphi}})_{r_1} \otimes \mbox{\dots} \otimes (\hat{\pmb{\psi}}\hat{\pmb{\varphi}})_{r_{i-1}} \otimes \sum_{u=1}^{r_i}\hat{\pmb{q}}_{r_i}(\mathbb{1}_V^{\otimes  u-1} \otimes \delta_1 \otimes \mathbb{1}_V^{\otimes  r_i-u})\otimes \mathbb{1}_V^{\otimes k-i})
\end{equation}
\begin{equation}
\tag{Q1.2 + 2.2b}
\label{Q1.2 + 2.2b}
+\sum_{C(n), r_i > 1} \delta_k((\hat{\pmb{\psi}}\hat{\pmb{\varphi}})_{r_1} \otimes \mbox{\dots} \otimes (\hat{\pmb{\psi}}\hat{\pmb{\varphi}})_{r_{i-1}} \otimes \sum_{A'(r_i)}\hat{\pmb{q}}_{k'}(\mathbb{1}_V^{\otimes  i'-1} \otimes \delta_{\ell'} \otimes \mathbb{1}_V^{\otimes  r_i-k'}) + \hat{\pmb{q}}_1\delta_{r_i}\otimes \mathbb{1}_V^{\otimes k-i}) 
\end{equation}
\begin{equation}
\tag{Q1.2 + 2.2c}
\label{Q1.2 + 2.2c}
+\sum_{C(n), r_i = 1} \delta_k((\hat{\pmb{\psi}}\hat{\pmb{\varphi}})_{r_1} \otimes \mbox{\dots} \otimes (\hat{\pmb{\psi}}\hat{\pmb{\varphi}})_{r_{i-1}} \otimes \hat{h}\circ\hat{\pmb{q}}_{r_i}\delta_1 \otimes \mathbb{1}_V^{\otimes k-i}) \end{equation}
$$+ \sum_{C(n), r_i = 1} \delta_k((\hat{\pmb{\psi}}\hat{\pmb{\varphi}})_{r_1} \otimes \mbox{\dots} \otimes (\hat{\pmb{\psi}}\hat{\pmb{\varphi}})_{r_{i-1}} \otimes (-\hat{g}\hat{f} + \mathbb{1}_V)\hat{\pmb{q}}_{r_i} \otimes \mathbb{1}_V^{\otimes k-i}).$$
The non-numbered terms (first, second and sixth) can be further simplified. We notice 
$$-\sum_{C(n), r_i > 1} \delta_k((\hat{\pmb{\psi}}\hat{\pmb{\varphi}})_{r_1} \otimes \mbox{\dots} \otimes (\hat{\pmb{\psi}}\hat{\pmb{\varphi}})_{r_{i-1}} \otimes \sum_{B'(r_i)}\hat{h} \circ \hat{\pmb{p}}_{k'}(\hat{g}\hat{f} \circ \hat{\pmb{q}}_{r'_1} \otimes \mbox{\dots} \otimes \hat{g}\hat{f} \circ \hat{\pmb{q}}_{r'_{k'}})\otimes \mathbb{1}_V^{\otimes k-i}) $$
$$-\sum_{C(n), r_i > 1} \delta_k((\hat{\pmb{\psi}}\hat{\pmb{\varphi}})_{r_1} \otimes \mbox{\dots} \otimes (\hat{\pmb{\psi}}\hat{\pmb{\varphi}})_{r_{i-1}} \otimes \hat{g}\hat{f}\circ \hat{\pmb{q}}_{r_i}\otimes \mathbb{1}_V^{\otimes k-i}) $$
$$+ \sum_{C(n), r_i = 1} \delta_k((\hat{\pmb{\psi}}\hat{\pmb{\varphi}})_{r_1} \otimes \mbox{\dots} \otimes (\hat{\pmb{\psi}}\hat{\pmb{\varphi}})_{r_{i-1}} \otimes (-\hat{g}\hat{f} + \mathbb{1}_V)\hat{\pmb{q}}_{r_i} \otimes \mathbb{1}_V^{\otimes k-i}) $$
$$= -\sum_{C(n)} \delta_k((\hat{\pmb{\psi}}\hat{\pmb{\varphi}})_{r_1} \otimes \mbox{\dots} \otimes (\hat{\pmb{\psi}}\hat{\pmb{\varphi}})_{r_{i-1}} \otimes (\hat{\pmb{\psi}}\hat{\pmb{\varphi}})_{r_i} \otimes \mathbb{1}_V^{\otimes k-i}) $$
$$+ \sum_{C(n)} \delta_k((\hat{\pmb{\psi}}\hat{\pmb{\varphi}})_{r_1} \otimes \mbox{\dots} \otimes (\hat{\pmb{\psi}}\hat{\pmb{\varphi}})_{r_{i-1}} \otimes \mathbb{1}_V^{\otimes k-i+1}) $$
$$= -\sum_{B(n)} \delta_k((\hat{\pmb{\psi}}\hat{\pmb{\varphi}})_{r_1} \otimes \mbox{\dots} \otimes (\hat{\pmb{\psi}}\hat{\pmb{\varphi}})_{r_{k}}) + \delta_n(\mathbb{1}_V^{\otimes  n}).$$
By Lemma~\ref{k2l23}, this expression equals to
\begin{equation}
\tag{Q1.2 + 2.2d}
\label{Q1.2 + 2.2d}
-\sum_{B(n)}\hat{\pmb{p}}_k(\hat{g}\hat{f} \circ \hat{\pmb{q}}_{r_1} \otimes \mbox{\dots} \otimes \hat{g}\hat{f} \circ \hat{\pmb{q}}_{r_k}) + \hat{\pmb{q}}_1\delta_n.
\end{equation}

\noindent \textbf{III.} In the last step we pair various contributions together: 
the first step can be written as
$$\delta_1\hat{q}_n = \eqref{Q1.1} + \eqref{Q1.2} + \eqref{Q1.3} + \eqref{Q2.1} + \eqref{Q2.2} + \eqref{Q2.3},$$
while the second step as
$$\eqref{Q1.1} + \eqref{Q2.1} = \eqref{Q1.1 + 2.1a} + \eqref{Q1.1 + 2.1b} + \eqref{Q1.1 + 2.1c} + \eqref{Q1.1 + 2.1d}$$
and $$\eqref{Q1.2} + \eqref{Q2.2} = \eqref{Q1.2 + 2.2a} + \eqref{Q1.2 + 2.2b} + \eqref{Q1.2 + 2.2c} +  \eqref{Q1.2 + 2.2d}.$$
Taken altogether,
\begin{align*}
\eqref{Q1.3} + \eqref{Q1.1 + 2.1a} + &\eqref{Q1.1 + 2.1b} + \eqref{Q1.2 + 2.2a} +  \eqref{Q1.2 + 2.2c} =\\
&=\sum_{u=1}^{n}\hat{\pmb{q}}_n(\mathbb{1}_V^{\otimes  u-1} \otimes \delta_1 \otimes \mathbb{1}_V^{\otimes  n-u}),\\
\eqref{Q2.3} + \eqref{Q1.1 + 2.1c} + &\eqref{Q1.1 + 2.1d} + \eqref{Q1.2 + 2.2b} =\\ 
&=\sum_{A(n)}\hat{\pmb{q}}_{k}(\mathbb{1}_V^{\otimes  i-1} \otimes \delta_{\ell} \otimes \mathbb{1}_V^{\otimes  n-k}),\\
\eqref{Q1.2 + 2.2d} &= - \sum_{B(n)}\hat{\pmb{p}}_k(\hat{g}\hat{f} \circ \hat{\pmb{q}}_{r_1} \otimes \mbox{\dots} \otimes \hat{g}\hat{f} \circ \hat{\pmb{q}}_{r_k}) + \hat{\pmb{q}}_1\delta_n.
\end{align*}
The proof is complete.
\end{proof}

\section{Homotopy transfer and the homological perturbation lemma}
\label{sec:4}

In the present section we discuss a motivation to find explicit formulas
for the transfer of $A_\infty$ algebra structure presented in an apparently 
arbitrary form in \eqref{Antanz}. In the following, we recall the homological 
perturbation lemma and show that it gives a recipe to search for
the transfer problem exactly in the form \eqref{Antanz}. This is the approach with which
we develop and formalize \cite[Remark $4$]{Markl06}. 

\begin{lemma}[Homological perturbation lemma, \cite{Crainic04}]\label{k2l25}
Let $(V, \partial_V)$ and $(W, \partial_W)$ be chain complexes together with 
quasi-isomorphisms $f: V \to W$ and $g: W \to V$ such that 
$gf - \mathbb{1}_V = \partial_Vh + h\partial_V$ for a linear map 
$h: V \to V$. Let $\pmb{\mu}: V \to V$ 
be a linear map of the same degree as $\partial_V$ such that 
$(\partial_V + \pmb{\mu})^2 = 0$ and the linear map $\mathbb{1}_V - \pmb{\mu} h$ 
is invertible ($\pmb{\mu}$ is called in this context perturbation.) We define
\begin{align}
\pmb{\nu} = \partial_W + fAg, \;\;\;\; \pmb{\psi} = g + hAg, \;\;\;\; \pmb{\varphi} = f + fAh, 
\;\;\;\; \pmb{H} = h + hAh,
\end{align}
where $A = (\mathbb{1}_V - \pmb{\mu}h)^{-1}\pmb{\mu}$.
Then $(V, \partial_V + \pmb{\mu})$ and $(W, \pmb{\nu})$
are chain complexes and $\pmb{\varphi}: V \to V$, 
$\pmb{\psi}: W \to W$ their quasi-isomorphisms with 
$\pmb{\psi\varphi} - \mathbb{1}_V = (\partial_V + \pmb{\mu})\pmb{H} + \pmb{H}(\partial_V + \pmb{\mu}).$	
\end{lemma}

In our case, on $(V, \partial_V)$ we have an additional $A_\infty$ structure given by
a collection of multilinear maps $\pmb{\mu} = (\mu_2, \mu_3, \mbox{\dots})$ fulfilling certain
axioms. In order to regard $\pmb{\mu}$ as a perturbation, we have to pass to the (suspended) 
tensor algebra generated by $V$.  
Let us consider $\overline{T}sV$ with a coderivation $\delta_V$ and $\overline{T}sW$ with a coderivation 
$\delta_W$, $\hat{F}$ and $\hat{G}$ morphisms and $\hat{H}$ a homotopy between $\hat{G}\hat{F}$ 
and the identity on $\overline{T}sV$. 
Here $\delta_V$ is given by components $\{s \circ \partial_V \circ \omega: sV \to sV\} \cup 
\{0: sV^{\otimes n} \to sV\}_{n \geq 2}$ in the sense of Theorem~\ref{k1v5}, and it is 
codifferential by Lemma~\ref{k1t6} because $\partial_V$ is a differential on $V$. Analogous 
conclusions do apply to $\delta_W$.
The map $\hat{F}: (\overline{T}sV, \delta_V) \to (\overline{T}sW, \delta_W)$ is given by 
components $\{s \circ f \circ \omega: sV \to sW\} \cup \{0: (sV)^{\otimes n} \to sW\}_{n \geq 2}$ 
(Lemma~\ref{k1t8}). By Lemma~\ref{k1t9}, $\hat{F}$ is a morphism ($f$ is a map of chain complexes), i.e. 
$\hat{F}|_{(sV)^{\otimes n}} = \hat{f}^{\otimes n}$ for $\hat{f} = s \circ f \circ \omega$.
Analogous conclusions apply to $\hat{G}$ as well.
Homotopy $\hat{H}: \overline{T}sV \to \overline{T}sV$ is a map given by 
$\{\hat{g}\hat{f}: sV \to sV\} \cup \{0: (sV)^{\otimes n} \to sV\}_{n \geq 2}$ on the left, 
$\{s \circ h \circ \omega: sV \to sV\} \cup \{0: (sV)^{\otimes n} \to sV\}_{n \geq 2}$ in 
the middle and $\{\mathbb{1}_V: sV \to sV\} \cup \{0: (sV)^{\otimes n} \to sV\}_{n \geq 2}$ 
on the right in the sense of Theorem~\ref{k1v4}. 
Because $h$ is a homotopy between $gf$ and $\mathbb{1}_V$, $\hat{H}$ is a homotopy between 
$\hat{G}\hat{F}$ and the identity on $\overline{T}sV$ according to  Theorem~\ref{k1v11}; 
the notation is 
$\hat{h} = s \circ h \circ \omega.$

Let $\delta_{\pmb{\mu}}$ be a coderivation on $\overline{T}sV$ corresponding to 
$\pmb{\mu}$, whose components are given by 
$\{0: sV \to sV\} \cup \{s \circ \mu_n \circ \omega^{\otimes n}: (sV)^{\otimes n} \to sV\}_{n \geq 2}$ 
in the sense of Theorem \ref{k1v13}. Because $\partial_V$ and $\pmb{\mu}$ form an 
$A_\infty$ structure on $V$, $(\delta_V + \delta_{\pmb{\mu}})^2 = 0$ by Theorem~\ref{k1v13}; 
we use the notation $\delta_n = s \circ \mu_n \circ \omega^{\otimes n}.$

The remaining assumption in Lemma~\ref{k2l25} is the invertibility of the map 
$\mathbb{1}-\delta_{\pmb{\mu}}\hat{H}$. We know
$$\hat{H}|_{(sV)^{\otimes n}} = \sum_{\substack{i+j = n-1,\\ i,j \geq 0}} (\hat{g}\hat{f})^{\otimes i} \otimes \hat{h} \otimes \mathbb{1}_V^{\otimes j}$$
so that $\hat{H}((sV)^{\otimes n}) \subseteq (sV)^{\otimes n}$ for all $n \geq 1$, and also  
$\delta_{\pmb{\mu}}|_{sV} = 0$ implies
$$\delta_{\pmb{\mu}}|_{(sV)^{\otimes n}} = \delta_n + \sum_{A(n)} \mathbb{1}_V^{\otimes i-1} \otimes \delta_\ell \otimes \mathbb{1}_V^{\otimes n-k}$$
for all $n \geq 2$ with $A(n)$ as in \eqref{A}. Consequently, for all 
$n \geq 2$ {holds} $\delta_{\pmb{\mu}}((sV)^{\otimes n}) \subseteq sV \oplus \mbox{\dots} \oplus (sV)^{\otimes n-1},$ and its iteration results in
$(\delta_{\pmb{\mu}} \hat{H})^{n-1}((sV)^{\otimes n}) \subseteq sV$, 
$(\delta_{\pmb{\mu}} \hat{H})^{n}((sV)^{\otimes n}) = 0$.

By previous discussion and in accordance with Remark $2.3$, \cite{Crainic04},
$$(\mathbb{1}-\delta_{\pmb{\mu}}\hat{H})^{-1}|_{sV \oplus \mbox{\dots} \oplus (sV)^{\otimes n}} = \mathbb{1} + \sum_{i = 1}^{n-1}(\delta_{\pmb{\mu}} \hat{H})^{n},$$ 
which means that $\mathbb{1}-\delta_{\pmb{\mu}}\hat{H}$ is invertible. Now all assumptions of Lemma~\ref{k2l25}
are fulfilled and we can write
\begin{align}
\notag
	& \delta_W + \delta_{\pmb{\nu}} = \delta_W + \hat{F}\left(\delta_{\pmb{\mu}}\sum_{n \geq 0}(\hat{H}\delta_{\pmb{\mu}})^{n}\right)\hat{G},\quad \hat{\pmb{\psi}} = \hat{G} + \hat{H}\left(\delta_{\pmb{\mu}}\sum_{n \geq 0}(\hat{H}\delta_{\pmb{\mu}})^{n}\right)\hat{G},\\
\notag
&	\hat{\pmb{\varphi}} = \hat{F} + \hat{F}\left(\sum_{n \geq 1}(\delta_{\pmb{\mu}}\hat{H})^{n}\right),\quad
\hat{\pmb{H}} = \hat{H} + \hat{H}\left(\sum_{n \geq 1}(\delta_{\pmb{\mu}}\hat{H})^{n}\right).
\end{align}
Here we see immediately the motivation for \eqref{Antanz}: 
 $\delta_{\pmb{\mu}}\sum_{n \geq 0}(\hat{H}\delta_{\pmb{\mu}})^{n}$ corresponds to 
the $\hat{\pmb{p}}-$kernels and $\sum_{n \geq 1}(\delta_{\pmb{\mu}}\hat{H})^{n}$ corresponds 
to the $\hat{\pmb{q}}-$kernels. For our purposes it is more convenient to write
\begin{equation}
\label{PLrozpis}
	\begin{split}
		&\delta_W + \delta_{\pmb{\nu}} = \delta_W +  \hat{F}\delta_{\pmb{\mu}}\hat{G} +\hat{F}\left(\sum_{n \geq 1}(\delta_{\pmb{\mu}}\hat{H})^{n}\right)\delta_{\pmb{\mu}}\hat{G},\\
		&\hat{\pmb{\psi}} = \hat{G} + \hat{H}\delta_{\pmb{\mu}}\hat{G} +\hat{H}\left(\sum_{n \geq 1}(\delta_{\pmb{\mu}}\hat{H})^{n}\right)\delta_{\pmb{\mu}}\hat{G},\\
	&\hat{\pmb{\varphi}} = \hat{F} + \hat{F}\delta_{\pmb{\mu}}\hat{H} + \hat{F}\left(\sum_{n \geq 1}(\delta_{\pmb{\mu}}\hat{H})^{n}\right)\delta_{\pmb{\mu}}\hat{H},\\
	&\hat{\pmb{H}} = \hat{H} + \hat{H}\delta_{\pmb{\mu}}\hat{H} + \hat{H}\left(\sum_{n \geq 1}(\delta_{\pmb{\mu}}\hat{H})^{n}\right)\delta_{\pmb{\mu}}\hat{H}.	
	\end{split}
\end{equation}
There is a drawback related to these formulas, however: by a direct inspection we see that
$\delta_W + \delta_{\pmb{\nu}}$ is not a coderivation in the sense of Theorem~\ref{k1v5}, 
$\hat{\pmb{\varphi}}$ and $\hat{\pmb{\psi}}$ do not define a morphism in the sense of
Lemma~\ref{k1t8}, and $\hat{\pmb{H}}$ does not fulfill the first part of morphism definition
in the sense of Theorem~\ref{k1v4}. 

In what follows we prove that on the additional assumptions (see \cite[Remark $4$]{Markl06}):
\begin{equation}
	\label{Okraj}
	\hat{f}\hat{g} = \mathbb{1}, \;\;\;\; \hat{f}\hat{h} = 0, \;\;\;\; \hat{h}\hat{g} = 0, \;\;\;\; \hat{h}\hat{h} = 0,
\end{equation}
the homological perturbation lemma gives the results compatible with Section \ref{sect:hotr}. 

\begin{lemma}\label{k2l26}
Let us assume the formulas in \eqref{Okraj} are satisfied. Then  
	\begin{enumerate}
		\item $\hat{\pmb{q}}_n \circ \hat{g}^{\otimes n} = 0$ for $n \geq 2$,
		\item 
		$\hat{\pmb{q}}_{i+1+j} \circ ((\hat{g}\hat{f})^{\otimes i} \otimes \hat{h} \otimes \mathbb{1}_V^{\otimes j}) = 0$
		for all $i,j \geq 0$, $i+j \geq 1$.
	\end{enumerate}
\end{lemma}
\begin{proof}
	$(1)$: The proof goes by induction. By definition 
	$\hat{\pmb{q}}_2 = \delta_2(\hat{g}\hat{f} \otimes \hat{h}) + \delta_{2}(\hat{h} \otimes \mathbb{1}_V)$
	for $n = 2$, so that 
	$\hat{\pmb{q}}_2 \otimes \hat{g}^{\otimes 2} = \delta_2(\hat{g}\hat{f}\hat{g} \otimes \hat{h}\hat{g}) 
	+ \delta_{2}(\hat{h}\hat{g} \otimes \hat{g})$ and the claim follows from \eqref{Okraj} ($\hat{h}\hat{g} = 0.$)
	
 We assume the assertion is true for all natural numbers less than $n\in\mN$ ($n\geq 2$.) 
By definition
$$\hat{\pmb{q}}_n\circ \hat{g}^{\otimes n} = \sum_{C(n)}\delta_k([\![\hat{\pmb{\psi}}\hat{\pmb{\varphi}}]\!]_{r_1} \circ \hat{g}^{\otimes r_1} \otimes \mbox{\dots} \otimes [\![\hat{\pmb{\psi}}\hat{\pmb{\varphi}}]\!]_{r_{i-1}} \circ \hat{g}^{\otimes r_{i-1}}\otimes \hat{h} \circ \hat{\pmb{q}}_{r_i} \circ \hat{g}^{\otimes r_i}  \otimes \hat{g}^{\otimes k-i}),$$ 
where 
\begin{equation}
\label{p,q-symbol}
[\![\hat{\pmb{\psi}}\hat{\pmb{\varphi}}]\!]_{m} = \hat{g}\hat{f}\circ \hat{\pmb{q}}_m + \sum_{B(m)}\hat{h}\circ\hat{\pmb{p}}_k(\hat{g}\hat{f} \circ \hat{\pmb{q}}_{r_1} \otimes \mbox{\dots} \otimes \hat{g}\hat{f} \circ \hat{\pmb{q}}_{r_k})	
\end{equation}  
with $[\![\hat{\pmb{\psi}}\hat{\pmb{\varphi}}]\!]_1 = \hat{g}\hat{f}.$ 
In the case $r_i > 1$, the composition $\hat{h} \circ \hat{\pmb{q}}_{r_i} \circ \hat{g}^{\otimes r_i}$ is 
trivial by the induction hypothesis. If $r_i = 1$, $\hat{h} \circ \hat{\pmb{q}}_1 \circ \hat{g} = \hat{h}\hat{g}$ 
is trivial by \eqref{Okraj}.

\noindent $(2)$: The proof is by induction on $n = i+1+j$. For $n = 2$ we prove 
$$\hat{\pmb{q}}_2(\hat{h} \otimes \mathbb{1}_V) = 0, \;\;\;\;\hat{\pmb{q}}_2(\hat{g}\hat{f} \otimes \hat{h}) = 0.$$ 
As we know $\hat{\pmb{q}}_2(\hat{h} \otimes \mathbb{1}_V) = (-1)^{|\hat{h}||\hat{h}|}\delta_2(\hat{g}\hat{f}\hat{h} \otimes \hat{h}) + \delta_2(\hat{h}\hat{h} \otimes \mathbb{1}_V)$ and $\hat{\pmb{q}}_2(\hat{g}\hat{f} \otimes \hat{h}) = \delta_2(\hat{g}\hat{f}\hat{g}\hat{f} \otimes \hat{h}\hat{h}) + \delta_2(\hat{h}\hat{g}\hat{f} \otimes \hat{h})$, the claim 
follows thanks to \eqref{Okraj}.

Let the claim hold for $m \geq 2$ and all natural numbers less than $n$, we prove it is true for $n$. 
First of all, for $n > i'+j'+1 \geq 2$ we have
$$[\![\hat{\pmb{\psi}}\hat{\pmb{\varphi}}]\!]_{i'+1+j'} \circ ((\hat{g}\hat{f})^{\otimes i'} \otimes \hat{h} \otimes \mathbb{1}_V^{\otimes j'}) = 0$$
and also $\hat{g}\hat{f} \circ \hat{\pmb{q}}_1 \circ \hat{h} = \hat{g}\hat{f} \circ \hat{h} = 0.$ 
By definition
$$[\![\hat{\pmb{\psi}}\hat{\pmb{\varphi}}]\!]_{i'+1+j'} \circ ((\hat{g}\hat{f})^{\otimes i'} \otimes \hat{h} \otimes \mathbb{1}_V^{\otimes j'}) = \hat{g}\hat{f} \circ \hat{\pmb{q}}_{i' + 1 + j'}\circ ((\hat{g}\hat{f})^{\otimes i'} \otimes \hat{h} \otimes \mathbb{1}_V^{\otimes j'})$$
$$+ \sum_{B(i' + 1 + j')}\hat{h} \circ \hat{\pmb{p}}_k(\hat{g}\hat{f} \circ \hat{\pmb{q}}_{r_1} \otimes \mbox{\dots} \otimes \hat{g}\hat{f} \circ \hat{\pmb{q}}_{r_k})\circ ((\hat{g}\hat{f})^{\otimes i'} \otimes \hat{h} \otimes \mathbb{1}_V^{\otimes j'}),$$
and by induction hypothesis $\hat{\pmb{q}}_{i' + 1 + j'}\circ ((\hat{g}\hat{f})^{\otimes i'} \otimes \hat{h} \otimes \mathbb{1}_V^{\otimes j'}) = 0.$ The last summation can be conveniently rewritten as 
$$\sum_{B(i' + 1 + j')}\hat{h} \circ\hat{\pmb{p}}_k(\hat{g}\hat{f} \circ \hat{\pmb{q}}_{r_1} \otimes \mbox{\dots} \otimes \hat{g}\hat{f} \circ \hat{\pmb{q}}_{r_k})\circ ((\hat{g}\hat{f})^{\otimes i'} \otimes \hat{h} \otimes \mathbb{1}_V^{\otimes j'}) $$
$$= \sum_{B(i' + 1 + j')}\hat{h} \circ \hat{\pmb{p}}_k(\hat{g}\hat{f} \circ \hat{\pmb{q}}_{r_1} \circ (\hat{g}\hat{f})^{\otimes r_1} \otimes \mbox{\dots}$$
$$ \mbox{\dots} \otimes \hat{g}\hat{f} \circ \hat{\pmb{q}}_{r_u} \circ ((\hat{g}\hat{f})^{\otimes \star} \otimes \hat{h} \otimes \mathbb{1}_V^{\otimes \star})\otimes \mbox{\dots} \otimes \hat{g}\hat{f} \circ \hat{\pmb{q}}_{r_k}),$$
and the induction implies $\hat{\pmb{q}}_{r_u} \circ ((\hat{g}\hat{f})^{\otimes \star} \otimes \hat{h} \otimes \mathbb{1}_V^{\otimes \star}) = 0$ for $r_u > 1$. We already showed 
$\hat{g}\hat{f} \circ \hat{\pmb{q}}_{r_u} \circ ((\hat{g}\hat{f})^{\otimes \star} \otimes \hat{h} \otimes \mathbb{1}_V^{\otimes \star}) = \hat{g}\hat{f} \circ \hat{\pmb{q}}_1 \circ \hat{h} = 0$ for $r_u = 1$.

We now return back to the main thread of the proof and show 
$\hat{\pmb{q}}_{n} \circ ((\hat{g}\hat{f})^{\otimes i} \otimes \hat{h} \otimes \mathbb{1}_V^{\otimes j}) = 0$.
We consider $k,i',r_1, \mbox{\dots}, r_{i'-1}, r_{i'}$ in $C(n)$ given by \eqref{C}, and compute
$$\delta_k([\![\hat{\pmb{\psi}}\hat{\pmb{\varphi}}]\!]_{r_1} \otimes \mbox{\dots} \otimes [\![\hat{\pmb{\psi}}\hat{\pmb{\varphi}}]\!]_{r_{i'-1}} \otimes \hat{h} \circ \hat{\pmb{q}}_{r_{i'}} \otimes \mathbb{1}_V^{\otimes k-i'})\circ ((\hat{g}\hat{f})^{\otimes i} \otimes \hat{h} \otimes \mathbb{1}_V^{\otimes j}).$$
After substitution for $[\![\hat{\pmb{\psi}}\hat{\pmb{\varphi}}]\!]$, there are the following three possibilities 
for indices $i$ a $i'$:  
\begin{description}
	\item[\underline{$i < r_1 + \mbox{\dots} + r_{i'-1}$:}] Then there exist $1 \leq u < i$ such that 
	$[\![\hat{\pmb{\psi}}\hat{\pmb{\varphi}}]\!]_{r_u}\circ ((\hat{g}\hat{f})^{\otimes \star} \otimes \hat{h} \otimes \mathbb{1}_V^{\otimes \star})$. For $r_u \geq 2$ we already proved $[\![\hat{\pmb{\psi}}\hat{\pmb{\varphi}}]\!]_{r_u}\circ ((\hat{g}\hat{f})^{\otimes \star} \otimes \hat{h} \otimes \mathbb{1}_V^{\otimes \star}) = 0$, for $r_u = 1$ 
	we have $[\![\hat{\pmb{\psi}}\hat{\pmb{\varphi}}]\!]_{r_u}\circ ((\hat{g}\hat{f})^{\otimes \star} \otimes \hat{h} \otimes \mathbb{1}_V^{\otimes \star}) = \hat{g}\hat{f} \circ \hat{h} = 0$.
	\item[\underline{$r_1 + \mbox{\dots} + r_{i'-1} \leq i < r_1 + \mbox{\dots} + r_{i'}$:}] In the tensor product there
	is a term of the form
	$\hat{h} \circ \hat{\pmb{q}}_{r_{i'}} \circ ((\hat{g}\hat{f})^{\otimes \star} \otimes \hat{h} \otimes \mathbb{1}_V^{\otimes \star}),$ which is by the induction hypothesis $0$ for $r_{i'} > 1$. If $r_{i'} = 1$, then 
	$\hat{h} \circ \hat{\pmb{q}}_{r_{i'}} \circ ((\hat{g}\hat{f})^{\otimes \star} \otimes \hat{h} \otimes \mathbb{1}_V^{\otimes \star}) = \hat{h}\hat{h}$ equals to $0$ by \eqref{Okraj}.
	\item[\underline{$r_1 + \mbox{\dots} + r_{i'} \leq i$:}] In this case we get in the tensor product the term of the form
	$\hat{h} \circ \hat{\pmb{q}}_{r_{i'}} \circ (\hat{g}\hat{f})^{\otimes r_{i'}} = \hat{h} \circ \hat{\pmb{q}}_{r_{i'}} \circ \hat{g}^{\otimes r_{i'}} \circ \hat{f}^{\otimes r_{i'}}$, which is trivial for $r_{i'} \geq 2$ by $(1)$ of the lemma. 
	If $r_{i'} = 1$, then $\hat{h} \circ \hat{\pmb{q}}_{r_{i'}} \circ (\hat{g}\hat{f})^{\otimes r_{i'}} = \hat{h} \circ \hat{g}\hat{f}$ equals to zero again by \eqref{Okraj}.
\end{description}
Because $k,i',r_1, \mbox{\dots}, r_{i'-1}, r_{i'}$ in $C(n)$ was chosen arbitrarily, we get
$$\sum_{C(n)} \delta_k([\![\hat{\pmb{\psi}}\hat{\pmb{\varphi}}]\!]_{r_1} \otimes \mbox{\dots} \otimes [\![\hat{\pmb{\psi}}\hat{\pmb{\varphi}}]\!]_{r_{i'-1}} \otimes \hat{h} \circ \hat{\pmb{q}}_{r_{i'}} \otimes \mathbb{1}_V^{\otimes k-i'})\circ ((\hat{g}\hat{f})^{\otimes i} \otimes \hat{h} \otimes \mathbb{1}_V^{\otimes j}) = 0,$$
and so finally $\hat{\pmb{q}}_{n} \circ ((\hat{g}\hat{f})^{\otimes i} \otimes \hat{h} \otimes \mathbb{1}_V^{\otimes j}) = 0.$
\end{proof}

\begin{remark}\label{remarkpsifi}
We easily observe:
\begin{enumerate}
	\item For all $n \geq 2$ and for linear mappings $\{\pmb{a}_n: (sV)^{\otimes n} \to sV\}_{n \geq 1}$, 
	$$\sum_{B(n)} \pmb{a}_{r_1} \otimes \mbox{\dots} \otimes \pmb{a}_{r_k} = \sum_{\substack{B(n),\\ r_k > 1}} \pmb{a}_{r_1} \otimes \mbox{\dots} \otimes \pmb{a}_{r_k} $$$$+ \sum_{u = 1}^{n-3}\sum_{\substack{B(n-u),\\ r_k > 1}}\pmb{a}_{r_1} \otimes \mbox{\dots} \otimes \pmb{a}_{r_k} \otimes \pmb{a}_1^{\otimes u}  + \sum_{u = 2}^{n-1}\pmb{a}_u \otimes \pmb{a}_1^{\otimes n-u} + \pmb{a}_1^{\otimes n},$$
	where $B(n)$ given as in \eqref{B},
	\item For all $n \geq 2$, we have  $$[\![\hat{\pmb{\psi}}\hat{\pmb{\varphi}}]\!]_{n} \circ \hat{g}^{\otimes n}= \hat{h} \circ \hat{\pmb{p}}_n \circ \hat{g}^{\otimes n},$$
	and if $\hat{h} \circ \hat{\pmb{p}}_1 = \mathbb{1}_V$ (Definition~\ref{k2d17}) the formula is true for 
	$n = 1$ as well.
	\item For all $n \geq 1$ and $0 \leq u \leq n-1$:
	$$[\![\hat{\pmb{\psi}}\hat{\pmb{\varphi}}]\!]_{n} \circ ((\hat{f}\hat{g})^{\otimes u} \otimes \hat{h} \otimes \mathbb{1}_V^{\otimes n-1-u}) = 0.$$
\end{enumerate}	
\end{remark}

\begin{lemma}\label{k2t27}
	Let us assume \eqref{Okraj} is true for $n \geq 2$. Then 
	\begin{equation}
	\label{T27}
	\hat{\pmb{p}}_n \circ \hat{g}^{\otimes n} = \delta_n \circ \hat{g}^{\otimes n} + \sum_{i = 2}^{n-1} \hat{\pmb{q}}_{n-i+1}\circ\left(\sum_{u = 0}^{n-i}\mathbb{1}_V^{\otimes u} \otimes \delta_{i} \otimes \mathbb{1}_V^{\otimes n-i-u}\right)\circ \hat{g}^{\otimes n}.
	\end{equation}
\end{lemma}
\begin{proof}
The proof goes by induction on $n$. As for $n = 2$ we have $\hat{\pmb{p}}_2 = \delta_2$, hence the claim follows.

We now assume the assertion holds for all natural numbers greater than $1$ and less than $n$. 
Let us consider $2 \leq m < n$ and $k,i,r_1, \mbox{\dots}, r_{i-1}, r_i$ as given in $C(m)$, so that
$$\delta_k([\![\hat{\pmb{\psi}}\hat{\pmb{\varphi}}]\!]_{r_1} \otimes \mbox{\dots} \otimes [\![\hat{\pmb{\psi}}\hat{\pmb{\varphi}}]\!]_{r_{i-1}} \otimes \hat{h} \circ \hat{\pmb{q}}_{r_{i}} \otimes \mathbb{1}_V^{\otimes k-i})\circ (\hat{g}^{\otimes u} \otimes \delta_\ell \otimes \hat{g}^{\otimes m-1-u})= 0$$
whenever $u < r_1 + \mbox{\dots} + r_{i - 1}$ or $r_1 + \mbox{\dots} + r_{i} \leq u$ because  
$\hat{h} \circ \hat{\pmb{q}}_{r_i} \circ \hat{g}^{\otimes r_i} = 0$ for all $r_i \geq 1$ 
by Lemma~\ref{k2l26}.

We fix $n-1 \geq k \geq 2, k \geq i \geq 1$ and $r_1, \mbox{\dots}, r_{i-1} \geq 1$ as in \eqref{C}. 
As follows from the previous observation, all terms in 
$$\sum_{i = 2}^{n-1} \hat{\pmb{q}}_{n-i+1}\circ\left(\sum_{u = 0}^{n-i}\mathbb{1}_V^{\otimes u} \otimes \delta_{i} \otimes \mathbb{1}_V^{\otimes n-i-u}\right)\circ \hat{g}^{\otimes n}$$
are of the form
$\delta_k([\![\hat{\pmb{\psi}}\hat{\pmb{\varphi}}]\!]_{r_1}\circ \hat{g}^{\otimes r_1} \otimes \mbox{\dots} \otimes [\![\hat{\pmb{\psi}}\hat{\pmb{\varphi}}]\!]_{r_{i-1}}\circ \hat{g}^{\otimes r_{i-1}} \otimes \star \otimes \hat{g}^{\otimes k-i})$ with $\star$ representing a mapping $(sV)^{\otimes \star} \to sV$
(the $\hat{\pmb{q}}-$kernels are given by \eqref{q-kernel}. We can rewrite them in the form
$$\delta_k([\![\hat{\pmb{\psi}}\hat{\pmb{\varphi}}]\!]_{r_1}\circ \hat{g}^{\otimes r_1} \otimes \mbox{\dots} \otimes [\![\hat{\pmb{\psi}}\hat{\pmb{\varphi}}]\!]_{r_{i-1}} \circ \hat{g}^{\otimes r_{i-1}} \otimes$$
$$ \otimes \left[\delta_{n'} \circ \hat{g}^{\otimes r_1} + \sum_{i = 2}^{n'-1} \hat{\pmb{q}}_{n'-i+1}\circ\left(\sum_{u = 0}^{n'-i}\mathbb{1}_V^{\otimes u} \otimes \delta_{i} \otimes \mathbb{1}_V^{\otimes n'-i-u}\right)\circ \hat{g}^{\otimes n'}\right] \otimes \hat{g}^{\otimes k-i}),$$
where $n' = n + i - k - (r_1 + \mbox{\dots} + r_{i-1})$, $n' > 1$.
Applying the second point of Remark~\ref{remarkpsifi} to 
$[\![\hat{\pmb{\psi}}\hat{\pmb{\varphi}}]\!]_{\star}\circ \hat{g}^{\otimes \star}$,
the inducing hypothesis reduces the last display to
\begin{equation}
\label{T27_1}
\delta_k(\hat{h} \circ \hat{\pmb{p}}_{r_1} \circ \hat{g}^{\otimes r_1} \otimes \mbox{\dots} \otimes \hat{h} \circ \hat{\pmb{p}}_{r_{i-1}} \circ \hat{g}^{\otimes r_{i-1}} \otimes \hat{h} \circ \hat{\pmb{p}}_{n'} \circ \hat{g}^{\otimes n'} \otimes \hat{g}^{\otimes k-i})
\end{equation}
(we write $\hat{g} = \hat{h}\circ\hat{\pmb{p}}_1 \circ \hat{g}.$) 
By the first point of Remark~\ref{remarkpsifi},
$$\sum_{i = 2}^{n-1} \hat{\pmb{q}}_{n-i+1}\circ\left(\sum_{u = 0}^{n-i}\mathbb{1}_V^{\otimes u} \otimes \delta_{i} \otimes \mathbb{1}_V^{\otimes n-i-u}\right)\circ \hat{g}^{\otimes n}$$
$$= \sum_{C(n), r_i > 1}\delta_k(\hat{h} \circ \hat{\pmb{p}}_{r_1} \circ \hat{g}^{\otimes r_1} \otimes \mbox{\dots} \otimes \hat{h} \circ \hat{\pmb{p}}_{r_{i-1}} \circ \hat{g}^{\otimes r_{i-1}} \otimes \hat{h} \circ \hat{\pmb{p}}_{r_i} \circ \hat{g}^{\otimes r_i} \otimes \hat{g}^{\otimes k-i})$$
$$= \sum_{B(n), k \neq n}\delta_k(\hat{h} \circ \hat{\pmb{p}}_{r_1} \circ \hat{g}^{\otimes r_1} \otimes \mbox{\dots} \otimes \hat{h} \circ \hat{\pmb{p}}_{r_k} \circ \hat{g}^{\otimes r_k}),$$
so that for each term in the sum there exists $u$, $r_u > 1$ (they are of the form of terms in \eqref{T27_1} with $n' > 1$.)
Adding the remaining term $\delta_n \circ \hat{g}^{\otimes n}$ and using the formula \eqref{p-kernel}
for the $\hat{\pmb{p}}-$kernels, the proof concludes.
\end{proof}

\begin{lemma}\label{k2l28}
	Let us assume \eqref{Okraj}, and also 
\begin{equation}
\label{L28}
\hat{\pmb{q}}_n = \delta_n \circ \hat{H}|_{(sV)^{\otimes n}} + \sum_{i = 2}^{n-1} \hat{\pmb{q}}_{n-i+1}\circ\left(\sum_{u = 0}^{n-i}\mathbb{1}_V^{\otimes u} \otimes \delta_{i} \otimes \mathbb{1}_V^{\otimes n-i-u}\right)\circ \hat{H}|_{(sV)^{\otimes n}}	
\end{equation}
to be true for all $2 \leq m \leq n$. Then 
	$$[\![\hat{\pmb{\psi}}\hat{\pmb{\varphi}}]\!]_{n} - \hat{h} \circ \hat{\pmb{p}}_n \circ (\hat{g}\hat{f})^{\otimes n} =$$
	$$= [\![\hat{\pmb{\psi}}\hat{\pmb{\varphi}}]\!]_{1} \circ \delta_n \circ \hat{H}|_{(sV)^{\otimes n}} $$$$+ \sum_{i = 2}^{n-1} [\![\hat{\pmb{\psi}}\hat{\pmb{\varphi}}]\!]_{n-i+1}\circ\left(\sum_{u = 0}^{n-i}\mathbb{1}_V^{\otimes u} \otimes \delta_{i} \otimes \mathbb{1}_V^{\otimes n-i-u}\right)\circ \hat{H}|_{(sV)^{\otimes n}}.$$
\end{lemma}
\begin{proof} By \eqref{p,q-symbol}, we have for all $m \geq 2$
$$[\![\hat{\pmb{\psi}}\hat{\pmb{\varphi}}]\!]_{m} = \hat{g}\hat{f}\circ \hat{\pmb{q}}_m + \sum_{B(m)}\hat{h}\circ\hat{\pmb{p}}_k(\hat{g}\hat{f} \circ \hat{\pmb{q}}_{r_1} \otimes \mbox{\dots} \otimes \hat{g}\hat{f} \circ \hat{\pmb{q}}_{r_k}),$$ 
(and $[\![\hat{\pmb{\psi}}\hat{\pmb{\varphi}}]\!]_{1} = \hat{g}\hat{f}.$) We can split 
$$[\![\hat{\pmb{\psi}}\hat{\pmb{\varphi}}]\!]_{1} \circ \delta_n \circ \hat{H}|_{(sV)^{\otimes n}} $$$$+ \sum_{i = 2}^{n-1} [\![\hat{\pmb{\psi}}\hat{\pmb{\varphi}}]\!]_{n-i+1}\circ\left(\sum_{u = 0}^{n-i}\mathbb{1}_V^{\otimes u} \otimes \delta_{i} \otimes \mathbb{1}_V^{\otimes n-i-u}\right)\circ \hat{H}|_{(sV)^{\otimes n}}$$
in two components and write
\begin{align}
		\label{L28_1}
\hat{g}\hat{f} &\circ \hat{\pmb{q}}_1 \circ \delta_n \circ \hat{H}|_{(sV)^{\otimes n}}\\ \notag+ \sum_{i = 2}^{n-1} \hat{g}\hat{f} \circ \hat{\pmb{q}}_{n-i+1}\circ&\left(\sum_{u = 0}^{n-i}\mathbb{1}_V^{\otimes u} \otimes \delta_{i} \otimes \mathbb{1}_V^{\otimes n-i-u}\right)\circ \hat{H}|_{(sV)^{\otimes n}}
\end{align}
\begin{equation}
	\label{L28_2}
+ \sum_{i = 2}^{n-1}\sum_{B(n-i+1)}\hat{h}\circ\hat{\pmb{p}}_k(\hat{g}\hat{f} \circ \hat{\pmb{q}}_{r_1} \otimes \mbox{\dots} \otimes \hat{g}\hat{f} \circ \hat{\pmb{q}}_{r_k}) \circ \left(\sum_{u = 0}^{n-i}\mathbb{1}_V^{\otimes u} \otimes \delta_{i} \otimes \mathbb{1}_V^{\otimes n-i-u}\right)\circ \hat{H}|_{(sV)^{\otimes n}}.
\end{equation}
Because $\hat{\pmb{q}}_n = \mathbb{1}_{V}$, we have $\eqref{L28_1} = \hat{g}\hat{f} \circ \hat{\pmb{q}}_n$ 
thanks to \eqref{L28}.
As for the second component \eqref{L28_2}, consider $k, r_1, \mbox{\dots}, r_k \in B(n-i+1)$ for 
some $i \geq 2$ with $B(n-i+1)$ as in \eqref{B}. Then
$$\hat{h}\circ\hat{\pmb{p}}_k(\hat{g}\hat{f} \circ \hat{\pmb{q}}_{r_1} \otimes \mbox{\dots} \otimes \hat{g}\hat{f} \circ \hat{\pmb{q}}_{r_k}) \circ \left(\sum_{u = 0}^{n-i}\mathbb{1}_V^{\otimes u} \otimes \delta_{i} \otimes \mathbb{1}_V^{\otimes n-i-u}\right)\circ \hat{H}|_{(sV)^{\otimes n}} $$
$$= \sum_{u=1}^{k} \hat{h}\circ\hat{\pmb{p}}_k(\hat{g}\hat{f} \circ \hat{\pmb{q}}_{r_1} \circ (\hat{g}\hat{f})^{\otimes r_1} \otimes \mbox{\dots} \otimes \hat{g}\hat{f} \circ \hat{\pmb{q}}_{r_{u-1}} \circ (\hat{g}\hat{f})^{\otimes r_{u-1}} \otimes$$
$$ \otimes \left[\hat{g}\hat{f} \circ \hat{\pmb{q}}_{r_u} \circ \left(\sum_{v = 0}^{r_u-1}\mathbb{1}_V^{\otimes v} \otimes \delta_{i} \otimes \mathbb{1}_V^{\otimes r_u-1-v}\right)\circ \hat{H}|_{(sV)^{\otimes r_u-1+i}}\right] \otimes \mbox{\dots} \otimes \hat{g}\hat{f} \circ \hat{\pmb{q}}_{r_k}).$$
The reason for the appearance of such terms is that when $r_\star \geq 2$ and $\hat{h}$ were 
in any other $\hat{\pmb{q}}-$kernel than $\delta_i$, we would get $\hat{\pmb{q}}_{r_\star}\circ ((gf)^{\otimes \star} \otimes h \otimes \mathbb{1}_V^{\otimes \star})$ which is trivial by Lemma~\ref{k2l26}. If $r_\star = 1$, 
we get $\hat{g}\hat{f} \circ \hat{\pmb{q}}_1 \circ \hat{h} = 0$ because 
$\hat{\pmb{q}}_1 = \mathbb{1}_V$ and $\hat{f}\hat{h} = 0$ by \eqref{Okraj}.
Thus we have for $i \geq 2$:
$$\sum_{B(n-i+1)}\hat{h}\circ\hat{\pmb{p}}_k(\hat{g}\hat{f} \circ \hat{\pmb{q}}_{r_1} \otimes \mbox{\dots} \otimes \hat{g}\hat{f} \circ \hat{\pmb{q}}_{r_k})\circ\left(\sum_{u = 0}^{n-i}\mathbb{1}_V^{\otimes u} \otimes \delta_{i} \otimes \mathbb{1}_V^{\otimes n-i-u}\right)\circ \hat{H}|_{(sV)^{\otimes n}} $$
$$= \sum_{B(n-i+1)}\hat{h}\circ\hat{\pmb{p}}_k(\underline{\hat{g}\hat{f} \circ \hat{\pmb{q}}_{r_1}} \otimes \mbox{\dots} \otimes \hat{g}\hat{f} \circ \hat{\pmb{q}}_{r_k}) $$
$$+\sum_{u = 1}^{n-i-1}\sum_{B(n-i+1-u)}\hat{h}\circ\hat{\pmb{p}}_{u + k}((\hat{g}\hat{f} \circ \hat{\pmb{q}}_1)^{\otimes u} \otimes \underline{\hat{g}\hat{f} \circ \hat{\pmb{q}}_{r_1}} \otimes \mbox{\dots} \otimes \hat{g}\hat{f} \circ \hat{\pmb{q}}_{r_k}) $$
$$+ \sum_{r_1 = 1}^{n-1} \hat{h} \circ \hat{\pmb{p}}_{n-r_1 + 1}((\hat{g}\hat{f} \circ \hat{\pmb{q}}_1)^{\otimes n- r_1} \otimes \underline{\hat{g}\hat{f} \circ \hat{\pmb{q}}_{r_1}}),$$
where $\underline{\hat{g}\hat{f} \circ \hat{\pmb{q}}_{r_1}} = \hat{g}\hat{f} \circ \hat{\pmb{q}}_{r_1} \circ \left(\sum_{v = 0}^{r_1-1}\mathbb{1}_V^{\otimes v} \otimes \delta_{i} \otimes \mathbb{1}_V^{\otimes r_1-1-v}\right)\circ \hat{H}|_{(sV)^{\otimes r_1 - 1 + i}}$. We notice $\hat{\pmb{q}}_m \circ (\hat{g}\hat{f})^{\otimes m} \neq 0$
if and only if $m = 1$ (cf., Lemma~\ref{k2l26}.)
Therefore, we expand the second contribution into
$$\eqref{L28_2} = \sum_{i = 2}^{n-1}\sum_{B(n-i+1)}\hat{h}\circ\hat{\pmb{p}}_k(\underline{\hat{g}\hat{f} \circ \hat{\pmb{q}}_{r_1}} \otimes \mbox{\dots} \otimes \hat{g}\hat{f} \circ \hat{\pmb{q}}_{r_k}) $$
$$+\sum_{i = 2}^{n-1}\sum_{u = 1}^{n-i-1}\sum_{B(n-i+1-u)}\hat{h}\circ\hat{\pmb{p}}_{u + k}((\hat{g}\hat{f} \circ \hat{\pmb{q}}_1)^{\otimes u} \otimes \underline{\hat{g}\hat{f} \circ \hat{\pmb{q}}_{r_1}} \otimes \mbox{\dots} \otimes \hat{g}\hat{f} \circ \hat{\pmb{q}}_{r_k}) $$
$$+ \sum_{i = 2}^{n-1}\sum_{r_1 = 1}^{n-1} \hat{h} \circ \hat{\pmb{p}}_{n-r_1 + 1}((\hat{g}\hat{f} \circ \hat{\pmb{q}}_1)^{\otimes n- r_1} \otimes \underline{\hat{g}\hat{f} \circ \hat{\pmb{q}}_{r_1}})$$
and for fixed $k, i, r_2, \mbox{\dots}, r_i$ sum up all terms of the form 
$\hat{h} \circ \hat{\pmb{p}}_k((\hat{g}\hat{f} \circ \hat{\pmb{q}}_1)^{\otimes k-i} \otimes \underline{\hat{g}\hat{f} \circ \hat{\pmb{q}}_{\star}} \otimes \hat{g}\hat{f} \circ \hat{\pmb{q}}_{r_2} \otimes \mbox{\dots} \otimes \hat{g}\hat{f} \circ \hat{\pmb{q}}_{r_i})$ in \eqref{L28_2}:
$$\hat{h} \circ \hat{\pmb{p}}_k((\hat{g}\hat{f} \circ \hat{\pmb{q}}_1)^{\otimes k-i} \otimes \sum_{r_1 = 1}^{r'}\underline{\hat{g}\hat{f} \circ \hat{\pmb{q}}_{r_1}} \otimes \hat{g}\hat{f} \circ \hat{\pmb{q}}_{r_2} \otimes \mbox{\dots} \otimes \hat{g}\hat{f} \circ \hat{\pmb{q}}_{r_i}),$$
where $r' = n - k + i - (r_2 + \mbox{\dots} + r_i).$ We recall
$\underline{\hat{g}\hat{f} \circ \hat{\pmb{q}}_{r_1}} = \hat{g}\hat{f} \circ \hat{\pmb{q}}_{r_1}\circ\left(\sum_{u = 0}^{r_1 - 1}\mathbb{1}_V^{\otimes u} \otimes \delta_{r' - r_1 + 1} \otimes \mathbb{1}_V^{\otimes r_1-1-u}\right)\circ \hat{H}|_{(sV)^{\otimes r_1 - 1 + i}}$ and use \eqref{L28} to get
\begin{equation}
\label{L28_3}
	\hat{h} \circ \hat{\pmb{p}}_k((\hat{g}\hat{f} \circ \hat{\pmb{q}}_1)^{\otimes k-i} \otimes \hat{g}\hat{f} \circ \hat{\pmb{q}}_{r'} \otimes \hat{g}\hat{f} \circ \hat{\pmb{q}}_{r_2} \otimes \mbox{\dots} \otimes \hat{g}\hat{f} \circ \hat{\pmb{q}}_{r_i}).
\end{equation}
Clearly $r' > 1$, and the summation over all terms in \eqref{L28_2} leads to
$$\eqref{L28_2} = \sum_{B(n), r_1 > 1}\hat{h} \circ \hat{\pmb{p}}_k(\hat{g}\hat{f} \circ \hat{\pmb{q}}_{r_1} \otimes \mbox{\dots} \otimes \hat{g}\hat{f} \circ \hat{\pmb{q}}_{r_k})$$
$$+ \sum_{u = 1}^{n-3} \sum_{B(n-u), r_1 > 1} \hat{h} \circ \hat{\pmb{p}}_k((\hat{g}\hat{f} \circ \hat{\pmb{q}}_1)^{\otimes u} \otimes \hat{g}\hat{f} \circ \hat{\pmb{q}}_{r_2} \otimes \mbox{\dots} \otimes \hat{g}\hat{f} \circ \hat{\pmb{q}}_{r_k}) 
$$
$$+ \sum_{r = 2}^{n-1}\hat{h} \circ \hat{\pmb{p}}_{n-r+1}((\hat{g}\hat{f} \circ \hat{\pmb{q}}_1)^{\otimes n-r} \otimes \hat{g}\hat{f} \circ \hat{\pmb{q}}_{r}).$$
We conclude
$$\eqref{L28_2} = \sum_{B(n), k \neq n}\hat{h} \circ \hat{\pmb{p}}_k(\hat{g}\hat{f} \circ \hat{\pmb{q}}_{r_1} \otimes \mbox{\dots} \otimes \hat{g}\hat{f} \circ \hat{\pmb{q}}_{r_k}),$$
because all terms are as those in \eqref{L28_3} and there is always at least one $u$ such that 
$r_u > 1$ (this is equivalent to $r' > 1$ in \eqref{L28_3}.) Recall we started with
$$[\![\hat{\pmb{\psi}}\hat{\pmb{\varphi}}]\!]_{1} \circ \delta_n \circ \hat{H}|_{(sV)^{\otimes n}} + \sum_{i = 2}^{n-1} [\![\hat{\pmb{\psi}}\hat{\pmb{\varphi}}]\!]_{n-i+1}\circ\left(\sum_{u = 0}^{n-i}\mathbb{1}_V^{\otimes u} \otimes \delta_{i} \otimes \mathbb{1}_V^{\otimes n-i-u}\right)\circ \hat{H}|_{(sV)^{\otimes n}}$$
$$= \eqref{L28_1} + \eqref{L28_2}$$
and showed
$$\eqref{L28_1} + \eqref{L28_2} = \hat{g}\hat{f} \circ \hat{\pmb{q}}_n + \sum_{B(n), k \neq n}\hat{h} \circ \hat{\pmb{p}}_k(\hat{g}\hat{f} \circ \hat{\pmb{q}}_{r_1} \otimes \mbox{\dots} \otimes \hat{g}\hat{f} \circ \hat{\pmb{q}}_{r_k}).$$ 
Taking into account the definition of $[\![\hat{\pmb{\psi}}\hat{\pmb{\varphi}}]\!]_{n}$ in \eqref{p,q-symbol}, the 
desired conclusion follows immediately.
\end{proof}

\begin{lemma}\label{k2t29}
	Let us assume \eqref{Okraj} to be true. Then for all $n \geq 2$
	\begin{equation}
	\label{T29}
		\hat{\pmb{q}}_n = \delta_n \circ \hat{H}|_{(sV)^{\otimes n}} + \sum_{i = 2}^{n-1} \hat{\pmb{q}}_{n-i+1}\circ\left(\sum_{u = 0}^{n-i}\mathbb{1}_V^{\otimes u} \otimes \delta_{i} \otimes \mathbb{1}_V^{\otimes n-i-u}\right)\circ \hat{H}|_{(sV)^{\otimes n}}.
	\end{equation}
\end{lemma}
\begin{proof}
The proof is by the induction hypothesis on $n$. For $n = 2$, by \eqref{T29} 
we have $\delta_2(\hat{g}\hat{f} \otimes \hat{h}) + \delta_2(\hat{h} \otimes \mathbb{1}_V) 
= \delta_2 \circ (\hat{g}\hat{f} \otimes \hat{h} + \hat{h} \otimes \mathbb{1}_V)$ which
is certainly true. 

We assume the claim is true for all natural numbers greater than 
$1$ and strictly less than $n$. Let us consider $2 \leq j < n$ and 
$k,i,r_1, \mbox{\dots}, r_{i-1}, r_i$ as given in $C(n-j+1)$. 
The same reasoning as in Lemma~\ref{k2l28} leads to
$$\delta_k([\![\hat{\pmb{\psi}}\hat{\pmb{\varphi}}]\!]_{r_1} \otimes \mbox{\dots} \otimes [\![\hat{\pmb{\psi}}\hat{\pmb{\varphi}}]\!]_{r_{i-1}} \otimes \hat{h} \circ \hat{\pmb{q}}_{r_{i}} \otimes \mathbb{1}_V^{\otimes k-i})\circ$$
$$\circ \left(\sum_{u = 0}^{n-j}\mathbb{1}_V^{\otimes u} \otimes \delta_{j} \otimes \mathbb{1}_V^{\otimes n-j-u}\right)\circ \hat{H}|_{(sV)^{\otimes n}} 
$$
\begin{equation}
\label{T29_1}
	\begin{split}
&= \delta_k(\underline{[\![\hat{\pmb{\psi}}\hat{\pmb{\varphi}}]\!]_{r_1}} \otimes \mbox{\dots} \otimes [\![\hat{\pmb{\psi}}\hat{\pmb{\varphi}}]\!]_{r_{i-1}} \otimes \hat{h} \circ \hat{\pmb{q}}_{r_{i}} \otimes \mathbb{1}_V^{\otimes k-i}) \\
&+ \delta_k(\hat{h} \circ \hat{\pmb{p}}_{r_1} \circ (\hat{g}\hat{f})^{\otimes r_1}  \otimes \underline{[\![\hat{\pmb{\psi}}\hat{\pmb{\varphi}}]\!]_{r_2}} \otimes \mbox{\dots} \otimes [\![\hat{\pmb{\psi}}\hat{\pmb{\varphi}}]\!]_{r_{i-1}} \otimes \hat{h} \circ \hat{\pmb{q}}_{r_{i}} \otimes \mathbb{1}_V^{\otimes k-i}) \\
& + \delta_k(\hat{h} \circ \hat{\pmb{p}}_{r_1} \circ (\hat{g}\hat{f})^{\otimes r_1}  \otimes \mbox{\dots} \otimes \hat{h} \circ \hat{\pmb{p}}_{r_{i-2}} \circ (\hat{g}\hat{f})^{\otimes r_{i-2}}  \otimes \underline{[\![\hat{\pmb{\psi}}\hat{\pmb{\varphi}}]\!]_{r_{i-1}}} \otimes \hat{h} \circ \hat{\pmb{q}}_{r_{i}} \otimes \mathbb{1}_V^{\otimes k-i}) \\
&+ \delta_k(\hat{h} \circ \hat{\pmb{p}}_{r_1} \circ (\hat{g}\hat{f})^{\otimes r_1}  \otimes \mbox{\dots} \otimes \hat{h} \circ \hat{\pmb{p}}_{r_{i-1}} \circ (\hat{g}\hat{f})^{\otimes r_{i-1}}  \otimes \underline{\hat{h} \circ \hat{\pmb{q}}_{r_{i}}} \otimes \mathbb{1}_V^{\otimes k-i}) \\
&+ \delta_k(\hat{h} \circ \hat{\pmb{p}}_{r_1} \circ (\hat{g}\hat{f})^{\otimes r_1}  \otimes \mbox{\dots} \otimes \hat{h} \circ \hat{\pmb{p}}_{r_{i-1}} \circ (\hat{g}\hat{f})^{\otimes r_{i-1}}  \otimes \overline{\hat{h} \circ \hat{\pmb{q}}_{r_{i}}} \otimes \hat{H}|_{(sV)^{\otimes k -i}}).
	\end{split}
\end{equation}
Hereby we expanded a general summand in the definition of $\hat{\pmb{q}}_{n-j+1}$ as in
\eqref{q-kernel}, where 
$$ \underline{[\![\hat{\pmb{\psi}}\hat{\pmb{\varphi}}]\!]_{r_\ell}} = [\![\hat{\pmb{\psi}}\hat{\pmb{\varphi}}]\!]_{r_1}\left(\sum_{u = 0}^{r_\ell - 1}\mathbb{1}_V^{\otimes u} \otimes \delta_{j} \otimes \mathbb{1}_V^{\otimes r_\ell - 1 -u}\right)\circ \hat{H}|_{(sV)^{\otimes r_i - 1 + j}},$$
$$ \underline{\hat{h} \circ \hat{\pmb{q}}_{r_{i}}} = \hat{h} \circ \hat{\pmb{q}}_{r_{i}}\left(\sum_{u = 0}^{r_i - 1}\mathbb{1}_V^{\otimes u} \otimes \delta_{j} \otimes \mathbb{1}_V^{\otimes r_i - 1 - u}\right)\circ \hat{H}|_{(sV)^{\otimes r_i - 1 + j}},$$
$$ \overline{\hat{h} \circ \hat{\pmb{q}}_{r_{i}}} = \hat{h} \circ \hat{\pmb{q}}_{r_{i}}\left(\sum_{u = 0}^{r_i - 1}\mathbb{1}_V^{\otimes u} \otimes \delta_{j} \otimes \mathbb{1}_V^{\otimes r_i - 1 - u}\right)\circ (\hat{g}\hat{f})^{\otimes r_i - 1 + j}.$$
In the previous formulas there are no signs whatsoever, because 
$\hat{h} \circ \hat{\pmb{q}}_{r_i}$ pass through the terms of degree $0$, and 
$\hat{h}$ in $\hat{H}$ and $\delta_j$ are of degree $1$ and $-1$, respectively,
so that their sign contributions cancel out. 

In the next few steps we show how the terms are organized:

\noindent \textbf{I.} Let us choose $k, i, r_1, \mbox{\dots}, r_i$ given in $C(n)$ such
that $r_i > 1$, and sum up all terms of the form 
$\delta_k(\hat{h} \circ \hat{\pmb{p}}_{r_1} \circ (\hat{g}\hat{f})^{\otimes r_1}  \otimes \mbox{\dots} \otimes \hat{h} \circ \hat{\pmb{p}}_{r_{i-1}} \circ (\hat{g}\hat{f})^{\otimes r_{i-1}}  \otimes \underline{\hat{h} \circ \hat{\pmb{q}}_r} \otimes \mathbb{1}_V^{\otimes k-i})$ out of the summation 
$$ \delta_n \circ \hat{H}|_{(sV)^{\otimes n}} + \sum_{i = 2}^{n-1} \hat{\pmb{q}}_{n-i+1}\circ\left(\sum_{u = 0}^{n-i}\mathbb{1}_V^{\otimes u} \otimes \delta_{i} \otimes \mathbb{1}_V^{\otimes n-i-u}\right)\circ \hat{H}|_{(sV)^{\otimes n}}$$
for all allowable $r$. We get
$$\delta_k(\hat{h} \circ \hat{\pmb{p}}_{r_1} \circ (\hat{g}\hat{f})^{\otimes r_1}  \otimes \mbox{\dots} \otimes \hat{h} \circ \hat{\pmb{p}}_{r_{i-1}} \circ (\hat{g}\hat{f})^{\otimes r_{i-1}}  \otimes \sum_{r = 1}^{r_i}\underline{\hat{h} \circ \hat{\pmb{q}}_{r}} \otimes \mathbb{1}_V^{\otimes k-i}),$$
where
$$\underline{\hat{h}\circ \hat{\pmb{q}}_{r}} = \hat{h}\circ \hat{\pmb{q}}_{r} \circ\left(\sum_{u = 0}^{r - 1}\mathbb{1}_V^{\otimes u} \otimes \delta_{r_i - r + 1} \otimes \mathbb{1}_V^{\otimes r-1-u}\right)\circ \hat{H}|_{(sV)^{\otimes r_i}}.$$
Because $r_i < n$, we get by the induction hypothesis
\begin{equation}\label{terma}
\delta_k(\hat{h} \circ \hat{\pmb{p}}_{r_1} \circ (\hat{g}\hat{f})^{\otimes r_1}  \otimes \mbox{\dots} \otimes \hat{h} \circ \hat{\pmb{p}}_{r_{i-1}} \circ (\hat{g}\hat{f})^{\otimes r_{i-1}}  \otimes \hat{h} \circ \hat{\pmb{q}}_{r_i} \otimes \mathbb{1}_V^{\otimes k-i}).
\end{equation}
If $r_{i - 1} > 1$, we sum up all terms of the form 
$\delta_k(\hat{h} \circ \hat{\pmb{p}}_{r_1} \circ (\hat{g}\hat{f})^{\otimes r_1}  \otimes \mbox{\dots} \otimes\underline{[\![\hat{\pmb{\psi}}\hat{\pmb{\varphi}}]\!]_{\star}} \otimes \hat{h} \circ \hat{\pmb{q}}_{r_i} \otimes \mathbb{1}_V^{\otimes k-i})$:
$$ \delta_k(\hat{h} \circ \hat{\pmb{p}}_{r_1} \circ (\hat{g}\hat{f})^{\otimes r_1}  \otimes \mbox{\dots} \otimes \sum_{r = 1}^{r_{i-1}}\underline{[\![\hat{\pmb{\psi}}\hat{\pmb{\varphi}}]\!]_{r}} \otimes \hat{h} \circ \hat{\pmb{q}}_{r_i} \otimes \mathbb{1}_V^{\otimes k-i}),
$$
with
$$\underline{[\![\hat{\pmb{\psi}}\hat{\pmb{\varphi}}]\!]_{r}} = [\![\hat{\pmb{\psi}}\hat{\pmb{\varphi}}]\!]_{r} \circ\left(\sum_{u = 0}^{r - 1}\mathbb{1}_V^{\otimes u} \otimes \delta_{r_{i-1} - r + 1} \otimes \mathbb{1}_V^{\otimes r-1-u}\right)\circ \hat{H}|_{(sV)^{\otimes r_{i-1}}}.$$
Because $r_i < n$, by the induction hypothesis is Lemma~\ref{k2l28} 
fulfilled and the last display reduces to
$$ \delta_k(\hat{h} \circ \hat{\pmb{p}}_{r_1} \circ (\hat{g}\hat{f})^{\otimes r_1}  \otimes \mbox{\dots} \otimes \hat{h} \circ \hat{\pmb{p}}_{r_{i-2}} \circ (\hat{g}\hat{f})^{\otimes r_{i-2}} \otimes$$
$$ \otimes \left[ [\![\hat{\pmb{\psi}}\hat{\pmb{\varphi}}]\!]_{r_{i-1}} - \hat{h} \circ \hat{\pmb{p}}_{r_{i-1}} \circ (\hat{g}\hat{f})^{r_{i-1}} \right]\otimes \hat{h} \circ \hat{\pmb{q}}_{r_i} \otimes \mathbb{1}_V^{\otimes k-i}).$$
The sum of the last display and \eqref{terma} results in 
$$\delta_k(\hat{h} \circ \hat{\pmb{p}}_{r_1} \circ (\hat{g}\hat{f})^{\otimes r_1}  \otimes \mbox{\dots} \otimes \hat{h} \circ \hat{\pmb{p}}_{r_{i-2}} \circ (\hat{g}\hat{f})^{\otimes r_{i-2}} \otimes [\![\hat{\pmb{\psi}}\hat{\pmb{\varphi}}]\!]_{r_{i-1}}\otimes \hat{h} \circ \hat{\pmb{q}}_{r_i} \otimes \mathbb{1}_V^{\otimes k-i}),$$
which is the same expression as for $r_{i-1} = 1$ because $[\![\hat{\pmb{\psi}}\hat{\pmb{\varphi}}]\!]_{r_{i-1}} = \hat{h} \circ \hat{\pmb{p}}_{r_{i-1}} \circ (\hat{g}\hat{f})^{r_{i-1}}$
in this case.
Repeating this procedure, we arrive at
$\delta_k([\![\hat{\pmb{\psi}}\hat{\pmb{\varphi}}]\!]_{r_{1}} \otimes \mbox{\dots} \otimes [\![\hat{\pmb{\psi}}\hat{\pmb{\varphi}}]\!]_{r_{i-1}}\otimes \hat{h} \circ \hat{\pmb{q}}_{r_i} \otimes \mathbb{1}_V^{\otimes k-i}).$

We summarize the previous considerations: for $k, i, r_1, \mbox{\dots}, r_i$ as in 
$C(n)$ such that $r_i > 1$, we have
\begin{equation}
\label{T29_2}
	\begin{split}
&\delta_k([\![\hat{\pmb{\psi}}\hat{\pmb{\varphi}}]\!]_{r_1} \otimes \mbox{\dots} \otimes [\![\hat{\pmb{\psi}}\hat{\pmb{\varphi}}]\!]_{r_{i-1}} \otimes \hat{h} \circ \hat{\pmb{q}}_{r_{i}} \otimes \mathbb{1}_V^{\otimes k-i})\\
&= \delta_k(\hat{h} \circ \hat{\pmb{p}}_{r_1} \circ (\hat{g}\hat{f})^{\otimes r_1}  \otimes \mbox{\dots} \otimes \hat{h} \circ \hat{\pmb{p}}_{r_{i-1}} \circ (\hat{g}\hat{f})^{\otimes r_{i-1}}  \otimes \sum_{r = 1}^{r_i}\underline{\hat{h} \circ \hat{\pmb{q}}_{r}} \otimes \mathbb{1}_V^{\otimes k-i}) \\
&+ \delta_k(\hat{h} \circ \hat{\pmb{p}}_{r_1} \circ (\hat{g}\hat{f})^{\otimes r_1}  \otimes \mbox{\dots} \otimes \sum_{r = 1}^{r_{i-1}}\underline{[\![\hat{\pmb{\psi}}\hat{\pmb{\varphi}}]\!]_{r_{i-i}}} \otimes \hat{h} \circ \hat{\pmb{q}}_{r_i} \otimes \mathbb{1}_V^{\otimes k-i}) + \mbox{\dots}\\
& + \delta_k(\hat{h} \circ \hat{\pmb{p}}_{r_1} \circ (\hat{g}\hat{f})^{\otimes r_1} \otimes \sum_{r = 1}^{r_2}\underline{[\![\hat{\pmb{\psi}}\hat{\pmb{\varphi}}]\!]_{r}}
 \otimes \mbox{\dots} \otimes [\![\hat{\pmb{\psi}}\hat{\pmb{\varphi}}]\!]_{r_{i-i}} \otimes \hat{h} \circ \hat{\pmb{q}}_{r_i} \otimes \mathbb{1}_V^{\otimes k-i}) \\
&+ \delta_k(\sum_{r = 1}^{r_1}\underline{[\![\hat{\pmb{\psi}}\hat{\pmb{\varphi}}]\!]_{r}} \otimes [\![\hat{\pmb{\psi}}\hat{\pmb{\varphi}}]\!]_{r_2}
 \otimes \mbox{\dots} \otimes [\![\hat{\pmb{\psi}}\hat{\pmb{\varphi}}]\!]_{r_{i-i}} \otimes \hat{h} \circ \hat{\pmb{q}}_{r_i} \otimes \mathbb{1}_V^{\otimes k-i}).
 \end{split}
\end{equation}
\textbf{II.} Let us choose $k, i, r_1, \mbox{\dots}, r_i$ as in $C(n)$ such that $i > 1$, $r_i = 1$, 
and there exists $1 \leq u \leq i - 1$ such that $r_u > 1$ and $r_{u + 1} = \mbox{\dots} = r_{i-1} = 1$. Then
\begin{equation}
\label{T29_3}
	\begin{split}
& \delta_k([\![\hat{\pmb{\psi}}\hat{\pmb{\varphi}}]\!]_{r_1}
 \otimes \mbox{\dots} \otimes [\![\hat{\pmb{\psi}}\hat{\pmb{\varphi}}]\!]_{r_{u}} \otimes(\hat{g}\hat{f})^{\otimes i - u + 1}\otimes \hat{h} \otimes \mathbb{1}_V^{\otimes k-i}) \\
& = \delta_k(\sum_{r = 1}^{r_1}\underline{[\![\hat{\pmb{\psi}}\hat{\pmb{\varphi}}]\!]_{r}} \otimes [\![\hat{\pmb{\psi}}\hat{\pmb{\varphi}}]\!]_{r_2}
 \otimes \mbox{\dots} \otimes [\![\hat{\pmb{\psi}}\hat{\pmb{\varphi}}]\!]_{r_{u}} \otimes(\hat{g}\hat{f})^{\otimes i - u + 1}\otimes \hat{h} \otimes \mathbb{1}_V^{\otimes k-i}) \\
& + \delta_k(\hat{h} \circ \hat{\pmb{p}}_{r_1} \circ (\hat{g}\hat{f})^{\otimes r_1} \otimes \sum_{r = 1}^{r_2}\underline{[\![\hat{\pmb{\psi}}\hat{\pmb{\varphi}}]\!]_{r}}
 \otimes \mbox{\dots} \otimes [\![\hat{\pmb{\psi}}\hat{\pmb{\varphi}}]\!]_{r_{u}} \otimes\\
&\quad \otimes(\hat{g}\hat{f})^{\otimes i - u + 1}\otimes \hat{h} \otimes \mathbb{1}_V^{\otimes k-i}) + \mbox{\dots}\\
& + \delta_k(\hat{h} \circ \hat{\pmb{p}}_{r_1} \circ (\hat{g}\hat{f})^{\otimes r_1}  \otimes \mbox{\dots} \otimes \hat{h} \circ \hat{\pmb{p}}_{r_{u-1}} \circ (\hat{g}\hat{f})^{\otimes r_{u-1}}  
\otimes \sum_{r = 1}^{r_u}\underline{[\![\hat{\pmb{\psi}}\hat{\pmb{\varphi}}]\!]_{r}} \otimes\\
& \quad\otimes (\hat{g}\hat{f})^{\otimes i - u + 1}\otimes \hat{h} \otimes \mathbb{1}_V^{\otimes k-i})\\  
& + \delta_k(\hat{h} \circ \hat{\pmb{p}}_{r_1} \circ (\hat{g}\hat{f})^{\otimes r_1}  \otimes \mbox{\dots} \otimes \hat{h} \circ \hat{\pmb{p}}_{r_{u-1}} \circ (\hat{g}\hat{f})^{\otimes r_{u-1}} \otimes \sum_{r = 1}^{r_u}\overline{\hat{h} \circ \hat{\pmb{q}}_{r}} \otimes\\
& \quad\otimes (\hat{g}\hat{f})^{\otimes i - u + 1}\otimes \hat{h} \otimes \mathbb{1}_V^{\otimes k-i})
	\end{split}
\end{equation}
with
$$\overline{\hat{h} \circ \hat{\pmb{q}}_{r}} = \hat{h} \circ \hat{\pmb{q}}_{r} \circ \left(\sum_{v = 0}^{r - 1}\mathbb{1}_V^{\otimes v} \otimes \delta_{r_{u} - r + 1} \otimes \mathbb{1}_V^{\otimes r-1-v}\right)\circ (\hat{g}\hat{f})^{\otimes r_u}.$$
By Lemma~\ref{k2t27}
$$\sum_{r = 1}^{r_u}\overline{\hat{h} \circ \hat{\pmb{q}}_{r}} = \hat{h} \circ \hat{\pmb{p}}_{r_u} \circ (\hat{g}\hat{f})^{\otimes r_u},$$
which can be justified in the same way as in the first step \textbf{I.}

We expand all terms in the summation (denoted \eqref{T29_1})
$$\sum_{i = 2}^{n-1} \hat{\pmb{q}}_{n-i+1}\circ\left(\sum_{u = 0}^{n-i}\mathbb{1}_V^{\otimes u} \otimes \delta_{i} \otimes \mathbb{1}_V^{\otimes n-i-u}\right)\circ \hat{H}|_{(sV)^{\otimes n}}$$
and use \eqref{T29_2} a \eqref{T29_3} to rewrite terms in the definition of $\hat{\pmb{q}}_n$:
$$\sum_{i = 2}^{n-1} \hat{\pmb{q}}_{n-i+1}\circ\left(\sum_{u = 0}^{n-i}\mathbb{1}_V^{\otimes u} \otimes \delta_{i} \otimes \mathbb{1}_V^{\otimes n-i-u}\right)\circ \hat{H}_{(sV)^{\otimes n}}$$
$$ =\sum_{C(n), k \neq n}\delta_k([\![\hat{\pmb{\psi}}\hat{\pmb{\varphi}}]\!]_{r_1} \otimes \mbox{\dots} \otimes [\![\hat{\pmb{\psi}}\hat{\pmb{\varphi}}]\!]_{r_{i-1}} \otimes \hat{h} \circ \hat{\pmb{q}}_{r_i} \otimes \mathbb{1}_V^{\otimes k-i}).$$
Certainly,
$$ \delta_n \circ \hat{H}_{(sV)^{\otimes n}} = \sum_{C(n), k = n}\delta_k([\![\hat{\pmb{\psi}}\hat{\pmb{\varphi}}]\!]_{r_1} \otimes \mbox{\dots} \otimes [\![\hat{\pmb{\psi}}\hat{\pmb{\varphi}}]\!]_{r_{i-1}} \otimes \hat{h} \circ \hat{\pmb{q}}_{r_i} \otimes \mathbb{1}_V^{\otimes k-i}),$$
which together with \eqref{q-kernel} completes the proof.
\end{proof}

\begin{remark}
Adopting slight changes in the proofs, our claims can be reformulated as follows:
\begin{description}
\item[{Lemma \ref{k2t27}}:] {\em On the assumption \eqref{Okraj} holds for 
all $n \geq 2$}
	\begin{equation}
		\label{T27*}
		\begin{split}
	\left(\sum_{B(n)} \hat{h} \circ \hat{\pmb{p}}_{r_1} \circ \hat{g}^{\otimes r_1} \otimes \mbox{\dots} \otimes \hat{h} \circ \hat{\pmb{p}}_{r_k}\right) \circ \hat{g}^{\otimes r_k} = \hat{g}^{\otimes n} \\
	+ \sum_{i = 2}^{n-1} \sum_{C(n-i+1)}[\![\hat{\pmb{\psi}}\hat{\pmb{\varphi}}]\!]_{r_1} \otimes \mbox{\dots} \otimes [\![\hat{\pmb{\psi}}\hat{\pmb{\varphi}}]\!]_{r_{i-1}} \otimes \hat{h} \circ \hat{\pmb{q}}_{r_i} \otimes \mathbb{1}_V^{\otimes k-i}\circ\\
	\circ\left(\sum_{u = 0}^{n-i}\mathbb{1}_V^{\otimes u} \otimes \delta_{i} \otimes \mathbb{1}_V^{\otimes n-i-u}\right)\circ \hat{g}^{\otimes n},
		\end{split}
	\end{equation}
	where we write $\hat{h} \circ \hat{\pmb{p}}_{r_1} \circ \hat{g}^{\otimes r_1} \otimes \mbox{\dots} \otimes \hat{h} \circ \hat{\pmb{p}}_{r_k}$
	instead of $\delta_k(\hat{h} \circ \hat{\pmb{p}}_{r_1} \circ \hat{g}^{\otimes r_1} \otimes \mbox{\dots} \otimes \hat{h} \circ \hat{\pmb{p}}_{r_k})$ 
and $[\![\hat{\pmb{\psi}}\hat{\pmb{\varphi}}]\!]_{r_1} \otimes \mbox{\dots} \otimes [\![\hat{\pmb{\psi}}\hat{\pmb{\varphi}}]\!]_{r_{i-1}} \otimes \hat{h} \circ \hat{\pmb{q}}_{r_i} \otimes \mathbb{1}_V^{\otimes k-i}$ instead of	
	$\delta_k([\![\hat{\pmb{\psi}}\hat{\pmb{\varphi}}]\!]_{r_1} \otimes \mbox{\dots} \otimes [\![\hat{\pmb{\psi}}\hat{\pmb{\varphi}}]\!]_{r_{i-1}} \otimes \hat{h} \circ \hat{\pmb{q}}_{r_i} \otimes \mathbb{1}_V^{\otimes k-i})$;
\item[{Lemma \ref{k2l28}}:] {\em On the assumption \eqref{Okraj} holds for all $n \geq 2$}
	\begin{equation}
		\label{L28*}
		\begin{split}
\hat{f} \circ \hat{\pmb{q}}_n + \sum_{B(n)}\hat{f} \circ \hat{\pmb{q}}_{r_1} \otimes \mbox{\dots} \otimes\hat{f} \circ \hat{\pmb{q}}_{r_k} - \hat{f}^{\otimes n} = \hat{f} \circ \delta_n \circ \hat{H}|_{(sV)^{\otimes n}} \\
	+ \sum_{i = 2}^{n-1} \left(\hat{f} \circ \hat{\pmb{q}}_{n-i+1} + \sum_{B(n-i+1)}\hat{f} \circ \hat{\pmb{q}}_{r_1} \otimes \mbox{\dots} \otimes\hat{f} \circ \hat{\pmb{q}}_{r_k}\right)\circ\\
	\circ \left(\sum_{u = 0}^{n-i}\mathbb{1}_V^{\otimes u} \otimes \delta_{i} \otimes \mathbb{1}_V^{\otimes n-i-u}\right)\circ \hat{H}|_{(sV)^{\otimes n}},
		\end{split}
	\end{equation}
	where we write 
	$\hat{f} \circ \hat{\pmb{q}}_{r_1} \otimes \mbox{\dots} \otimes\hat{f} \circ \hat{\pmb{q}}_{r_k}$	
	instead of	
		$\hat{h} \circ \hat{\pmb{p}}_k(\hat{g}\hat{f} \circ \hat{\pmb{q}}_{r_1} \otimes \mbox{\dots} \otimes \hat{g}\hat{f} \circ \hat{\pmb{q}}_{r_k})$;
\item[{Lemma \ref{k2t29}}:] {\em On the assumption \eqref{Okraj}, we have for all $n \geq 2$}
		\begin{equation}
		\label{T29*}
		\begin{split}
			\sum_{C(n)}[\![\hat{\pmb{\psi}}\hat{\pmb{\varphi}}]\!]_{r_1} \otimes \mbox{\dots} \otimes [\![\hat{\pmb{\psi}}\hat{\pmb{\varphi}}]\!]_{r_{i-1}} \otimes \hat{h} \circ \hat{\pmb{q}}_{r_i} \otimes \mathbb{1}_V^{\otimes k-i} = \hat{H}|_{(sV)^{\otimes n}}\\
		+ \sum_{i = 2}^{n-1} \sum_{C(n-i+1)}[\![\hat{\pmb{\psi}}\hat{\pmb{\varphi}}]\!]_{r_1} \otimes \mbox{\dots} \otimes [\![\hat{\pmb{\psi}}\hat{\pmb{\varphi}}]\!]_{r_{i-1}} \otimes \hat{h} \circ \hat{\pmb{q}}_{r_i} \otimes \mathbb{1}_V^{\otimes k-i}\circ
			\\
			\circ\left(\sum_{u = 0}^{n-i}\mathbb{1}_V^{\otimes u} \otimes \delta_{i} \otimes \mathbb{1}_V^{\otimes n-i-u}\right)\circ \hat{H}|_{(sV)^{\otimes n}},
		\end{split}
	\end{equation}
	where we write $[\![\hat{\pmb{\psi}}\hat{\pmb{\varphi}}]\!]_{r_1} \otimes \mbox{\dots} \otimes [\![\hat{\pmb{\psi}}\hat{\pmb{\varphi}}]\!]_{r_{i-1}} \otimes \hat{h} \circ \hat{\pmb{q}}_{r_i} \otimes \mathbb{1}_V^{\otimes k-i}$
instead of
 $\delta_k([\![\hat{\pmb{\psi}}\hat{\pmb{\varphi}}]\!]_{r_1} \otimes \mbox{\dots} \otimes [\![\hat{\pmb{\psi}}\hat{\pmb{\varphi}}]\!]_{r_{i-1}} \otimes \hat{h} \circ \hat{\pmb{q}}_{r_i} \otimes \mathbb{1}_V^{\otimes k-i})$.
\end{description}	
\end{remark}

\begin{theorem}
	On the assumption \eqref{Okraj}, the formulas produced by the homological 
	perturbation lemma \eqref{PLrozpis} fulfill
	\begin{equation}
		\tag{\textit{{1}}}
		\label{V30a}
	\delta_{\pmb{\nu}}|_{(sV)^{\otimes n}} = \hat{f} \circ \hat{\pmb{p}}_n \circ \hat{g}^{\otimes n} + \sum_{A(n)} \mathbb{1}_W^{\otimes i - 1} \otimes \hat{f} \circ \hat{\pmb{p}}_\ell \circ \hat{g}^{\otimes \ell} \otimes \mathbb{1}_W^{\otimes n-k},
	\end{equation}
	\begin{equation}
		\tag{\textit{{2}}}
		\label{V30b}
			\hat{\pmb{\psi}}|_{(sV)^{\otimes n}} = \hat{h} \circ \hat{\pmb{p}}_n \circ \hat{g}^{\otimes n} + \sum_{B(n)} \hat{h} \circ \hat{\pmb{p}}_{r_1} \circ \hat{g}^{\otimes r_1} \otimes \mbox{\dots} \otimes \hat{h} \circ \hat{\pmb{p}}_{r_k} \circ \hat{g}^{\otimes r_k},
	\end{equation}
	\begin{equation}
		\tag{\textit{{3}}}
		\label{V30c}
		\hat{\pmb{\varphi}}|_{(sV)^{\otimes n}} = \hat{f} \circ \hat{\pmb{q}}_n + \sum_{B(n)} \hat{f} \circ \hat{\pmb{q}}_{r_1} \otimes \mbox{\dots} \otimes \hat{f} \circ \hat{\pmb{q}}_{r_k},
	\end{equation}
	\begin{equation}
		\tag{\textit{{4}}}
		\label{V30d}
		\hat{\pmb{H}}|_{(sV)^{\otimes n}} = \hat{h} \circ \hat{\pmb{q}}_n + \sum_{C(n)}[\![\hat{\pmb{\psi}}\hat{\pmb{\varphi}}]\!]_{r_1} \otimes \mbox{\dots} \otimes [\![\hat{\pmb{\psi}}\hat{\pmb{\varphi}}]\!]_{r_{i-1}} \otimes \hat{h} \circ \hat{\pmb{q}}_{r_i} \otimes \mathbb{1}_V^{\otimes k-i}
	\end{equation}
	for all $n \geq 2$. In particular, $\delta_W + \delta_{\pmb{\nu}}$ is a codifferential, $\hat{\pmb{\psi}}$, $\hat{\pmb{\varphi}}$ are morphisms and $\hat{\pmb{H}}$ is a homotopy between 
	$\hat{\pmb{\psi}}\hat{\pmb{\varphi}}$ and $\mathbb{1}$. When expressed in terms of $A_\infty$ algebras,
	the relevant objects fulfill \eqref{Antanz}.
\end{theorem}
\begin{proof}
We already noticed 
$$ (\delta_{\pmb{\mu}}\hat{H})((sV)^{\otimes n}) \subseteq sV \oplus \mbox{\dots} \oplus (sV)^{\otimes n-1},$$
$$ (\delta_{\pmb{\mu}} \hat{H})^{n-1}((sV)^{\otimes n}) \subseteq sV, \;\;\; (\delta_{\pmb{\mu}} \hat{H})^{n}((sV)^{\otimes n}) = 0.$$
\eqref{V30c} \& \eqref{V30a}: We prove by the induction hypothesis \eqref{V30c}. For $n = 2$, 
we get by \eqref{PLrozpis} 
$$\hat{\pmb{\varphi}}|_{(sV)^{\otimes 2}} = \hat{F}|_{(sV)^{\otimes 2}} + \hat{F}\delta_{\pmb{\mu}} \hat{H}|_{(sV)^{\otimes 2}} = \hat{f} \otimes \hat{f} + \hat{f}\circ (\delta_2(\hat{g}\hat{f} \otimes \hat{h}) + \delta_2(\hat{h} \otimes \mathbb{1}_V))$$
$$= \hat{f} \circ \hat{\pmb{q}}_1 \otimes \hat{f} \circ \hat{\pmb{q}}_1 + \hat{f} \circ \hat{\pmb{q}}_2.$$
Let us assume \eqref{V30c} holds for all natural number greater than $1$ and less than $n$. 
Because $\delta_{\pmb{\mu}} \hat{H}$ decreases the homogeneity,
$$\hat{\pmb{\varphi}}|_{(sV)^{\otimes n}} = \hat{F}|_{(sV)^{\otimes n}} + \hat{F}\delta_{\pmb{\mu}} \hat{H}|_{(sV)^{\otimes n}} + \hat{F}\left(\sum_{m = 1}^{n-2}(\delta_{\pmb{\mu}} \hat{H})^{m}\right)\delta_{\pmb{\mu}} \hat{H}|_{(sV)^{\otimes n}}$$
$$= \hat{F}|_{(sV)^{\otimes n}} + \hat{F} \circ\left(\sum_{u = 0}^{n-i}\mathbb{1}_V^{\otimes u} \otimes \delta_{i} \otimes \mathbb{1}_V^{\otimes n-i-u}\right)\circ \hat{H}|_{(sV)^{\otimes n}}$$
$$+ \hat{F}\left(\sum_{m = 1}^{n-2}(\delta_{\pmb{\mu}} \hat{H})^{m}\right)\circ\left(\sum_{u = 0}^{n-i}\mathbb{1}_V^{\otimes u} \otimes \delta_{i} \otimes \mathbb{1}_V^{\otimes n-i-u}\right)\circ \hat{H}|_{(sV)^{\otimes n}}.$$
The mapping $\left(\sum_{u = 0}^{n-i}\mathbb{1}_V^{\otimes u} \otimes \delta_{i} \otimes \mathbb{1}_V^{\otimes n-i-u}\right)\circ \hat{H}|_{(sV)^{\otimes n}}$ is of homogeneity $n-i+1$, so 
\eqref{PLrozpis} allows us to rewrite the last result as
$$\hat{F}|_{(sV)^{\otimes n}} + \hat{f} \circ \delta_n \circ \hat{H}|_{(sV)^{\otimes n}}$$
$$+  \sum_{i = 2}^{n-1} \hat{\pmb{\varphi}}|_{(sV)^{\otimes n-i+1}}\circ\left(\sum_{u = 0}^{n-i}\mathbb{1}_V^{\otimes u} \otimes \delta_{i} \otimes \mathbb{1}_V^{\otimes n-i-u}\right)\circ \hat{H}|_{(sV)^{\otimes n}}$$
and the combination of induction hypothesis $\hat{\pmb{\varphi}}|_{(sV)^{\otimes n-i+1}}$ 
and Lemma~\ref{k2t29}, \eqref{L28*}, gives the required form
$$\hat{F}|_{(sV)^{\otimes n}} + \hat{f} \circ \hat{\pmb{q}}_n + \sum_{B(n)}\hat{f} \circ \hat{\pmb{q}}_{r_1} \otimes \mbox{\dots} \otimes\hat{f} \circ \hat{\pmb{q}}_{r_k} - \hat{f}^{\otimes n}.$$
Let us remark that \eqref{PLrozpis} gives for all $n \geq 2$ 
$$\delta_{\pmb{\nu}}|_{(sV)^{\otimes n}} = \hat{f} \circ \delta_n \circ \hat{g}^{\otimes n} $$$$+ \sum_{i = 2}^{n-1} \hat{\pmb{\varphi}}|_{(sV)^{\otimes n-i+1}}\circ\left(\sum_{u = 0}^{n-i}\mathbb{1}_V^{\otimes u} \otimes \delta_{i} \otimes \mathbb{1}_V^{\otimes n-i-u}\right)\circ \hat{g}^{\otimes n}.$$
Choosing $2 \leq j \leq n-1$, Lemma~\ref{k2l26} implies
$$\hat{\pmb{\varphi}}|_{(sV)^{\otimes n-i+1}}\circ\left(\sum_{u = 0}^{n-i}\mathbb{1}_V^{\otimes u} \otimes \delta_{i} \otimes \mathbb{1}_V^{\otimes n-i-u}\right)\circ \hat{g}^{\otimes n}$$
$$= \hat{f} \circ \hat{\pmb{q}}_{n-i+1} \circ\left(\sum_{u = 0}^{n-i}\mathbb{1}_V^{\otimes u} \otimes \delta_{i} \otimes \mathbb{1}_V^{\otimes n-i-u}\right)\circ \hat{g}^{\otimes n}$$
$$+\sum_{u = 1}^{n-i+1} (\hat{f}\hat{g})^{\otimes u - 1} \otimes \hat{f} \circ \delta_i \circ \hat{g}^{\otimes i} \otimes (\hat{f}\hat{g})^{\otimes n-i+1-u}$$
$$ + \sum_{A(n-i+1)} (\hat{f}\hat{g})^{\otimes i - 1} \otimes \hat{f} \circ \hat{\pmb{q}}_{\ell} \circ\left(\sum_{u = 0}^{\ell-1}\mathbb{1}_V^{\otimes u} \otimes \delta_{i} \otimes \mathbb{1}_V^{\otimes \ell-1 - u}\right) \circ \hat{g}^{\otimes \ell - 1 + i} \otimes (\hat{f}\hat{g})^{\otimes n-i+1-k}.$$
We take into account \eqref{Okraj},  $\hat{f}\hat{g} = \mathbb{1}_W$, 
and sum up over all terms of the form 
$\mathbb{1}_W^{\otimes \star} \otimes \star \otimes \mathbb{1}_W^{\otimes \star}$:
$$\delta_{\pmb{\nu}}|_{(sV)^{\otimes n}} = \hat{f} \circ \delta_n \circ \hat{g}^{\otimes n} + \sum_{i = 2}^{n-1} \hat{f} \circ \hat{\pmb{q}}_{n-i+1} \circ\left(\sum_{u = 0}^{n-i}\mathbb{1}_V^{\otimes u} \otimes \delta_{j} \otimes \mathbb{1}_V^{\otimes n-i-u}\right)\circ \hat{g}^{\otimes n} $$
$$ + \sum_{A(n)} (\hat{f}\hat{g})^{\otimes i - 1} \otimes \Bigg[ \hat{f} \circ \delta_{\ell} \circ \hat{g}^{\otimes \ell}$$
$$  + \sum_{i = 2}^{\ell - 1} \hat{f} \circ \hat{\pmb{q}}_{\ell - i + 1} \circ\left(\sum_{u = 0}^{\ell-i}\mathbb{1}_V^{\otimes u} \otimes \delta_{i} \otimes \mathbb{1}_V^{\otimes \ell-i-u}\right) \circ \hat{g}^{\otimes \ell}\Bigg] \otimes (\hat{f}\hat{g})^{\otimes n-k}.$$
The application of Lemma~\ref{k2t27} concludes the proof.

\noindent \eqref{V30d} \& \eqref{V30b}: Similarly to the previous part of the proof, we first 
concentrate on \eqref{V30d} and then derive \eqref{V30b}. For $n = 2$,  
it follows from \eqref{PLrozpis}
$$\hat{\pmb{H}}|_{(sV)^{\otimes 2}} = \hat{H}|_{(sV)^{\otimes 2}} + \hat{H}\delta_{\pmb{\mu}} \hat{H}|_{(sV)^{\otimes 2}} = \hat{g}\hat{f} \otimes \hat{h} + \hat{h} \otimes \mathbb{1}_V + \hat{h}\circ (\delta_2(\hat{g}\hat{f} \otimes \hat{h}) + \delta_2(\hat{h} \otimes \mathbb{1}_V))$$
$$= [\![\hat{\pmb{\psi}}\hat{\pmb{\varphi}}]\!] \otimes \hat{h} \circ \hat{\pmb{q}}_1 + \hat{h} \circ \hat{\pmb{q}}_1 \otimes \mathbb{1}_V + \hat{h} \circ \hat{\pmb{q}}_2.$$
By the induction hypothesis, we assume \eqref{V30d} holds for all natural numbers greater than 
$1$ and less than $n$. We can write
$$\hat{\pmb{H}}|_{(sV)^{\otimes n}} = \hat{H}|_{(sV)^{\otimes n}} + \hat{h} \circ \delta_n \circ \hat{H}|_{(sV)^{\otimes n}}$$
$$+ \sum_{i = 2}^{n-1} \hat{\pmb{H}}|_{(sV)^{\otimes n-i+1}}\circ\left(\sum_{u = 0}^{n-i}\mathbb{1}_V^{\otimes u} \otimes \delta_{i} \otimes \mathbb{1}_V^{\otimes n-i-u}\right)\circ \hat{H}|_{(sV)^{\otimes n}}$$
Thanks to the induction hypothesis we can expand $\hat{\pmb{H}}|_{(sV)^{\otimes n-i+1}}$, and 
apply Lemma~\ref{k2t29} together with \eqref{T29*}:
$$\hat{h} \circ \hat{\pmb{q}}_n + \sum_{C(n)}[\![\hat{\pmb{\psi}}\hat{\pmb{\varphi}}]\!]_{r_1} \otimes \mbox{\dots} \otimes [\![\hat{\pmb{\psi}}\hat{\pmb{\varphi}}]\!]_{r_{i-1}} \otimes \hat{h} \circ \hat{\pmb{q}}_{r_i} \otimes \mathbb{1}_V^{\otimes k-i},$$
which completes the proof of the first assertion.
Now we use again \eqref{PLrozpis} for $n \geq 2$:
$$\hat{\pmb{\psi}}|_{(sV)^{\otimes n}} = \hat{G}|_{(sV)^{\otimes n}} + \hat{h} \circ \delta_n \circ \hat{G}_{(sV)^{\otimes n}}$$
$$+ \sum_{i = 2}^{n-1} \hat{\pmb{H}}|_{(sV)^{\otimes n-i+1}}\circ\left(\sum_{u = 0}^{n-i}\mathbb{1}_V^{\otimes u} \otimes \delta_{i} \otimes \mathbb{1}_V^{\otimes n-i-u}\right)\circ \hat{G}|_{(sV)^{\otimes n}}.$$
A consequence of Lemma~\ref{k2l26}, $2 \leq j \leq n-1$ arbitrary, is
$$\hat{\pmb{H}}|_{(sV)^{\otimes n-j+1}}\circ\left(\sum_{u = 0}^{n-j}\mathbb{1}_V^{\otimes u} \otimes \delta_{j} \otimes \mathbb{1}_V^{\otimes n-j-u}\right)\circ \hat{G}|_{(sV)^{\otimes n}}$$
$$= \hat{h} \circ \hat{\pmb{q}}_{n-j+1} \circ\left(\sum_{u = 0}^{n-j}\mathbb{1}_V^{\otimes u} \otimes \delta_{j} \otimes \mathbb{1}_V^{\otimes n-j-u}\right)\circ \hat{G}|_{(sV)^{\otimes n}}$$
$$+ \sum_{C(n-j+1)}\hat{h} \circ \hat{\pmb{p}}_{r_1} \circ \hat{g}^{\otimes r_1} \otimes \mbox{\dots} \otimes \hat{h} \circ \hat{\pmb{p}}_{r_{i-1}} \circ \hat{g}^{\otimes r_{i-1}} \otimes$$
$$\otimes \hat{h} \circ \hat{\pmb{q}}_{r_i} \circ\left(\sum_{u = 0}^{r_i - 1}\mathbb{1}_V^{\otimes u} \otimes \delta_{j} \otimes \mathbb{1}_V^{\otimes r_i - 1-u}\right)\circ \hat{G} \otimes (\hat{h} \circ \hat{\pmb{p}}_1 \circ \hat{g})^{\otimes k-i},$$
because $\hat{h} \circ \hat{\pmb{q}}_m \circ g^{\otimes m} = 0$ for all $m \geq 1$. 
In other words, if $\delta_\star$ in the last summation would not fit into 
$\hat{h} \circ \hat{\pmb{q}}_\star$ the corresponding term will be trivial. 
The summation then leads to
$$\hat{\pmb{\psi}}|_{(sV)^{\otimes n}} = \hat{G}|_{(sV)^{\otimes n}} + \hat{h} \circ \delta_n \circ \hat{G}_{(sV)^{\otimes n}} $$
$$+ \sum_{i = 2}^{n-1} \hat{h} \circ \hat{\pmb{q}}|_{(sV)^{\otimes n-i+1}}\circ\left(\sum_{u = 0}^{n-i}\mathbb{1}_V^{\otimes u} \otimes \delta_{i} \otimes \mathbb{1}_V^{\otimes n-i-u}\right)\circ \hat{G}|_{(sV)^{\otimes n}} $$
$$+ \sum_{C(n)}\hat{h} \circ \hat{\pmb{p}}_{r_1} \circ \hat{g}^{\otimes r_1} \otimes \mbox{\dots} \otimes \hat{h} \circ \hat{\pmb{p}}_{r_{i-1}} \circ \hat{g}^{\otimes r_{i-1}} \otimes$$
$$\otimes \left[ \hat{h} \circ \delta_{\ell} \circ \hat{G} + \sum_{i = 2}^{\ell - 1} \hat{h} \circ \hat{\pmb{q}}_{\ell - i + 1} \circ\left(\sum_{u = 0}^{i-1}\mathbb{1}_V^{\otimes u} \otimes \delta_{i} \otimes \mathbb{1}_V^{\otimes i-u}\right) \circ \hat{G}\right] \otimes (\hat{h} \circ \hat{\pmb{p}}_1 \circ \hat{g})^{\otimes k-i}.$$
In order to finish the proof, we remind the equality 
$\hat{h} \circ \hat{\pmb{p}}_1 = \mathbb{1}_V$ and use
Lemma~\ref{k2t27}.
\end{proof}


\subsection*{Acknowledgments} 
The present article is based on the thesis by the author. The author wants to thank Petr Somberg very much for his great help in preparing this text for publication $-$ if it hadn't been for it, this article would have never come to being.

\addvspace{\medskipamount}
 J.\, Kop\v{r}iva, Charles University, Faculty of Mathematics and Physics,
	Sokolovsk\'{a} 83, 180\,00 Praha, Czech Republic (\texttt{jakub.kopriva@outlook.com}). 

\begin{thebibliography}{99}

\bibitem{Crainic04}
Crainic, M., On the perturbation lemma, and deformations, 2004, ArXiv
preprint math.AT/0403266.

\bibitem{Hatcher02}
Hatcher, A., Algebraic topology, Cambridge University Press,
Cambridge, New York, 2002, ISBN 0-521-79540-0.

\bibitem{Huebschmann10}
Huebschmann, J., On the construction of $A_\infty$-structures, Georgian
Mathematical Journal, 17(1), 2010, 161--202.

\bibitem{Keller01}
Keller, B., Introduction to $A_\infty$ algebras and modules, 
Homology Homotopy Appl., 3(1), 2001, 1--35.

\bibitem{koso00}
Kontsevich, M., and Soibelman, Y., Homological mirror symmetry and torus fibrations, Symplectic
geometry and mirror symmetry (Seoul, 2000), World Sci. Publ., River Edge, NJ, 2001,
203--263. MR 1882331 (2003c:32025).

\bibitem{Lefevre03}
Lef\`evre-Hasegawa, K., Sur les $A_\infty$ cat\'egories, Th\`ese de doctorat,
Universit\'e Paris 7 - Denis Diderot, 2003.


\bibitem{mss02}
Markl, M., Shnider, S., and Stasheff, J. D., Operads in Algebra, Topology and Physics,
Volume {96} of Mathematical Surveys and Monographs, American Mathematical Society,
Providence, Rhode Island, 2002.

\bibitem{Markl06}
Markl, M., Transferring $A_\infty$ (strongly homotopy associative) structures,
Rend. Circ. Mat. Palermo (2) Suppl., No. 79, 2006, 139--151.

\bibitem{Mer98}
Merkulov, S., Strongly Homotopy Algebras of a K\"ahler Manifold, 
Internat. Math. Res. Notices, 1999, no. {3}, 153--164, math.AG/9809172.


\bibitem{Proute86}
Prout\'{e}, A., $A_\infty$-structures: Mod\`{e}les Minimaux de Baues-Lemaire et
Kadeishvili et Homologie des Fibrations, Th\`{e}se d'\'{e}tat, Universit\'{e} 
Paris 7 - Denis Diderot, 1986.

\bibitem{Weibel95}
Weibel, C. A., An introduction to homological algebra, Cambridge University Press,
1995, ISBN-13: 978-0521559874.

\end{thebibliography}
  \end{document}